\documentclass[reqno, 12pt]{amsart}
\usepackage{array}
\usepackage{amsmath}
\usepackage{amsfonts}
\usepackage{amssymb}
\usepackage{mathabx}
\usepackage{enumerate}
\usepackage{amsthm}
\usepackage{amsmath, amscd}
\usepackage{caption}
\usepackage[usenames, dvipsnames]{color}
\usepackage{xy}
\xyoption{all}
\usepackage{hyperref}
\usepackage{tikz}
\usepackage{tikz-cd}
\usepackage{marginnote}
\usetikzlibrary{matrix,arrows,backgrounds}
\newtheorem{theorem}{Theorem}[section]
\newtheorem{lemma}[theorem]{Lemma}
\newtheorem{proposition}[theorem]{Proposition}
\newtheorem{corollary}[theorem]{Corollary}
\newtheorem{conjecture}[theorem]{Conjecture}

\theoremstyle{definition}
\newtheorem{definition}[theorem]{Definition}
\newtheorem{notation}[theorem]{Notation}
\newtheorem{remark}[theorem]{Remark}
\newtheorem{example}[theorem]{Example}
\newtheorem*{theorem*}{Theorem}
\newtheorem*{corollary*}{Corollary}
\newtheorem*{conjecture*}{Conjecture}
\numberwithin{equation}{section}
\newtheorem{question}[theorem]{Question}

\newcommand{\newterm}{\textsf}

\newcommand{\dbcoh}[1]{\operatorname{D}^{\operatorname{b}}(\operatorname{coh }#1)}
\newcommand{\dbmod}[1]{\operatorname{D}^{\operatorname{b}}(\operatorname{mod }#1)}

\newcommand{\Hom}{\operatorname{Hom}}

\newcommand{\Z}{\mathbb{Z}}
\newcommand{\R}{\mathbb{R}}

\newcommand{\Q}{\mathbb{Q}}
\newcommand{\A}{\mathbb{A}}
\newcommand{\opA}{\overset{\circ}{A}}
\def\O{\mathcal{O}}

\newcommand{\End}{\operatorname{End}}
\newcommand{\spec}{\operatorname{Spec}}
\newcommand{\Rhom}{\operatorname{RHom}}

\newcommand{\Ext}{\operatorname{Ext}}
\newcommand{\gldim}{\operatorname{gl}\dim}

\newcommand{\cone}{\operatorname{Cone}}

\newcommand{\conv}{\operatorname{Conv}}

\title[Conic modules, secondary fans and NCRs]{Conic modules, secondary fans and non-commutative resolutions}

\author{Aimeric Malter}

\email{aimericmalter@bimsa.cn}

\address{Beijing Institute of Mathematical Sciences and Applications, No. 544, Hefangkou Village, Huaibei Town, Huairou District, Beijing 101408}

\begin{document}
\bibliographystyle{alpha}
\begin{abstract}
Faber, Muller and Smith used complete sums of conic modules to construct non-commutative crepant resolutions (NCCR) of simplicial toric algebras. We link these conic modules to the Bondal-Thomsen collection of line bundles on smooth toric DM stacks. This viewpoint allows us to establish computational results relating to conic modules, reducing the complexity of the combinatorics involved significantly. We formulate necessary and sufficient conditions for an incomplete sum of conic modules to give an NC(C)R of a toric algebra. Furthermore, we prove that to check if a toric algebra $R$ admits an NCCR in the form of $\End_R(\mathbb{B})$ for an incomplete sum of conic modules, we may reduce to a case where the class group of the affine toric variety $\spec R$ does not have torsion and verify the statement there. Finally, we treat the case of almost simplicial Gorenstein cones, i.e. cones with $|\sigma(1)|=\dim \sigma+1$, classifying when such cones admit NCCRs via endomorphism algebras of conic modules.
\end{abstract}
\maketitle
\tableofcontents

\section{Introduction}

Non-commutative (crepant) resolutions, or NC(C)Rs, were introduced by Van den Bergh as an algebraic approach to encapsulate the properties a crepant resolution of singularities should fulfill. Given a singular algebraic variety, oftentimes there are several natural candidates for a resolution thereof, even when imposing additional conditions such as crepancy, and an interesting question to consider is what these resolutions must have in common. These common properties in a certain sense characterise the resolution, and Bondal-Orlov \cite{BO02} and Kawamata \cite{Kaw02} suggest that the derived category of coherent sheaves is such an invariant of crepant resolutions. In other words, if a normal algebraic variety $X$ has two distinct crepant resolutions $\pi_i:Y_i\rightarrow X$, one should expect an equivalence of categories $\dbcoh{Y_1}\cong \dbcoh{Y_2}$. Van den Bergh \cite{VdB3Dflops} verified this for a three-fold flop by using tilting theory and constructing a pair of Morita equivalent algebras $\Lambda_i$ whose module categories are equivalent to the derived categories $\dbcoh{Y_i}$. A non-commutative resolution to a noetherian domain $R$ is an algebra of the form $\Lambda=\End_R(M)$ for a finitely generated reflexive $R$-module such that the global dimension of $\Lambda$ is finite. If $R$ is Gorenstein, such a $\Lambda$ is called crepant if in addition it is a maximal Cohen-Macaulay $R$-module. NCCRs are somewhat elusive objects, as evidenced by the fact that the following conjecture remains an open one.
\begin{conjecture}(=Conjecture \ref{Conj:affineGor})
    An affine Gorenstein toric algebra always has an NCCR.
\end{conjecture}

Partial progress has been made using a wide variety of methods. Broomhead \cite{Broomhead} used dimer models to verify the conjecture for three-dimensional affine Gorenstein toric algebras, a result later reproven by \v{S}penko and Van den Bergh \cite{SVdBtoricII}. In \cite{Tomo25}, the author proves the existence of NCCRs for toric algebras associated to affine toric varieties with Picard rank 1, an alternative proof of which can be found in \cite{MS26}. Representation theory and tilting theory are among the more common tools to construct NCCRs, and a good summary of the progress made can be found in \cite{VdB23}. 

Of particular interest to the present paper is the work of Faber, Muller and Smith \cite{FMS19}, who used \newterm{conic modules} to show that simplicial toric algebras admit an NCCR (a result previously shown by Craw and Quintero V\'{e}lez \cite{CQV12}). Given a cone $\sigma$ with associated toric algebra $R$, one can construct for each $v\in M_\R$ a conic module $A_v$, which is a Cohen-Macaulay $R$-module. These modules come in isomorphism classes and it suffices to examine the fundamental region $[0,1)^n$ to fully classify the conic modules. A direct sum $\mathbb{A}$ of such modules which contains at least one representative of each isomorphism class has been shown to give an NCCR $\End_R(\mathbb{A})$ of $R$ if and only if $\sigma$ is a simplicial cone. However, the authors give an example of a non-simplicial cone which allows for such an NCCR $\End_R(\mathbb{B})$ by using a direct sum $\mathbb{B}$ of conic modules not containing each isomorphism class, but instead excluding some. This is consistent with the idea NCCRs are minimal among NCRs, and so an NCR might be expected to have some subalgebra that is crepant. Naturally, this begs the following question.
\begin{question}[= Question \ref{Qn:IncompleteConMod}]
\label{Qn:IntroInc}
When does an incomplete sum of conic modules $\mathbb{B}=\bigoplus_{A_v\in I} A_v$ give an NCCR $\End_R(\mathbb{B})$ of $R=k[\sigma^\vee\cap M]$?
\end{question}

To verify the global dimension of the endomorphism algebras and their crepancy as NCRs, the authors of \cite{FMS19} construct explicit minimal projective resolutions of the graded simples and thus compute their projective dimension. It should be noted that, up to isomorphism, the set of graded simples is in bijection with the isomorphism classes of conic modules and so we can write $S_v$ for the simple corresponding to $v$. The projective resolutions are obtained by constructing, for each conic module $A_v$, a complex $K_v^\bullet$ whose degree zero component is precisely $A_v$, and then applying $\Hom(\mathbb{A},-)$ to that complex. Extending to the right by $S_v$ yields the required resolutions for the simplicial case. The construction of the complexes $K_v^\bullet$ is based on the combinatorial structure of a certain cell-decomposition on $M_\R$.

The complex $K_v^\bullet$ is constructed by observing the combinatorial structure of the cell-decomposition on $M_\R$ that is obtained when considering the loci which give the same conic modules (i.e. $v,v'$ lie in the same cell if $A_v=A_{v'}$).

In the given non-simplicial example by Faber-Muller-Smith, one can splice the complexes $K_v^\bullet$ to obtain new complexes $K_v^{\dagger,\bullet}$, in bijection with the conic modules $A_v$ appearing in the incomplete direct sum $\mathbb{B}$. We formalise this splicing process, via the procedure of \newterm{substitution} of conic modules, and define the notion of \newterm{lockable} sets of conic modules. These are precisely the sets such that a finite sequence of substitutions yields a set of complexes $K_v^{\dagger,\bullet}$ with each conic module appearing in the complex also appearing in the direct sum $\mathbb{B}$. Such a set is called \newterm{incredulous} if furthermore all the complexes have the same length, so that the top degree is $\dim \sigma$. With this notation, we formalise the argument used in \cite{FMS19}:
\begin{theorem}[= Theorem \ref{Thm:NCRiffLock}]
    Let $\sigma$ be a cone with associated toric algebra $R$ and collection of conic modules $\{A_v\}_{v\in S}$. For a subset $I$ of $S$, consider the incomplete direct sum of conic modules $\mathbb{B}=\bigoplus_{A_v\in S}A_v$. Then the endomorphism algebra $\Lambda'=\End_R(\mathbb{B})$ is an NCR of $R$ if and only if $I$ is lockable. Furthermore, $\Lambda'$ is an NCCR of $R$ if and only if $I$ is incredulous.
\end{theorem}

Closely related to conic modules is the Bondal-Thomsen collection, a collection of line bundles defined on toric varieties and stacks. In \cite{BBB+}, the authors explicitly define and describe the Bondal-Thomsen collection associated to a smooth toric Deligne-Mumford stack and note that if the underlying toric variety is affine, the collection of line bundles corresponds to the collection of conic modules. 
The Bondal-Thomsen collection can be used to construct a tilting bundle on the so-called \newterm{Cox category}, which is a derived category that in a certain sense encodes the bounded derived categories of the different toric DM stacks $\mathcal{X}_i$ that share the same Cox ring. These toric DM stacks are related by birational maps inducing Fourier-Mukai transforms between the bounded derived categories, and the Cox category can be interpreted as a category glued together along these transforms, acting as transition functions.

In the present paper, we explicitly provide the relationship between the collection of conic modules of a cone $\sigma$ and the Bondal-Thomsen collection associated to a simplicial subdivision $\Sigma$ of $\sigma$ that does not introduce any additional rays. Leveraging this relationship, we introduce another point of view to compute the complexes $K_v^\bullet$ associated to conic modules. The Bondal-Thomsen collection associated to a toric DM stack $\mathcal{X}$ is, up to torsion, in bijection with lattice points inside a zonotope $Z_\mathcal{X}$. Consider the divisorial exact sequence for a toric variety, \[
M\xrightarrow{f} \bigoplus_{\rho\in \Sigma(1)}\Z\cdot D_\rho\rightarrow \operatorname{coker}(f)\rightarrow 0.
\]
Here, the divisors $D_\rho$ are the torus-invariant Weil divisors associated to the rays $\rho\in \Sigma(1)$ and the map $f$ is given by $\sum_{\rho\in \Sigma(1)}\langle m,u_\rho\rangle D_\rho$. Note that the map $f$ can be represented by a matrix with rows corresponding to the vectors $u_\rho$. Tensoring with $\R$, the cokernel map becomes a linear map we denote by $f_{\sigma}$. Representing it by a matrix whose rows are primitive $\Z$-vectors, we consider the column vectors $\beta_\rho$, one for each $\rho\in \Sigma(1)$. These span a generalised fan known as the secondary fan, or GKZ fan, and the zonotope $Z_{\mathcal{X}}$ is the convex set $\sum (-1,0]\beta_\rho$. Interior lattice points correspond to elements of the Bondal-Thomsen collection, up to torsion, and these line bundles correspond to conic modules. For each lattice point $P$ in the zonotope, there are $|\operatorname{Tors}(\mathcal{X})|$ conic modules mapping to $P$; two conic modules map to the same lattice point if and only if the divisors $-d(v)$  differ by torsion.
For $v\in M_\R$, we denote by $f_\sigma(-d(v))$ the lattice point corresponding to the conic module $A_v$. We use this correspondence and define the notion of \newterm{valid path} between such lattice points. 

\noindent A path $\beta_J, J\subset \sigma(1)$, from $P_1$ to $P_2$ is a set $\{\beta_\rho\mid \rho\in J\subset\sigma(1)\}$ such that $P_1+\sum_{\rho\in J}\beta_\rho=P_2$. It is \newterm{valid} if $\sum_{\rho\in J}\alpha_\rho\beta_\rho$ lies in the set $\sum_{\rho\in J^c}(-\alpha_\rho-1,-\alpha_\rho)\beta_\rho$. The length of a path $\beta_J$ is the dimension of $\operatorname{Span}(u_\rho\vert\rho\in J)$. Using the notion of valid paths, we note the following condition for a conic modules $A_w$ to appear in the complex $K_v^\bullet$.

\begin{proposition}[= Proposition \ref{Prop:CplxViaPaths2}]
    There is a valid path from a lattice point $P_1$ to $f_\sigma(-d(v))$, if and only if there is a $w\in M_\R$ such that the isomorphism class of $A_w$ appears in $K_v^\bullet$ and $P_1=f_\sigma(-d(w))$. For a given valid path $\beta_J$, this $w\in M_\R$ fulfills $-d(w)\sim -d(v)-\sum_{\rho\in J}D_\rho$ and the degree in which $A_w$ appears in $K_v^\bullet$ is equal to the length of the path $\beta_J$.
\end{proposition}

Formally introducing a symbol $A_P$ for each lattice point $P\in Z_{\mathcal{X}}$, this allows the construction of complexes for each $P\in Z_{\mathcal{X}}$ and, in analogy to the complexes $K_v^\bullet$, we obtain a notion of incredulous sets of lattice points. We prove the following.

\begin{theorem}[= Theorem \ref{Thm:Incr definitions agree}]
    Let $\sigma$ be a cone with a collection of conic modules $S$. Let $I\subseteq S$ be an incredulous set of conic modules. Then the set of lattice points $I'=\{f_\sigma(-d(v))\mid A_v\in I\}$ is an incredulous set of lattice points. In particular, the following are equivalent.
    \begin{enumerate}
        \item There exists an incredulous set of conic modules.
        \item There exists an incredulous set of lattice points in $Z_\mathcal{X}$.
        \item There exists an NCCR of $\sigma$ of the form $\End(\mathbb{A})$ where $\mathbb{A}$ is a direct sum of conic modules.
    \end{enumerate}
\end{theorem}

This result significantly simplifies computations, as we can essentially reduce to the case where the class group of the toric variety has no torsion. So the number of complexes to compute reduces by a factor of $|\operatorname{Tors}(\mathcal{X})|$. 

Using incredulous sets of lattice points, we are able to answer Question \ref{Qn:IntroInc} for the case of \newterm{almost simplicial} Gorenstein cones, i.e. Gorenstein cones $\sigma$ with $|\sigma(1)|=\dim\sigma+1$.

\begin{theorem}[= Theorem \ref{Thm:AlmSimplConicNCCR}]
    Let $\sigma$ be an almost simplicial  Gorenstein cone. Then there exists an incomplete sum $\mathbb{B}$ of conic modules such that $\End_{R}(\mathbb{B})$  is an NCCR of $R=k[\sigma^\vee\cap M]$ if and only if any of the following holds:
    \begin{itemize}
        \item The collection of $\beta_\rho$ associated to $\sigma$ is, up to flipping all signs, $\{2,1,-1,-1,-1\}$.
        \item The collection of $\beta_\rho$ associated to $\sigma$ is $\{1,1,1,-1,-1,-1\}$.
        \item $\sigma$ is a 3-dimensional Gorenstein cone lattice equivalent to $\sigma=\cone(P\times\{1\})$, where $P$ is a trapezoid. 
    \end{itemize}
\end{theorem}

\subsection{Structure and notation}

We begin this paper by reviewing the works of Faber-Muller-Smith \cite{FMS19} and Ballard et al. \cite{BBB+} in $\S$ \ref{sec:Background}, also giving a brief introduction to non-commutative crepant resolutions. In $\S$ \ref{sec:Link}, we formalise the link between conic modules and the Bondal-Thomsen collection, introducing the notion of substitution of conic modules and proving that an incomplete sum of conic modules gives an NCCR if and only if the corresponding set of conic modules is an incredulous set. The next part of the paper, $\S$ \ref{sec:Compute}, provides a number of computational results. It is here we introduce the notion of viable paths, explaining how the combinatorics of the secondary fan simplify the search for incredulous sets (and thus NCCRs), eventually proving that we can reduce to the search for incredulous sets of lattice points. Finally, in $\S$ \ref{sec:Gor}, we illustrate the advantage of the methods introduced before by classifying which almost simplicial Gorenstein cone admit NCCRs via incomplete sums of conic modules.

Throughout the paper, unless stated otherwise, we adopt a few pieces of notation. For a given toric variety, the character lattice is usually denoted by $M$ and the cocharacter lattice by $N$. Their pairing $M\times N\rightarrow \Z$ extends naturally to a pairing of $M_\R:=M\otimes_\Z\R$ and $N_\R:=N\otimes_\Z\R$. Given a fan $\Sigma\subset N_\R$, for each ray $\rho\in\Sigma(1)$, we denote its primitive generator in $N$ as $u_\rho$ and we denote the torus invariant Weil divisor associated to $\rho$ by $D_\rho$.

Given a normal noetherian domain $R$, a finitely generated $R$-module is called \newterm{reflexive} if the canonical map $M\mapsto \Hom_R(\Hom_R(M,R),R)$ is an isomorphism, and these modules form a category denoted by $\operatorname{ref}R$. Given a reflexive $R$-algebra $\Lambda$, we also consider the category of reflexive $\Lambda$-modules $\operatorname{ref}\Lambda$. Assuming further that $R$ is commutative, an $R$-module $M$ is said to be \newterm{maximal Cohen-Macaulay} if $M_m$ is maximal Cohen-Macaulay for every maximal ideal $m$. These modules form a category denoted by $\operatorname{CM}R$ and the word maximal will be omitted throughout this paper. Note that a module $M$ is in $\operatorname{CM}R$ if and only if $\Ext^i(M,R)=0$ for all $i>0$. Finally, we fix an algebraically closed field $k$ of characteristic 0. 

\subsection{Acknowledgments}
The author is supported by the Beijing Natural Science Foundation IS25013 and the Beijing Postdoctoral Research Foundation.
Furthermore, the author would like to thank Dr Will Donovan and Dr Artan Sheshmani for many fruitful discussions leading to the completion of this paper. 

\section{Background}\label{sec:Background}

In this section, we will revisit some notions appearing in the study of desingularisations of affine toric varieties. We shall first introduce the notion of \newterm{non-commutative (crepant) resolution (NCCR)}, due to Van den Bergh \cite{VdB04}. Then, we discuss the use of conic modules, which were first studied systematically in \cite{BG03}, to generate NCCRs. The here present introduction to conic modules is based on the paper by Faber-Muller-Smith \cite{FMS19}. Finally, we review the Bondal-Thomsen collection and the geometry of the secondary fan, with a view towards NCCRs, following closely the exposition by Ballard et al. \cite{BBB+}.
\subsection{NCCRs}

In a more classical algebraic geometry setting, a resolution of singularities for a singular variety $X$ is a proper birational morphism $\phi:Y\rightarrow X$ from a smooth variety $Y$ to $X$ such that $\phi$ is an isomorphism away from the singular locus of $X$. It is crepant if the canonical divisor pulls back without acquiring additional "discrepancies", i.e. $K_Y=\phi^\ast K_X$. Such a resolution is not unique, but we do expect (see \cite{BO02,Kaw02}) that two distinct crepant resolutions $Y_1, Y_2$ are equivalent on the level of derived categories, i.e. $\dbcoh{ Y_1}\cong \dbcoh{Y_2}$. This motivates the idea that the derived category itself in some sense is a crepant resolution. abstracting one step further, if a crepant resolution $Y$ admits a tilting object $T$ (see Definition \ref{Def:tilt} below) of its derived category, we have $\dbcoh{Y}\cong \dbmod{\End(T)}$, and so van den Bergh suggests that the ring $\Lambda=\End(T)$, a purely algebraic object, can also be considered a crepant resolution. To remain on a completely algebraic level, thus ignoring whether or not there exist underlying geometric objects, we consider $\Lambda$ to be a resolution of the coordinate ring $R$ of the variety. The following is always formulated in terms of affine schemes, i.e. $X=\spec R$, but this can be extended to non-affine cases by reducing to affine patches.

\begin{definition}
    \label{Def:NCCR}
    A \newterm{non-commutative resolution} of $R$ is a trivial reflexive Azumaya algebra $\Lambda=\End_R(M)$ for $M\in \operatorname{ref}R$ such that $\gldim\Lambda<\infty$. Assuming further that $R$ is Gorenstein, we call $\Lambda$ \newterm{crepant} if it is additionally a Cohen-Macaulay $R$-module.
\end{definition}

\noindent Non-commutative resolutions are intrinsically linked to \newterm{tilting objects}. 

\begin{definition}
    \label{Def:tilt}
     Let $Y$ be a Noetherian scheme. A perfect complex on $Y$ is \newterm{partial tilting} if $\Ext^i_Y(\mathcal{T},\mathcal{T})=0$ for $i\neq 0$. It is \newterm{tilting} if it generates $D_{Qch}(Y)$, i.e. $\Rhom_Y(\mathcal{T},\mathcal{F})=0$ implies $\mathcal{F}=0$. Analogously we define (partial) tilting complexes for smooth, separated Noetherian DM stacks.
\end{definition}

Tilting objects often generate NCCRs; if $X=\spec R$ has a crepant resolution $Y$ with tilting object $\mathcal{T}$, it is often straightforward to show that $\Lambda:=\End_Y(\mathcal{T})$ defines an NCCR of $R$. Given an endomorphism algebra $\Lambda=\End_R(M)$ of a reflexive module over a noetherian ring $R$, one obtains an NCCR if the global dimension is finite and if $\Lambda$ is Cohen-Macaulay. If $R$ itself is a normal Cohen-Macaulay domain of dimension $d$, then this corresponds to every simple $\Lambda$-module having projective dimension $d$.

Since the inception of NCCRs, significant amounts of effort have gone into proving the existence thereof for different cases. A first class of varieties one is naturally drawn to consider is that of affine Gorenstein toric varieties. 
For the convenience of the reader, let us here recall what affine Gorenstein toric varieties are. A cone $\sigma\subset N_\R$ is defined via the generators $u_{\rho_1},\dots,u_{\rho_s}$ of its one-dimensional facets, known as rays: \[
\sigma=\cone(u_{\rho_i})=\{\sum\lambda_i u_{\rho_i}\vert\lambda_i\geq 0\}\subseteq N_\R.
\]
To a cone $\sigma\subset N_\R$, which we assume to be full-dimensional, we associate a dual cone \[
\sigma^\vee:=\{m\in M_\R\vert \langle m,n\rangle \ge 0\quad \forall n\in \sigma\}.
\]
Intersecting with $M$ determines a semigroup $\sigma^\vee\cap M$ and so we can form the \newterm{toric algebra} $R_\sigma:=k[\sigma^\vee\cap M]$. The affine toric variety associated to the cone $\sigma$ is defined to be $X_\sigma=\spec R_\sigma$. We call such a variety \newterm{Gorenstein} if there exists an element $m\in M$ such that $\langle m,u_{\rho_i}\rangle=1$ for all primitive generators $u_{\rho_i}$ of the rays $\rho_i$ of $\sigma$. 

Even in the case of affine Gorenstein toric varieties, it is unclear if NCCRs always exist, and this is the content of the following conjecture of Van den Bergh's.
\begin{conjecture}
    \label{Conj:affineGor}
    An affine Gorenstein toric variety always has an NCCR.
\end{conjecture}

Whilst the general case is still unsolved, progress has been made towards proving this conjecture. Broomhead \cite{Broomhead} showed, using the theory of dimer models, that all 3-dimensional affine Gorenstein toric varieties have an NCCR. A later reproof by \v{S}penko and Van den Bergh \cite{SVdBtoricII} avoids the technical tool of dimer models by constructing tilting bundles on simplicial refinements of the Gorenstein cone underlying the toric varieties.
\begin{proposition}[Proposition 3.3 in \cite{SVdBtoricII}]
    \label{Prop:3.3SVDBtoricII}
    Given a Gorenstein cone $\sigma\subset N_\R$, write it in the form $\cone(P\times\{1\})$ for $P$ a lattice polytope. Choose a regular triangulation of $P$ without extra vertices and let $\Sigma$ be the corresponding fan. Let $\mathcal{T}$ be a tilting bundle on $\mathcal{X}_{\Sigma}$, the associated DM stack. Then $\Lambda=\End_{\mathcal{X}_{\Sigma}}(\mathcal{T})$ is an NCCR for $R=k[\sigma^\vee\cap M]$ corresponding to $M'=\Gamma(\mathcal{X}_{\Sigma},\mathcal{T})$.
\end{proposition}

Previous work by the author and Sheshmani \cite{MS25} extended this result.
\begin{theorem}[= Theorem 3.12 in \cite{MS25}]
    \label{Thm:ParttiltingGorCone}
    Given a Gorenstein cone $\sigma\subset N_{\R}$, write it in the form $\cone(P\times\{1\})$ for $P$ a lattice polytope.
    Choose a regular triangulation of $P$ and let $\Sigma$ be the corresponding fan refining $\sigma$. Let $\mathcal{T}$ be a partial tilting complex on $\mathcal{X}_{\Sigma}$, the associated toric DM stack. Assume that $\Lambda=\End_{\mathcal{X}_{\Sigma}}(\mathcal{T})$ has finite global dimension. Then it is an NCCR for $R=k[\sigma^\vee\cap M]$.
\end{theorem}

Note here that Donovan-Hara-Kapustka-Rampazzo \cite{DHKR25} proved that such a partial tilting object with finite global dimension in fact is tilting. Using Theorem \ref{Thm:ParttiltingGorCone}, in \cite{MS26} we show that toric algebras associated to affine toric varieties of Picard ranks 0 and 1 admit toric NCCRs. The latter has also been shown via a different methodology by Tomonaga \cite{Tomo25} whilst the former is known due to work by Craw and Quintero V\'{e}lez \cite{CQV12}. Another proof of the simplicial case is due to Faber, Muller and Smith \cite{FMS19} and relies on conic modules, which we will introduce next.

\subsection{Conic modules}\label{sec:conic}

 The following introduction to conic modules closely follows the work of Faber, Muller and Smith \cite{FMS19}. As setup, we consider a cone $\sigma\subset N_\R$ and its dual cone $\sigma^\vee\subset M_\R$ together with the associated noetherian domain $R=k[\sigma^\vee\cap M]$. We may assume $\sigma$ is strongly convex, which equivalently implies that $\sigma^\vee$ is of full dimension.  Note that \[
\sigma^\vee=\bigcap_{\rho\in\sigma(1)} \{v\in M_\R\vert \langle v,u_\rho\rangle\geq0\}.
\]

\begin{definition}
    \label{Def:ConicModule}
    For $v\in M_{\R}$, the \newterm{conic module defined by} $v$ is the $M$-graded $R$-submodule of $k[M]$ defined as
    \[
    A_v:=\operatorname{Span}\{x^m\vert m\in M\cap (\sigma^\vee+v)\}.
    \]
\end{definition}
Faber Muller and Smith show that conic modules are always Cohen-Macaulay \cite[Corollary 4.16]{FMS19}, and also demonstrate other key properties.
\begin{proposition}[ Proposition 3.2 in \cite{FMS19}]
    \label{Prop:ConicModulesProperties}
    Let $A_v$ be a conic module for $R=k[\sigma^\vee\cap M]$. Then
    \begin{enumerate}
        \item $A_v$ is torsion free and rank one over $R$.
        \item $A_v$ is spanned by monomial $x^m$, where \[
        m\in \bigcap_{\rho\in\sigma(1)}\{x\in M_\R\vert \langle x,u_\rho\rangle \geq \lceil\langle v,u_\rho\rangle\rceil\}\subset M.
        \]
        \item For any two conic modules $A_v, A_w$, the $R$-module $\Hom_R(A_v,A_w)$ is naturally isomorphic to the $M$-graded $R$-submodule of $k[M]$\[
        \operatorname{Span}\{x^m\vert m+(\sigma^\vee+v)\cap M\subset (\sigma^\vee+w)\}.
        \]
    \end{enumerate}
\end{proposition}
\noindent In addition to conic modules as defined above, we shall also consider \newterm{open conic modules}.
\begin{definition}
    \label{Def:OpenConicModule}
Let $v\in M_\R$. The \newterm{open conic module} associated to $v$ is the $M$-graded $R$-submodule of $k[M]$ defined as follows:
\[
\opA_v:=\operatorname{Span}\{x^m\vert m\in M\cap (\overset{\circ}{\sigma^\vee}+v)\}.
\]
\end{definition}

\noindent As it turns out, open conic modules are conic modules themselves -  associated to nearby points.
\begin{proposition}[Proposition 3.5 in \cite{FMS19}]
    \label{Prop:OpenconicIsNearby}
    For every open conic $R$ module $\opA_v$, there exists $w\in M_\R$ such that \[
    \opA_v=A_w.
    \]
    Explicitly, for any $p\in \overset{\circ}{\sigma^\vee}$ and $0<\varepsilon\ll1$, we can take $w=v+\varepsilon p$.
\end{proposition}

At this stage, it is a natural question to ask when two different $v,w\in M_\R$ define equal or at least isomorphic conic modules.
\begin{definition}
    \label{Def:ChamOfCons}
    The \newterm{chamber of constancy} $\Delta\subset M_\R$ containing $v\in M_\R$ is the set of all $w\in M_\R$ such that $A_v=A_w$. Equivalently, $v,w$ belong to the same chamber of constancy iff $(\sigma^\vee+v)\cap M=(\sigma^\vee+w)\cap M$.
\end{definition} 
Given $v\in M_\R$, we denote its chamber of constancy by $\Delta_v$, but if we are given just a chamber of constancy without explicit choice of element inside, we will write $A_\Delta$ for the conic module that the elements $v\in \Delta$ all share. These chambers of constancy decompose $M_\R$ into disjoint locally closed polyhedral regions \cite[Corollary 4.4]{FMS19} and can be explicitly computed.
\begin{proposition}[Proposition 4.3 in \cite{FMS19}]
    \label{Prop:ComputeChamofCon}
    For $v\in M_\R$, its chamber of constancy is given by the set
    
    \begin{align*}
      \Delta_v&=\{x \in M_\R\vert \lceil \langle v, u_\rho\rangle\rceil-1<\langle x,u_\rho\rangle\le \lceil\langle v, u_\rho\rangle\rceil\text{ for }\rho\in\sigma(1)\}\\
      &=\bigcap_{\rho\in\sigma(1)}\{x\in M_\R\vert \lceil \langle v, u_\rho\rangle\rceil-1<\langle x,u_\rho\rangle\le \lceil\langle v,u_\rho\rangle\rceil\}.
    \end{align*}
\end{proposition}

\begin{remark}
\label{Rem:ConModBundle}
    The proof that Faber, Muller and Smith give for Proposition \ref{Prop:ComputeChamofCon} further implies that a given conic module $A_v$ is uniquely determined by the list of integers\[
    (\lceil \langle v,u_\rho\rangle \rceil)_{\rho\in\sigma(1)}\in \Z^t,
    \]
    where $t=|\sigma(1)|$. This shows that $A_v=\Gamma(X_\sigma,\O_{X_\sigma}(-D))$, where $D=\sum_{\rho\in\sigma(1)} \lceil\langle v,u_\rho\rangle\rceil D_\rho$. However, the simple observation that for a non-smooth cone $\sigma$ the $u_\rho$ are linearly dependent points to the fact that not all lists of integers appear as $t$-tuple $(\lceil \langle v,u_\rho\rangle \rceil)_{\rho\in\sigma(1)}$.
\end{remark}

Given a strongly convex, full-dimensional cone $\sigma\subset N_\R$, the chambers of constancy determine a CW composition of $M_\R$ (\cite{Bru05}), which can be elucidated
by studying (open) conic modules. Now \cite[Corollary 4.9]{FMS19} establishes that $v,w\in M_\R$ lie in the same open cell of the CW decomposition if and only if both $A_v=A_w, \opA_v=\opA_w$. If $v,w\in \Delta$ lie in the same chamber of constancy, then there is an inclusion $\opA_v\subset \opA_w$ if and only if the open cell containing $v$ is contained in the boundary of the open cell containing $w$. For this reason, we may later denote the open conic module $\opA_v$ associated to $v\in M_\R$ by $\opA_\tau$, where $\tau$ is the open cell containing $v$. This notation will mainly come into play when the chamber of constancy $\Delta$ is fixed, as it is possible to have $\opA_\tau=\opA_{\tau'}$ for $\tau\neq \tau'$. 

While $A_\Delta\neq A_{\Delta'}$ for two distinct chambers, they might still be isomorphic. The following Proposition tells us when this can happen.
\begin{proposition}[Proposition 4.11 in \cite{FMS19}]
    \label{Prop:IsoConicModules}
    Two conic $R$-modules $A_\Delta, A_{\Delta'}$ are isomorphic if and only if there is an $m\in M$ such that $\Delta=\Delta'+m$.
\end{proposition}
Thus, $M$ acts via translation on the set of chambers of constancy, with orbits of this action corresponding to distinct isomorphism classes of conic modules. As such, all the information on the chambers of constancy we may wish to know can be found by studying the (real) torus $M_\R/M$ and its induced CW decomposition. The maximal cells correspond bijectively to the different isomorphism classes of conic modules for $R$. Since the torus is compact, we deduce that there is only finitely many different isomorphism classes of conic $R$-modules for any given toric algebra $R$. Note that there is a partial ordering on the set of conic $R$-modules, given by inclusion. This induces a partial ordering on the chambers of constancy, and we write $\Delta \preceq \Delta'$ if and only if $A_\Delta\subseteq A_{\Delta'}$.

Fix now a conic module $A_\Delta$ over $R$, corresponding to a chamber of constancy $\Delta$. We shall introduce for it a complex of $R$-modules that can be used to construct projective resolutions for the graded simples of the proposed NC(C)Rs later. For each, open cell of $\Delta$, fix an orientation. Note that the CW decomposition restricted to $\Delta$ will not skip any dimension, i.e. if $\tau_i$ is a cell of $\Delta$ of codimension $i>0$ then it is contained in the boundary of some cell $\tau_{i-1}$ of $\Delta$ of codimension $i-1$. As such, $\tau_i$ inherits an orientation from $\tau_{i-1}$, which may or may not agree with the orientation for $\tau_i$ we fixed earlier. Write $\operatorname{sgn}(\tau_i,\tau_{i-1})=+1$ if the orientations agree and $-1$ if they do not. If there is no inclusion of cells, define the sign to be 0. The inclusion of cells induces an inclusion of open conic modules $\iota_{\tau_i,\tau_{i-1}}:\opA_{\tau_i}\hookrightarrow\opA_{\tau_{i-1}}$. For sake of notation, we consider the map $\iota_{\tau_i,\tau_{i-1}}:\opA_{\tau_i}\hookrightarrow\opA_{\tau_{i-1}}$ to be the zero map whenever $\opA_{\tau_{i}}\not\subset\opA_{\tau_{i-1}}$ for codimension $i, i-1$ cells $\tau_i, \tau_{i-1}$. \begin{definition}
    \label{Def:CplxChamber}
    Given a chamber of constancy $\Delta$, define a complex $K_\Delta^\bullet$ of $M$-graded $R$-modules\[
    \dots\rightarrow \bigoplus_{\operatorname{codim}(\tau)=2}\opA_{\tau}\rightarrow\bigoplus_{\operatorname{codim}(\tau)=1}\opA_\tau\rightarrow\bigoplus_{\operatorname{codim}(\tau)=0}\opA_\tau=A_\Delta.
    \]
    The direct sums range over all cells of $\Delta$ of the specified codimension, and the maps are given by signed sums of inclusion maps $\sum \operatorname{sgn}(\tau_i,\tau_{i-1})\iota_{\tau_i,\tau_{i-1}}$. We say that a complex has length $l$ if the highest non-zero degree is $l$. Where clear, we write $K_v^\bullet$ for the complex $K_{\Delta_v}^\bullet$.
\end{definition}

\begin{remark}
    The fact that $K_\Delta^\bullet$ is indeed an $M$-graded complex of $R$-modules can be found in \cite[Lemma 5.1]{FMS19}. The length of this complex is equal to the codimension of the smallest dimensional cell appearing in $\Delta$, which is at most $\dim R$ (the Krull dimension of $R$). The length can be smaller than that, as for non-simplicial cones $\sigma$, not all chambers of constancy $\Delta$ contain zero-dimensional cells.
\end{remark}

Of particular interest for these complexes is the following result, providing a notion of acyclicity.
\begin{lemma}
    \label{Lem:Acyc}
    Fix a toric algebra $R$ and consider two conic modules $A_\Delta, A_{\Delta'}$ for $R$, associated to two chambers of constancy $\Delta, \Delta'$.
    \begin{enumerate}
        \item The complex $\Hom_R(A_{\Delta'},K_\Delta^\bullet)$ is exact if $A_\Delta\not\simeq A_{\Delta'}$;
        \item If $A_\Delta\simeq A_{\Delta'}$, then $\Hom_R(A_{\Delta'},K_\Delta^\bullet)$ has one-dimensional homology, which appears (only) in homological degree zero and $M$-degree $m$, where $m$ is the unique element of $M$ such that $\Delta=\Delta'+m$.
    \end{enumerate}
\end{lemma}
 For $M$ a finite direct sum of conic modules, Faber, Muller and Smith decide both Cohen-Macaulayness and finite global dimension of $\Lambda=\End_R(M)$ by giving minimal projective resolutions for all simple $\Lambda$-modules. To examine whether a given sum of conic modules $M$ gives an NCCR $\End_R(M)$, the strategy is as follows. Firstly, one can show that all graded simple $\End_R(M)$-modules arise as quotient of $P_\Delta:=\Hom_R(M,A_\Delta)$, for some chamber of constancy $\Delta$, by its unique maximal graded submodule (see \cite[Proposition 6.3]{FMS19}). Denote this simple module by $S_\Delta$. Furthermore, all graded indecomposable projective modules are of the form $P_\Delta$ for some $\Delta$ appearing in $M$. Then, one shows that $\End_R(M,K_\Delta^\bullet)$ forms a graded $\End_R(M)$-projective resolution of the simple module $S_\Delta$ and that this resolution has minimal length \cite[Theorem 6.5]{FMS19}. 

A first instinct to follow then is to take a \textit{complete} sum of conic modules, i.e. $M=\bigoplus A_\Delta^{r_\Delta}$ where the sum ranges over all chambers of constancy and $r_\Delta\in \Z_{>0}$. However, Faber Muller and Smith show that while this always produces an NCR (since we have finite length resolutions of length at most $d=\dim R$ of all simples $S_\Delta$), crepancy happens if and only if the cone $\sigma$ is simplicial. They do, however, give an example of how this problem can potentially be remedied by restricting one's attention to a subset of the chambers of constancy.
\begin{example}[Example 7.10 in \cite{FMS19}]
\label{Exa:710FMS}
    Consider the cone $\sigma\subset N_\R\cong \R^3$ spanned by the four rays with primitive generators $(1,0,0)$, $(0,1,0)$, $(-1,0,1)$, $(0,-1,1)$. The dual cone $\sigma^\vee$ has primitive ray generators $(0,0,1)$, $(1,0,1)$, $(0,1,1)$ and $(1,1,1)$ and corresponds to the cone over a square. The toric algebra is $R=k[x,y,z,w]/(xz-yw)$. There are three isomorphism types of conic $R$-modules:
    $A_0=R$, $A_1=(x,y)R$ and $A_2=(x,w)R$. The associated complexes $K_i^\bullet$ are as follows:
    \begin{eqnarray*}
        && K_0^\bullet=A_0\rightarrow A_0^{\oplus 4}\rightarrow A_1^{\oplus 2}\oplus A_2^{\oplus 2}\rightarrow A_0,\\
        && K_1^\bullet=A_2\rightarrow A_0^{\oplus 2}\rightarrow A_1,\\
        && K_2^\bullet=A_1\rightarrow A_0^{\oplus 2}\rightarrow A_2.\;
    \end{eqnarray*}
    We first note that the complexes are not of the same length, so while we obtain an NCR of $R$ via $\End_R(A_0\oplus A_1\oplus A_2)$, this will not be crepant. However, the authors of \cite{FMS19} observe that using the complex for $A_2$, one can build new complexes \begin{eqnarray*}
        && K_0^{\dagger,\bullet}=A_0\rightarrow A_1^{\oplus 2}\oplus A_0^{\oplus 4}\rightarrow A_0^{\oplus 4}\oplus A_1^{\oplus 2}\rightarrow A_0,\\
        && K_1^{\dagger,\bullet}=A_1\rightarrow A_0^{\oplus 2}\rightarrow A_0^{\oplus 2}\rightarrow A_1.\;
    \end{eqnarray*}
    The functor $\Hom_R(A_0\oplus A_1,-)$ takes these complexes to minimal length projective resolutions of the only two graded simples $\End_R(A_0\oplus A_1)$-modules. The arguments leading to NCCRs of simplicial cones then follow through to show that $\End_R(A_0\oplus A_1)$ is an NCCR of $R$.
\end{example}
The above example leads to the following question that this present paper aims to investigate.
\begin{question}
\label{Qn:IncompleteConMod}
    When does an incomplete sum of conic modules $M=\bigoplus_{\Delta\in I}A_\Delta$ give an NCCR $\End_R(M)$ of $R$?
\end{question}

\subsection{The Bondal-Thomsen Collection}

As a means to investigate Question \ref{Qn:IncompleteConMod}, we will examine the combinatorics of conic modules and relate them to the geometry of the secondary fan of a simplicial refinement $\Sigma$ of $\sigma$. The idea for this is inspired by the work of Ballard et al. \cite{BBB+}, where the authors link the \newterm{Cox category} they introduce to the NCCRs constructed by Faber, Muller and Smith. The discussion happens on the level of toric DM stacks, whose derived categories are usually better behaved than the underlying toric varieties'. For instance, a simplicial, but not necessarily smooth, toric fan gives a smooth DM stack with coarse moduli space the appropriate toric variety. Given now a fixed (simplicial) toric variety $X_\Sigma$, we consider its secondary fan $\Sigma_{GKZ}$, whose maximal chambers correspond to simplicial toric varieties $X_1,\dots,X_r$.  
Intuitively, the Cox category is a triangulated category which encodes the information of all the categories $\dbcoh{\mathcal{X}_i}$, where $\mathcal{X}_i$ is the DM stack associated to $X_i$. These toric DM stacks $(\mathcal{X}_i)_{i=1}^r$ are related by birational maps, whose graphs induce Fourier-Mukai transforms $\Phi_{i,j}:\dbcoh{\mathcal{X}_i}\rightarrow \dbcoh{\mathcal{X}_j}$ between the derived categories. One viewpoint of the Cox category is to consider the categories $\dbcoh{\mathcal{X}_i}$ to be analogous to affine patches that are glued together along the Fourier-Mukai transforms, acting as transition functions, to give the Cox category. We refer the reader to Appendix A in \cite{BBB+} to make this perspective precise using the Grothendieck construction. One of the main results of the paper \cite{BBB+} is a proof that the Cox category admits a strong, full exceptional collection. This collection comes from the \newterm{Bondal-Thomsen collection}. We shall inspect this collection closer and illustrate its links to the conic modules discussed above. We begin by defining the toric DM stack associated to a toric variety $X_\Sigma$.

We fix, as usual, a pair of dual lattices $M,N\cong \Z^n$ and a fan $\Sigma\subset N_\R$ with a set of $k$ primitive ray generators $\nu=\{u_\rho\mid \rho\in\Sigma(1)\}\subset N$. Consider now the vector space $\R^k$ with an elementary $\Z$-basis enumerated by the set of rays of $\Sigma$, $\{e_\rho\}_{\rho\in \Sigma(1)}$. Define the \newterm{Cox fan} of $\Sigma$ to be \[
\operatorname{Cox}(\Sigma):=\{\cone(e_\rho\mid\rho\in\sigma)\mid\sigma\in\Sigma\}.
\]
This fan is a subfan of the standard fan for $\A^k$ and thus its associated toric variety can be viewed as open subspace of the affine space $\A^k$, which we denote by $U_\Sigma:=X_{\operatorname{Cox}(\Sigma)}$.

Next, we consider the right exact sequence
\begin{eqnarray}\label{eqn:Coxseqpre}
    && M\xrightarrow{f}\Z^k\xrightarrow{\pi}\operatorname{coker}f\rightarrow 0,\\
   &&  m\mapsto \sum_{\rho\in\Sigma(1)}\langle m,u_\rho\rangle e_\rho\nonumber.\;
\end{eqnarray}
Apply the functor $\Hom(-,\mathbb{G}_m)$ to obtain the left exact sequence:
\[
1\rightarrow \Hom(\operatorname{coker}f,\mathbb{G}_m)\xrightarrow{\hat{\pi}}\mathbb{G}_m^k\rightarrow \mathbb{G}_m^n.
\]

The group $S_\Sigma:=\Hom(\operatorname{coker}f,\mathbb{G}_m)$ acts on $U_\Sigma$, and so we define the following quotient stack:
\begin{definition}
    \label{Def:Coxstack}
    The \newterm{Cox stack} associated to $\Sigma$ is
    $\mathcal{X}_\Sigma:=[U_\Sigma/S_\Sigma].$
\end{definition}

We often simply say toric (DM) stacks when we mean their Cox stacks. Let us explicitly recall the following Theorem, highlighting the strength of considering Cox stacks instead of only toric varieties.
\begin{theorem}[Theorem 4.12 in \cite{FK18}]
    \label{Thm:4.12FK18}
    If $\Sigma$ is simplicial, then $\mathcal{X}_\Sigma$ is a smooth Deligne-Mumford stack with coarse moduli space $X_\Sigma$. When $\Sigma$ is smooth (equivalently, the variety $X_\Sigma$ is smooth), $\mathcal{X}_\Sigma\cong X_\Sigma$.
\end{theorem}

Within the class group of $\mathcal{X}_\Sigma$, we now define the Bondal-Thomsen collection.
\begin{definition}
    \label{Def:BTcolln}
    The \newterm{Bondal-Thomsen collection} for $\mathcal{X}$ is the set $\Theta_{\mathcal{X}}$ of degrees $-d\in \operatorname{Cl}(\mathcal{X})$ that are, for some $\theta\in M_\R$, equivalent to
    \[
    \sum_{\rho\in\Sigma(1)}\lfloor\langle-\theta,u_\rho\rangle\rfloor D_\rho.
    \]
\end{definition}

Notationally, we may refer to the class in $\Theta_\mathcal{X}$ corresponding to $\theta\in M_\R$ by $-d(\theta)$. Such an element is entirely determined by the image of $\theta$ in the torus $M_\R/M$. 

\begin{remark}
    The definition of the Bondal-Thomsen collection in \cite{BBB+} is a little more general, as the authors consider toric stacks from stacky fans, Recall that a stacky fan is determined by the data of the fan $\Sigma$ and a homomorphism $\beta:\Z^{\Sigma(1)}\rightarrow N$ with $\beta(e_\rho)=b_\rho u_\rho$ for some $b_\rho>0$. Choosing $b_\rho=1$ for all $\rho\in\Sigma(1)$ we recover the definition of Cox stacks above.
\end{remark}

\begin{definition}
\label{Def:Zonotope}
Define the (partial) zonotope $Z_\mathcal{X}$ associated to the toric DM stack $\mathcal{X}$ to be the image of $(-1,0]^{\Sigma(1)}$ under the map $\R^{\Sigma(1)}\rightarrow \operatorname{Cl}(\mathcal{X})_{\R}$ induced by \eqref{eqn:Coxseqpre}.
\end{definition}

\noindent The Bondal-Thomsen collection has several equivalent characterisations.

\begin{proposition}[Proposition 2.17 in \cite{BBB+}]
    \label{Prop:BTcollnequiv}
    The following subsets of $\operatorname{Cl}(\mathcal{X})$ are equal:
    \begin{enumerate}
        \item The Bondal-Thomsen collection $\Theta_\mathcal{X}$.
        \item The collection of Weil divisors in $\operatorname{Cl}(\mathcal{X})$ linearly equivalent in $\operatorname{Cl}(\mathcal{X})_\Q$ to divisors of the form $\sum_{\rho\in \Sigma(1)}a_\rho D_\rho$ with $-1<a_\rho\le0$.
        \item Elements whose image in $\operatorname{Cl}(\mathcal{X})_\Q$ is a lattice point of the zonotope $Z_\mathcal{X}$.
    \end{enumerate}
\end{proposition}

\begin{remark}
    The closure of $Z_{\mathcal{X}}$ is a lattice polytope with respect to the lattice defined via the image of $\Z^{\Sigma(1)}$ under the above map. It should be emphasised here that when we talk about lattice points inside the zonotope $Z_{\mathcal{X}}$, we always refer to this natural lattice.
\end{remark}

Using the coarse moduli space map $\mathcal{X}\rightarrow X$ for the underlying toric variety $X$, we can study the pushforward to the Bondal-Thomsen collection. We keep the formulation in terms of general stacky fans, but will really only consider the toric DM stacks $\mathcal{X}_\Sigma$ we defined before.

\begin{proposition}
    \label{Prop:CoarseModuli}
    If $\Sigma$ is simplicial and $\pi:\mathcal{X}_{\Sigma,\beta}\rightarrow X_\Sigma$ is the natural coarse moduli space map, then for all $\theta\in M_\R$
    \[
    \pi_\ast\O_{\mathcal{X}_{\Sigma,\beta}}(-d(\theta))=\O_{X_\Sigma}(-d(\theta)).
    \]
\end{proposition}

Using the Bondal-Thomsen collection, the authors of \cite{BBB+} construct a tilting bundle on the \newterm{Cox category}, which we will now introduce. Consider the collection $(\mathcal{X}_i)_{i=1}^r$ of smooth toric DM stacks associated to the chambers of the secondary fan $\Sigma_{GKZ}$ of $X$. There exists (see $\S 3$ of \cite{BBB+}) a smooth toric DM stack $\widetilde{\mathcal{X}}$ equipped with $r$ proper, birational morphisms $\pi_i:\widetilde{\mathcal{X}}\rightarrow \mathcal{X}_i$ such that $\pi^\ast_i:\dbcoh{\mathcal{X}_i}\rightarrow \dbcoh{\widetilde{\mathcal{X}}}$ is fully faithful for all $i$. \begin{definition}
    \label{Def:CoxCat}
    We define the Cox category of $X$ to be $\operatorname{D}_{Cox}(X):=\langle \pi_1^\ast(\mathcal{X}_1),\dots,\pi^\ast_r(\mathcal{X}_r)\rangle$.
\end{definition}
This definition turns out to be independent of choice of $\widetilde{\mathcal{X}}$ (see \cite[Corollary 3.9]{BBB+}). Observe now that the toric variety $X$ we started with admits a natural map (via \eqref{eqn:Coxseqpre}) $\operatorname{Cl}(X)\rightarrow\operatorname{Cl}(\mathcal{X}_i)$ for all $i$, which is an isomorphism if and only if $X$ and $\mathcal{X}_i$ share the same rays, and otherwise is surjective. As we will now elaborate upon, the Cox category has a tilting bundle extrapolated from the data of the Bondal-Thomsen collection. The Bondal-Thomsen collection $\Theta_\mathcal{X}$ surjects onto $\Theta_{\mathcal{X}_i}$ via this map. Hence, any $-d\in \Theta_\mathcal{X}$ has a natural corresponding element in $\Theta_{\mathcal{X}_i}$, and via abuse of notation we permit ourselves to write $\O_{\mathcal{X}_i}(-d)$ for an element $-d\in \Theta_{\mathcal{X}}$. To ease notation a little further, write $\Theta:= \Theta_{\mathcal{X}}$. For each $-d\in \Theta$ we will now define an element $\O_{Cox}(-d)$ in the Cox category $\operatorname{D}_{Cox}(X)$. We could in principal pull back the line bundle $\O_{\mathcal{X}_j}(-d)$ for any of the chambers, so the question is which chamber is the "best'' choice. The answer uses the geometry of the secondary fan.
\begin{definition} 
    \label{OCox(-d)}
    Pick $-d\in \Theta$. Note that $d$ is an effective degree, and so its image in $\operatorname{Cl}(X)_{\R}$ lies in some maximal cone $\Gamma_i$ of the secondary fan $\Sigma_{GKZ}$. This maximal chamber $\Gamma_i$ corresponds to some $\mathcal{X}_i$ and so we define $\O_{Cox}(-d):=\pi_i^\ast\O_{\mathcal{X}_i}(-d)$.
\end{definition}
At first instinct, one may doubt that this is well-defined: the image of $-d$ could lie in the intersection of two (or more) chambers $\Gamma_j\cap\Gamma_i$. However, Ballard et al. prove \cite[Proposition 4.2]{BBB+} that in this case, $\pi_i^\ast\O_{\mathcal{X}_i}(-d)=\pi^\ast_j\O_{\mathcal{X}_j}(-d)$.

An important computation tool for the Bondal-Thomsen collection is the $\Theta$-Transform Lemma:
\begin{lemma}[Lemma 1.6 in \cite{BBB+}]
    \label{Lem:ThetaTrans}
    Let $-d\in \Theta$ be an element whose image in $\Sigma_{GKZ}$ lies in the chamber $\Gamma_i$ corresponding to the toric DM stack $\mathcal{X}_i$. For any $j$, $\Phi_{ij}(\O_{\mathcal{X}_i}(-d))=\O_{\mathcal{X}_j}(-d)$, where $\Phi_{ij}:\dbcoh{\mathcal{X}_i}\rightarrow\dbcoh{\mathcal{X}_j}$ is the Fourier-Mukai transform induced by the graph of the birational map $\mathcal{X}_i\dashrightarrow \mathcal{X}_j$.
\end{lemma}

Using the $\Theta$-Transform Lemma and adjunction, one obtains the following computation.
\begin{corollary}
    \label{Cor:HomDCox}
    Let $-d,-d'\in \Theta$. Then \[
    \Rhom(\O_{Cox}(-d),\O_{Cox}(-d'))\cong \Rhom_{\mathcal{X}_i}(\O_{\mathcal{X}_i}(-d),\O_{\mathcal{X}_i}(-d'))\cong H^\bullet(\mathcal{X}_i,\O_{\mathcal{X}_i}(d-d')).
    \]
\end{corollary}

Ballard et al. prove the following stacky version of Demazure vanishing, allowing us to deduce that higher cohomologies vanish.
\begin{theorem}[Theorem 2.11 in \cite{BBB+}]
    \label{Thm:StackyDemazure}
    If $\mathcal{X}=\mathcal{X}_{\Sigma,\beta}$ is a smooth toric DM stack, $X_\Sigma$ is semi-projective, and $D$ is nef $\Q$-Cartier on $X_\Sigma$, then $H^p(\O_{\mathcal{X}}(\beta^\ast\lfloor D\rfloor)=0$ for all $p>0$, where $\lfloor \sum r_\rho D_\rho\rfloor=\sum \lfloor r_\rho\rfloor D_\rho$.
\end{theorem}

Finally, using the combinatorial nature of our setup, we establish the following. 

\begin{corollary}[Corollary 4.25 in \cite{BBB+}]
    \label{Cor:HomBT}
    Let $-d,-d'\in \Theta$. For any $\theta,\theta'\in M_\R$ such that $d=d(\theta), d'=d(\theta')$, we have\[
    \Rhom(\O_{Cox}(-d),\O_{Cox}(-d'))=k\langle P_d\cap (M-\theta')\rangle=k\langle Q_{d-d'}\cap M\rangle,
    \]
    concentrated in degree 0. Here, $P_d=\{m\in M_\R\vert \langle m,u_\rho\rangle \geq \lfloor -\langle \theta,u_\rho\rangle\rfloor \quad \forall \rho\in \Sigma(1)\}$ and $Q_{d-d'}=\{m\in M_\R\mid \langle m,u_\rho\rangle\geq-\lceil\langle\theta,u_\rho\rangle\rceil+\lceil\langle\theta',u_\rho\rangle\rceil\quad \forall\rho\in\Sigma(1)\}$.
\end{corollary}

One of the main results of the paper \cite{BBB+} now states that the Bondal-Thomsen collection actually gives a tilting object for $\operatorname{D}_{Cox}(X)$.
\begin{theorem}[Theorem A in \cite{BBB+}]
    \label{Thm:BTgivesTilting}
    Let $X$ be a semiprojective toric variety. The direct sum of the line bundles in $\Theta$ is a tilting object for $\operatorname{D}_{Cox}(X)$. If $X$ is projective, then $\Theta$ forms a full strong exceptional collection for $\operatorname{D}_{Cox}(X)$ under a natural ordering.
\end{theorem}

 Consider the reflexive sheaf $\bigoplus_{-d\in \Theta}\O_X(-d) =:\mathcal{T}$ and its (underived) endomorphism algebra $\Lambda=\Rhom^0_X(\mathcal{T},\mathcal{T})$. Our intuition tells us that this $\Lambda$ may be an NC(C)R for $X$.
 \begin{theorem}[Theorem 1.5 in \cite{BBB+}]
 \label{Thm:NCRviaBT}
 The algebra $\Lambda$ is an NCR for $X$: we have $\gldim(\Lambda)=\dim X$ and the functor $\operatorname{Perf}(X)\rightarrow \dbmod{\Lambda}$ given by $\mathcal{E}\mapsto \Rhom_X(\mathcal{T},\mathcal{E})$ is fully faithful. Furthermore, the algebra $\Lambda$ is uniform among any semiprojective toric variety $X$ whose fan has the same rays as $X$.
 \end{theorem}

In general, the category $\operatorname{D}_{Cox}(X)\cong \dbmod{\Lambda}$ gives a categorical resolution in the sense of Kuznetsov \cite{Kuz08}. Theorem \ref{Thm:NCRviaBT} provides the existence of NCRs for affine toric varieties, but unfortunately the property of crepancy rarely holds. This is not surprising: the object $\mathcal{T}$ is not generally a tilting bundle for $X$ itself, but instead for $\operatorname{D}_{Cox}(X)$, which in some sense will mostly be too big to reasonably be expected to give a crepant resolution. 
Nonetheless, one should not abandon the approach of constructing NCCRs via Bondal-Thomsen collections just yet. As the authors of \cite{BBB+} remark, for any $\theta\in M_\R$ the line bundle $\O_X(-d(\theta))$ recovers the conic module $A_\theta$ studied by Faber, Muller and Smith (and introduced in the previous section $\S$ \ref{sec:conic}). We will make precise what this means in the next section $\S$ \ref{sec:Link} (Proposition \ref{Prop:BTisConic}), but note here that the object $\mathcal{T}$ directly correspond to the complete sum of conic modules studied in \cite{FMS19}, and thus we know that the full Bondal-Thomsen collection gives an NCCR for affine toric varieties if and only if the underlying cone is simplicial.

In light of Example \ref{Exa:710FMS}, where a non-complete conic module provided an NCCR, a reasonable approach to explore is whether a subset of the Bondal-Thomsen collection exists that gives an NCCR. In spirit, we expect that NCCRs are in a categorical sense "minimal'': their derived category should embed inside the derived category of a non-crepant resolution (see \cite[$\S$ 3.3]{VdB23}) and be generated therein by a partial tilting object. We aim to construct such partial tilting objects via subets of the Bondal-Thomsen collection. 
In the case of Gorenstein affine toric varieties, this aligns with Theorem \ref{Thm:ParttiltingGorCone}. 

Consider a Gorenstein cone $\sigma$ with a simplicial subdivision such that $\Sigma$ with $\Sigma(1)=\sigma(1)$. Associate to it a toric DM stack $\mathcal{X}$. Fixing a chamber $\Gamma_i$ of the secondary fan $\Sigma_{GKZ}$ of $\Sigma$, let $\Theta_i\subset \Theta$ be the subset of the Bondal-Thomsen collection consisting of those $-d\in \Theta$ such that the image of $-d$ under the map $f_\sigma:\R^{\Sigma(1)}\rightarrow\operatorname{Cl}(\mathcal{X})_\R$ induced by \eqref{eqn:Coxseqpre} lies in the chamber $\Gamma_i$. As $\pi_i:\widetilde{\mathcal{X}}\rightarrow \mathcal{X}_i$ (where $\mathcal{X}_i$ is the toric DM stack associated to the chamber $\Gamma_i$) induces a fully faithful pullback $\pi_i^\ast:\dbcoh{\mathcal{X}_i}\rightarrow \operatorname{D}_{Cox}(X)$, we observe that $\mathcal{T}_i:=\bigoplus_{-d\in \Theta_i}\O_{\mathcal{X}_i}(-d)$ is a partial tilting object. Furthermore, the endomorphism algebra $\Lambda'=\End_{\dbcoh{\mathcal{X}}}(\mathcal{T}_i)$ is equal to $\End_{\operatorname{D}_{Cox}(X)}(\pi^\ast_i\mathcal{T}_i)$, a direct summand of $\Lambda=\End_{\operatorname{D}_{Cox}(X)}(\mathcal{T})$. It is known that $\Lambda$ has finite global dimension, so the pertinent question is when the same is true for $\Lambda'$.

\begin{corollary}
    \label{Cor:FinGlDimSubchamber}
    In the situation above, $\Lambda'$ is an NCCR of $R=k[\sigma^\vee\cap M]$ if and only if $\gldim\Lambda'<\infty$.
\end{corollary}
\begin{proof}
    If $\Lambda'$ is an NCCR, by definition its global dimension is finite. If the global dimension is finite, we are in a position to apply Theorem \ref{Thm:ParttiltingGorCone}.
\end{proof}

Since $\Lambda'$ is a direct summand of $\Lambda$ and thus a $\Lambda$-module. Similarly, the inclusion $\Lambda'\hookrightarrow \Lambda$ gives $\Lambda$ a $\Lambda'$-module structure. 
\begin{lemma}
    \label{Lem:GldimSummand}
    If $\operatorname{pd}_{\Lambda'}(\Lambda)<\infty$ then $\gldim\Lambda'<\infty$. 
\end{lemma}

\begin{proof}
We start by noting that every $\Lambda'$-module $M$ can be considered as $\Lambda$-module. Indeed, consider the projection map $\pi:\Lambda\rightarrow\Lambda'$, which is a ring homomorphism (as $\Lambda'$ is a direct summand). Then we can define the $\Lambda$-module structure on $M$ via $\lambda\cdot m=\pi(\lambda)\cdot m$.
By \cite{WeibelBook}, Theorem 4.3.1 (General change of rings Theorem), if $f:R\rightarrow S$ is a map of rings and $A$ is an $S$-module, then as an $R$-module we have
\begin{equation}\label{eq:pdeqn}
   \operatorname{pd}_{R}(A)\le \operatorname{pd}_S(A) + \operatorname{pd}_R(S).
\end{equation}
So consider the inclusion map $f:\Lambda'\rightarrow \Lambda$. Let $M$ be a $\Lambda'$-module, which also has a $\Lambda$-module structure. 
Then, $\eqref{eq:pdeqn}$ implies that
\[
\operatorname{pd}_{\Lambda'}(M)\le \operatorname{pd}_{\Lambda}(M)+\operatorname{pd}_{\Lambda'}(\Lambda)\le \gldim \Lambda+\operatorname{pd}_{\Lambda'}(\Lambda).
\]
Since this is true for all $\Lambda'$-modules $M$, we have \[\gldim\Lambda'=\sup\{\operatorname{pd}_{\Lambda'}(M)\ \vert M\in\operatorname{mod}\Lambda'\}\le \gldim\Lambda+\operatorname{pd}_{\Lambda'}(\Lambda)<\infty.\] 
\end{proof}

\section{Studying conic modules via the Bondal-Thomsen collection}\label{sec:Link}

In this section, we will elaborate on the link between conic modules and the Bondal-Thomsen collection we mentioned earlier and from that deduce how the geometry of the secondary fan can be used to study (sums of) conic modules. Let us start by precisely stating how conic modules are recovered by the Bondal-Thomsen collection. Fix a cone $\sigma\subset N_\R$ together with a simplicial fan $\Sigma$ such that $\Sigma(1)=\sigma(1)$ (in particular, $|\Sigma|=\sigma$). 

\begin{proposition}
    \label{Prop:BTisConic}
    Let $v,w\in M_\R$ and consider the conic modules $A_v, A_w$ as well as the elements of the Bondal-Thomsen collection $-d(v),-d(w)$. Then\[
    \Hom(A_v,A_w)\cong\Hom(\O_{Cox}(-d(v)),\O_{Cox}(-d(w))).
    \]
    For any subset $S\subset \Theta$, consider a minimal corresponding subset $S'\subset M_\R$ such that $\{-d(s)\mid s\in S'\}=S$. Then there is an isomorphism of endomorphism algebras \[\End(\bigoplus_{s\in S'} A_{s})\cong \End(\bigoplus_{\theta\in S}\O_{Cox}(-d(\theta))).\]
\end{proposition}

\begin{proof}
We use Proposition \ref{Prop:ConicModulesProperties} and Corollary \ref{Cor:HomDCox} to compute the two Hom-sets:
\begin{eqnarray*}
&&\Hom(A_v,A_w)=\langle x^m\mid m+(\sigma^\vee+v)\cap M\subset (\sigma^\vee+w)\rangle;\\
&& \Hom(\O_{Cox}(-d(v)),\O_{Cox}(-d(w)))=\langle x^m\mid m\in Q_{d(v)-d(w)}\rangle
\end{eqnarray*}

First, suppose we are given an $m$ such that $m+(\sigma^\vee+v)\cap M\subset(\sigma^\vee+w)$.

For any $m'\in (\sigma^\vee+v)\cap M$, we have $m+m'-w\in \sigma^\vee$, i.e.\[
\langle m,u_\rho\rangle \ge\langle w,u_\rho\rangle-\langle m',u_\rho\rangle,\quad \forall\rho\in\sigma(1).
\]
The LHS is integer, and so we have, for all $\rho\in\sigma(1)$ and any $m'\in (\sigma^\vee+v)\cap M$,
\begin{eqnarray*}
&& \langle m,u_\rho\rangle\ge \lceil\langle w,u_\rho\rangle-\langle m',u_\rho\rangle\rceil\\
&& \quad \quad \quad\ \ge \lceil\langle w,u_\rho\rangle\rceil -\lceil \langle m',u_\rho\rangle\rceil.\\
\end{eqnarray*}

As $m'\in\sigma^\vee+v$, we have $\langle m',u_\rho\rangle\ge \langle v,u_\rho\rangle$. The shifted cone $\sigma^\vee+v$ contains points that attain this bound on its boundary and the hyperplane $\{x\mid\langle x,u_\rho\rangle =\lceil \langle v,u_\rho\rangle\rceil\}$ intersects $\sigma^\vee+v$, as it lies on the correct side of the supporting hyperplane. Thus there is a lattice point $m''\in \sigma^\vee+v$ where $\langle m'',u_\rho\rangle =\lceil\langle v,u_\rho\rangle\rceil$. Consequently, $\langle m,u_\rho\rangle\ge \lceil\langle w,u_\rho\rangle\rceil-\lceil\langle v,u_\rho\rangle\rceil$, i.e. $m\in Q_{d(v)-d(w)}$.

Conversely, suppose $m\in Q_{d(v)-d(w)}\cap M$. Let $m_2\in (\sigma^\vee+v)\cap M$. For any $\rho\in\sigma(1)$, $\langle m_2,u_\rho\rangle \ge \langle v,u_\rho\rangle$ and as the LHS is integer, we deduce $\langle m_2,u_\rho\rangle \ge \lceil\langle v,u_\rho\rangle\rceil$.

Then \begin{eqnarray*}
   && \langle m+m_2-w,u_\rho\rangle=\langle m,u_\rho\rangle +\langle m_2,u_\rho\rangle -\langle w,u_\rho\rangle\\
   && \ge -\lceil\langle v,u_\rho\rangle\rceil +\lceil\langle w,u_\rho\rangle\rceil +\langle m_2,u_\rho\rangle-\langle w,u_\rho\rangle\\
   &&=\left(\lceil \langle w,u_\rho\rangle \rceil-\langle w,u_\rho\rangle\right)+\left(\langle m_2,u_\rho\rangle-\lceil \langle v,u_\rho\rangle\rceil\right)\\
   && \ge 0+0=0.\;
\end{eqnarray*}

Hence, for all $\rho\in\sigma(1)$ and $m_2\in (\sigma^\vee+v)\cap M$, we have $\langle m+m_2-w,u_\rho\rangle\ge0$, i.e. $m+m_2\in \sigma^\vee+w$. So $m\in Q_{d(v)-d(w)}$ implies $m+(\sigma^\vee+v)\cap M\subset \sigma^\vee+w$. 

The second part of the statement follows directly from the first.
\end{proof}

Let us reexamine now Example \ref{Exa:710FMS}, drawing the link to Bondal-Thomsen collections and secondary fans. We are given the cone \[
\sigma=\cone((1,0,0),(0,1,0),(-1,0,1),(0,-1,1)).
\]
From $\sigma$, we build a simplicial fan $\Sigma_1$ by refining the cone into a union of two simplicial cones (obtained by inserting the diagonal connecting $(0,1,0)$, $(0,-1,1)$ into the square defining the cone) \[
\sigma_1=\cone((1,0,0),(0,1,0),(0,-1,1)),\quad \sigma_2=\cone((0,1,0),(-1,0,1),(0,-1,1)).
\]
Of course, we could have chosen to simplicially refine the cone via the other diagonal. This gives the fan $\Sigma_2$, whose two maximal cones are \[
\sigma_3=\cone((1,0,0),(0,1,0),(-1,0,1)),\quad \sigma_4=\cone((1,0,0),(-1,0,1),(0,-1,1)).
\]
The secondary fan $\Sigma_{GKZ}$ has two chambers, corresponding to the toric DM stacks $\mathcal{X}_{\Sigma_1}$ and $\mathcal{X}_{\Sigma_2}$. In fact, the secondary fan $\Sigma_{GKZ}\subset \R$ consists of two rays with primitive generators $-1,1$ together with the 0-dimensional cone at the origin. We do note that the rays both appear with multiplicity 2 and so the corresponding zonotope contains 3 lattice points: $-1$, $0$, $+1$. The point $+1$ lies in the chamber corresponding to $\mathcal{X}_{\Sigma_1}$ and the point $-1$ in the chamber corresponding to $\mathcal{X}_{\Sigma_2}$. Consider the three suggestively named points $v_0=(0,0,0)$, $v_{+}=(0,-\frac{1}{4},0)$ and $v_{-}=(-\frac{1}{4},0,0)$, all in $M_\R$. Recall that the toric algebra $R=k[\sigma^\vee\cap M]$ associated to the cone is $R=k[x,y,z,w]/(xz-yw)$. We compute the associated conic modules and identify them with $A_0, A_1$ and $A_2$ as written in Example \ref{Exa:710FMS}:
\begin{eqnarray*}
    && A_{v_0}=R=A_0,\\
    && A_{v_+}=(x,y)R=A_1,\\
    && A_{v_-}=(x,w)R=A_2.\;
\end{eqnarray*}
Enumerate the divisors associated to the rays of $\sigma$ in the following way:
\[
   (1,0,0) \mapsto D_1, \quad
    (0,1,0)\mapsto D_2, \quad
    (-1,0,1)\mapsto D_3, \quad  
    (0,-1,1)\mapsto D_4.
\]
Then $-d(v_0)=0$, $-d(v_{+})=-D_4$ and $-d(v_{-})=-D_3$. The map $\R^{\Sigma(1)}\rightarrow \operatorname{Cl}(\mathcal{X})_{\R}$ here is the linear map $(1,-1,1,-1)$ and so $-d(v_0)\mapsto 0$, $-d(v_+)\mapsto +1$ and $-d(v_-)\mapsto -1$. 

We note therefore that choosing the incomplete sum of conic modules $A_0\oplus A_1$ corresponds to taking the subobject $\O_{Cox}(-d(v_0))\oplus \O_{Cox}(-d(v_+))$ of the tilting object for $\operatorname{D}_{Cox}(X)$. This tilting object is built of line bundles where the $-d(\theta)$ have image in the same chamber of the secondary fan, the chamber which contains the point $+1$ and corresponds to the smooth toric DM stack $\mathcal{X}_{\Sigma_1}$. But then, $\O_{Cox}(-d(v_0))\oplus \O_{Cox}(-d(v_+))=\pi_1^\ast \O_{\mathcal{X}_{\Sigma_1}}(-d(v_0))\oplus\pi_1^\ast\O_{\mathcal{X}_{\Sigma_1}}(-d(v_+))$, where $\pi_1:\widetilde{\mathcal{X}}\rightarrow \mathcal{X}_{\Sigma_1}$ is the proper, birational morphism from $\widetilde{\mathcal{X}}$ to $\mathcal{X}_{\Sigma_1}$ used to define the Cox category in Definition \ref{Def:CoxCat}. In particular, $\pi_1^\ast$ is fully faithful, and so $\mathcal{T}':=\O_{\mathcal{X}_{\Sigma_1}}(-d(v_0))\oplus\O_{\mathcal{X}_{\Sigma_1}}(-d(v_+))$ is a partial tilting object on $\dbcoh{\mathcal{X}_{\Sigma_1}}$ such that $\End_{\mathcal{X}_{\Sigma_1}}(\mathcal{T}')=\End_{\operatorname{D}_{Cox}(X)}(\pi_1^\ast\mathcal{T}')$. To show this has finite global dimension, it is sufficient to observe that we obtain finite length projective resolutions of all the simple modules. For this it was sufficient to have the two complexes
\begin{eqnarray*}
        && A_0\rightarrow A_1^{\oplus 2}\oplus A_0^{\oplus 4}\rightarrow A_0^{\oplus 4}\oplus A_1^{\oplus 2}\rightarrow A_0,\\
        && A_1\rightarrow A_0^{\oplus 2}\rightarrow A_0^{\oplus 2}\rightarrow A_1.\;
    \end{eqnarray*}
These complexes only use conic modules in the set $A_0, A_1$, corresponding to the two points $0, +1$ inside the chamber of $\mathcal{X}_{\Sigma_1}$. Applying $\Hom(A_0\oplus A_1,-)$ produces the desired finite length projective resolutions of the two unique graded simple $\End(A_0\oplus A_1)$-modules. Of course we are technically done here by showing that these resolutions all have the same length, but knowing finite global dimension of the endomorphism algebra is also sufficient to apply Theorem \ref{Thm:ParttiltingGorCone} and obtain that $\Lambda=\End_{\mathcal{X}_{\Sigma_1}}(\mathcal{T}')=\End_R(A_0\oplus A_1)$ is an NCCR.

We formalise now the substitution argument used in the example above.

\subsection{Substitution of conic modules and NCCRs}
Fix a cone $\sigma$ together with a simplicial subdivision $\Sigma$ such that $\sigma(1)=\Sigma(1)$ and let $s$ be the number of isomorphism classes of conic modules and $r$ the number of lattice points in the zonotope $Z_\mathcal{X}$ of the smooth toric DM stack $\mathcal{X}_\Sigma$. Enumerate the isomorphism classes of conic modules $A_0,\dots,A_{s-1}$ and we choose a representative $\Delta_i$ of chamber of constancy for each isomorphism class. For $i\in\{0,\dots,s-1\}$ let $K_i^\bullet$ be the complex $K_{\Delta_i}^\bullet$.

\begin{definition}
    \label{Def:Substitution}
    If $A_j$ appears in a complex of conic modules $K^\bullet$, we define the \newterm{substitution by $j$}, $S(j,K^\bullet)$, to be the complex obtained by splicing all occurences of $A_j$ with $K_j^\bullet$. For a set $\mathcal{K}$ of complexes of conic modules, define the \newterm{substitution by $j$} $\mathcal{S}(j,\mathcal{K})$ of the set $\mathcal{K}$ by $j$ to be the set of complexes $\{S(j,K)\mid K\in \mathcal{K}\}$. Call a non-empty subset $I\subseteq \{0,\dots,s-1\}$ \newterm{lockable} if there is a finite sequence of successive substitutions by elements in $I^c$ that, applied to the set $\{K_i^\bullet\mid i\in I\}$, yields a set $\mathcal{K}^\dagger_I$ of complexes of conic modules, such that all conic modules appearing in any complex $K\in \mathcal{K}^\dagger_I$ are of the form $A_k, k\in I$.
\end{definition}

\begin{example}
    \label{Exa:FMSviaLockable}
    In Example \ref{Exa:710FMS}, the set $\{0,1\}$ is lockable as substituting $K_2$ for every appearance of $A_2$ in $K_0, K_1$ yields the set of complexes \[\mathcal{K}_{\{0,1\}}^\dagger=\{
    A_0\rightarrow A_1^{\oplus 2}\oplus A_0^{\oplus 4}\rightarrow A_0^{\oplus 4}\oplus A_1^{\oplus 2}\rightarrow A_0,\quad A_1\rightarrow A_0^{\oplus 2}\rightarrow A_0^{\oplus 2}\rightarrow A_1
    \}.\]
    Similarly for the set $\{0,2\}$.
\end{example}
\noindent The above example generalises quite straightforwardly.
\begin{example}
    If there is a complex such that $K_i^\bullet$ does not contain the conic modules $A_i$ outside of degree $0$, then $\{0,\dots,s-1\}\setminus \{i\}$ gives a lockable set.
\end{example}

We wish to show that, as above in Example \ref{Exa:FMSviaLockable}, lockable subsets give non-commutative resolutions. To do so, we show that for an incomplete sum of conic modules $\mathbb{A}$ coming from a lockable subset, every graded simple $\End_R(\mathbb{A})$-module has finite projective dimension. Specifically, we show that applying $\Hom_R(\mathbb{A},-)$ to the set of complexes $K_I^\dagger$ gives a set of finite projective resolutions to the simples, thus giving finite global dimension of $\End_R(\mathbb{A})$. Moreover, we show that these resolutions are minimal and hence we can determine if the lockable subset gives an NCCR or merely an NCR by checking if the projective dimension is equal to $\dim \sigma$ for each graded simple, the same way the authors do in \cite{FMS19}. This property is that of an \newterm{incredulous} set.
\begin{definition}
    \label{Def:Incredulous}
    A lockable set $I$ is called \newterm{incredulous} if all complexes $K\in K_I^\dagger$ have length $\dim \sigma$.
\end{definition}

Our main result of this section is a formal proof of the fact that the incomplete sums of conic modules which give NCCRs corresponds precisely to incredulous sets. 

\begin{theorem}
\label{Thm:NCRiffLock}
Let $\sigma$ be a cone with associated toric algebra $R$ and collection of conic modules $\{A_i\}_{i\in S}$ for a set of indices $S$. For a subset $I$ of $S$, consider the incomplete direct sum of conic modules $\mathbb{B}=\bigoplus_{i\in I}A_i$. Then the endomorphism algebra $\Lambda'=\End_R(\mathbb{B})$ is an NCR of $R$ if and only if $I$ is lockable. Furthermore, $\Lambda'$ is an NCCR of $R$ if and only if $I$ is incredulous.
\end{theorem}

We split the proof into two parts (Propositions \ref{Prop:LockGivesNCR} and \ref{Prop:NCRgivesLock}), each proving one direction of the equivalence. First, however, we digress shortly to recall some standard facts about endomorphism algebras, laid out in \cite{FMS19}.

Consider a direct sum of conic modules $\mathbb{B}=\bigoplus_{i\in I}A_i$. Then we have the following result.
\begin{proposition}[Proposition 6.3 \cite{FMS19}]
    \label{Prop:FMS 6.3}
    For any $i$,
    \begin{enumerate}
        \item The $\End_R(\mathbb{B})$-module \[
        P_i':=\Hom_R(\mathbb{B},A_i)
        \] is a graded indecomposable projective module and every graded indecomposable projective $\End_R(\mathbb{B})$-module is isomorphic to a module of this form.
        \item The projective module $P_i'$ has a unique maximal graded submodule. Every graded simple $\End_R(\mathbb{B})$-module is isomorphic to a quotient of some $P_i'$ by this maximal submodule. Denote by $S_i'$ the unique graded simple quotient of $P_i'$.
        \end{enumerate}
\end{proposition}
\begin{remark}
    In their paper, Faber, Muller and Smith formulate this result for complete direct sums of conic modules but the proof follows through as is for incomplete direct sums of conic modules. In fact, as remarked by the authors, this is essentially an adaptation of \cite[Lemma 4.1.1]{SVdB02}.
\end{remark}
The unique maximal graded submodule of $P_i'$ is the \newterm{radical} of $P_i'$. A homomorphism from a conic module $A_j$ to another conic module $A_i$ is said to be \newterm{radical} if it is not an isomorphism. In more generality, a homomorphism from a direct sum $\mathbb{B}$ of conic modules to another direct sum of conic modules $\mathbb{A}$ can be represented by a matrix $[\phi_{i,j}]$ and we say this homomorphism is radical if all of its components are. Thus, elements of $P_i'=\Hom_{\End_R(\mathbb{B})}(\mathbb{B},A_i)$ can be represented as row vectors with entries in $\Hom_R(A_j,A_i)$ with the radical consisting of those maps with entries in the respective radical of $\Hom_R(A_j,A_i)$. Denote the radical by $\operatorname{Rad}(\mathbb{B},A_i)$.

\begin{lemma}
    \label{Lem:vspIso}
    For any $i,j\in I$ there is a graded $k$-vector space isomorphism\[
    \Hom_{\End_R(\mathbb{B})}(P_i',S_j')\simeq \Hom_R(A_i,A_j)/\operatorname{Rad}(A_i,A_j).
    \]
    These vector spaces are of dimension 0 or 1, depending on whether $A_i\simeq A_j$.
\end{lemma}
\begin{proof}
    This is the version of \cite[Lemma 6.7]{FMS19} for incomplete sums of conic modules and the proof applies verbatim.
\end{proof}

We now begin the proof of Theorem \ref{Thm:NCRiffLock}, showing first that a lockable set yields an NCR. The first observation to make is that if a set is lockable, the set of complexes after substitution, $K_I^\dagger$, is uniquely determined.

\begin{lemma}
\label{Lem:SubstitutionIsUnique}
    Let $I$ be a lockable set and let $i\in I$. Any order of substitutions by $j\not\in I$ applied to $K_i^\bullet$ and resulting in a complex without appearance of conic modules $A_j, j\not\in I$ gives the same complex $K_i^\dagger$.
\end{lemma}
\begin{proof}
Consider an appearance of $A_j$ in $K_i^\bullet$ such that $j\not\in I$. Then $A_j$ remains in the complex at the same degree unless one substitutes by $K_j^\bullet$. When one does so, the same modules appear in the same degrees, irrespective of previous substitutions. Any of these modules of the form $A_{j'}. j'\not\in I$ similarly remains in its degree until a substitution by $K_{j'}^\bullet$ takes place, which then uniquely determines the appearance of the modules coming from that specific copy of the module $A_{j'}$. Thus, since one exactly substitutes by complexes $K_j^\bullet$ with $j\not\in I$ and stops doing so when none appear, the resulting complex $K_i^\dagger$ is the same regardless of order of substitutions.
\end{proof}

We begin by establishing exactness and resolution properties of the sequences $\Hom_R(\mathbb{B},K_i^\bullet)$.

\begin{lemma}
\label{Lem:Extend by 0 if not in set}
    If $j\not\in I$, then $\Hom_R(\mathbb{B},K_j^\bullet)$ is an exact complex that can be extended to the right by zero and remains exact. Hence, $\Hom(\mathbb{B},K_j^\bullet)$ is an exact complex ending in $\Hom_R(\mathbb{B},A_j)=P_j'\rightarrow 0$ and can thus be seen as a resolution of $P_j'$.
\end{lemma}
\begin{proof}
    Note that $\Hom_R(\mathbb{B},K_j)=\bigoplus_{i\in I}\Hom_R(A_i,K_j)$ and so by the Acyclicity Lemma \ref{Lem:Acyc} it is a direct sum of exact complexes, hence exact. It suffices to find the cokernel of the last map to show we can extend the complex by 0 to the right. But the image of the incoming map is just $\operatorname{Rad}(\mathbb{B}, A_j)$, which is precisely $\Hom(\mathbb{B},A_j)$.
\end{proof}

\begin{lemma}
\label{Lem:Substituted is proj resn}
  For $i\in I$,  $K_i^\dagger$ is an exact complex outside of degree 0. Extending to the right by $[S_i'\rightarrow 0]$ makes $K_i^\dagger$ exact and thus a projective resolution of $S_i'$.
\end{lemma}
\begin{proof}
    By the Acyclicity Lemma \ref{Lem:Acyc}, the complex $\Hom_R(\mathbb{B},K_i^\bullet)=\bigoplus_{i'\in I}\Hom_R(A_{i'},K_i^\bullet)$ is exact apart from perhaps the penultimate spot, where the homology is one-dimensional in homological degree 0, coming from the direct summand $\Hom_R(A_i,K_i^\bullet)$. Note that the kernel of the projection map $P_i'\rightarrow S_i'$ is the radical of $P_i'$ and thus coincides with the homology appearing in degree 0. Thus extending by $S_i'$ gives an exact complex $\Hom_R(\mathbb{B},K_i^\bullet)\rightarrow S_i'\rightarrow 0$. As the complexes $\Hom_R(\mathbb{B},K_j^\bullet)\rightarrow 0$ for $j\not\in I$ are exact, substitution by $K_j^\bullet$ in $K_i^\bullet$ corresponds to splicing the exact complexes resolving $P_j'$ by Lemma \ref{Lem:Extend by 0 if not in set}. The substitution process does this until no non-projective direct summand remains in any degree $>0$ and thus we obtain a projective resolution of $S_i'$.
\end{proof}

The following Lemma helps us compute the Ext groups between the graded simples, which is a useful tool in establishing minimal projective resolutions. Essentially, this is the version of \cite[Proposition 6.6]{FMS19} for incomplete direct sums of conic modules and the proof runs in parallel.
\begin{lemma}
    \label{Lem:DimExt}
    For any pair $i,j\in I$ we have
    \[
    \dim(\Ext^l_{\End_R(\mathbb{B})}(S'_i,S'_j))=\#\text{occurences of } A_j \text{ in }K_i^{\dagger,l}.
    \]
\end{lemma}

\begin{proof}
    This is essentially a variation of \cite[Proposition 6.6]{FMS19} for incomplete direct sums of conic modules and the proof is analogous.
\end{proof}

We immediately see that the complexes $K_i^{\dagger,\bullet}$ generate minimal projective resolutions of $S_i'$.
\begin{corollary}
\label{Cor:Substituted is minimal}
    For $i\in I$, $K_i^\dagger$ is a minimal projective resolution of $S_i'$.
\end{corollary}
\begin{proof}
    Let $A_k$ be a conic module appearing in the highest degree of $K_i^{\dagger,\bullet}$. Denote by $d$ the length of $\Hom_R(\mathbb{B},K_i^{\dagger,\bullet})$. Then \[
    \Ext^{d}_{\End_R(\mathbb{B})}(S_i',S_k')\neq 0.
    \]
    Hence, $S_i'$ cannot have a projective resolution of length shorter than $d$, and $\Hom_R(\mathbb{B},K_i^{\dagger,\bullet})$ is indeed a minimal projective resolution.
\end{proof}

With this, we can prove the first direction of Theorem \ref{Thm:NCRiffLock}.

\begin{proposition}
    \label{Prop:LockGivesNCR}
    Let $I$ be a lockable set and let $\mathbb{B}=\bigoplus_{i\in I}A_i$. Then $\Lambda'=\End_R(\mathbb{B})$ is an NCR. If $I$ is incredulous, $\Lambda'$ is an NCCR.
\end{proposition}
\begin{proof}
    The complexes $\Hom_R(\mathbb{B},K_i^\dagger)$ give projective resolutions of any graded simples $S_i'$ ($i\in I$) of $\End_R(\mathbb{B})$ and so as the set is lockable we have $\operatorname{pd}_{\Lambda'}(S_i')=\operatorname{len}(K_i^\dagger)<\infty$. There is a finite collection of these and so
    we have $\gldim\Lambda'<\infty$. This implies that $\Lambda'$ is an NCR. Furthermore, if the set is incredulous, the projective dimension of all simples agrees with the Krull dimension of $R$ and thus $\Lambda'\in \operatorname{CM}R$, making it an NCCR of $R$.
\end{proof}

We now proceed to prove the second half of Theorem \ref{Thm:NCRiffLock}.
\begin{lemma}
    \label{Lem:NoPureDegree1Loop}
    For $\dim \sigma\ge 2$, there is no complex $K_i$ of length 1 or length 2 (i.e. of the form $\bigoplus_{\tau} \opA_\tau\rightarrow A_i$). 
    Hence there does not exist a collection of conic modules $\{A_l\}_{l\in T}$ for some subset $T\subset S$ such that each $K_l^\bullet$ is of length 0 and there is a loop of the modules $A_l$.
\end{lemma}
\begin{proof}
Fix a chamber of constancy $\Delta$. We will first show that there exists some open cell of codimension $\ge 2$. $\Delta$ is defined via the set of inequalities $d_\rho-1<\langle x,u_\rho\rangle \le d_\rho$ where $d_\rho=\lceil\langle v,u_\rho\rangle \rceil$ for some element $v\in\Delta$. An open cell $\tau$ is the locus such that subset of these inequalities are strict. Thus, we need to show that there always needs to be some inequality that can be made strict, i.e. that the set of (in)-equalities given by \[\begin{cases}
    d_\rho-1<\langle x,u_\rho\rangle <d_\rho, \text{ for }\rho\neq \rho_1\\
    d_{\rho_1}=\langle x,u_{\rho_1}\rangle
\end{cases}\]
has a solution for some $\rho_1\in \sigma(1)$. Note that $\sigma^\vee$ is full dimensional, and so there exists a point $p\in \operatorname{Int(}\sigma^\vee)$, i.e. an element $p\in M_\R$ such that $\langle p, u_\rho\rangle>0$ for all $\rho\in \sigma(1)$. Pick $v\in \Delta$. If any inequality is strict, we have the existence of an open cell. So suppose not, i.e. $d_\rho-1<\langle v,u_\rho\rangle<d_\rho$ for all $\rho\in \sigma(1)$. 
Using the intermediate value theorem we see that there exists $\varepsilon>0$ such that $d_\rho-1< \langle v+\epsilon p, u_\rho\rangle\le d_\rho$ but with some strict equalities. 
Thus, there is always some open cell of a chamber of constancy $\Delta$. To show there has to be an open cell of codimension $\ge 2$, suppose all open cells have codimension 1 and one of them, $\tau_1$, is given by $\langle x,u_{\rho_1}\rangle =1$ for $\rho_1\in \sigma(1)$. 
Consider the space $L_{\rho_1}=\{y\in M_\R\mid\langle y,u_{\rho_1}\rangle=0\}$. Since $u_{\rho_1}$ is a minimal ray generator of $\sigma$, $\sigma^\vee\cap L_{u_{\rho_1}}$ is a codimension 1 facet of $\sigma^\vee$ with inward pointing normal $u_{\rho_1}$. Then there exists an element in the interior of that facet, i.e. $q\in \sigma^\vee\cap L_{\rho_1}$ with $\langle p,u_\rho\rangle>0$ for $\rho\neq \rho_1$ and $\langle p,u_{\rho_1}\rangle=0$. Let $w\in \tau_1$. Then for $\varepsilon>0$, $\langle w+\varepsilon q,u_\rho\rangle>\langle w,u_\rho\rangle$ if $\rho\neq \rho_1$ and $\langle w+\varepsilon q,u_{\rho_1}\rangle=\langle w,u_{\rho_1}\rangle$. Thus, another application of the intermediate value theorem gives us a value for $\varepsilon$ such that $d_\rho-1<\langle w+\varepsilon q,u_\rho\rangle\le d_\rho$ with equality for $\rho=\rho_1$ and a non-empty collection of further $\rho\in \sigma(1)$, and as any two $u_\rho$ are linearly independent, the open cell needs to at least have codimension 2. 
\end{proof}

 To prove that finite global dimension implies lockability, we use the language of $\mathcal{I}$-minimal resolutions, and refer the reader to \cite{DFI16} for a more detailed approach. 
\begin{definition}
    Let $\mathcal{I}$ be a full subcategory of $\operatorname{CM}R$. Fix an exact sequence $0\rightarrow Z\rightarrow Y\xrightarrow X$ in $\operatorname{CM}R$. We say $f$ is a \newterm{right $\mathcal{I}$-approximation} if $Y\in \mathcal{I}$ and if \[
    \mathcal{I}(-,Y)\xrightarrow{\Hom_R(f)} \mathcal{I}(-,X)\rightarrow 0
    \] 
    is exact. A right $\mathcal{I}$-approximation is said to be \newterm{minimal} if there is no non-zero direct summand of $Y$ mapped to zero under $f$. An exact sequence $F:\dots\xrightarrow{f_2}Y_1\xrightarrow{f_1}Y_0\xrightarrow{f_0} X$ is called an $\mathcal{I}$-resolution of $X$ if each $Y_i\in \mathcal{I}$ and $\mathcal{I}(-,F)$ is exact on $\mathcal{I}$. Such a resolution is \newterm{minimal} if each $f_i$ is. In the following, we may choose to drop the word \textit{right} from the notation.
\end{definition}

\begin{proposition}
    \label{Prop:fingldim implies lockable}
    If $\Lambda'=\End_R(\mathbb{B})$ has finite global dimension, then $I$ is lockable.
\end{proposition}
\begin{proof}
    Consider the full subcategory $\mathcal{I}=\operatorname{add}(I)$ of $\operatorname{CM}R$. By \cite[Lemma 2.7]{DFI16}, $\mathcal{I}$ is contravariantly finite, i.e. each $X\in \operatorname{CM}R$ has an $\mathcal{I}$-minimal resolution. Note that for $i\in I$, the minimal resolution is simply given by $0\rightarrow A_i\xrightarrow{id}A_i\rightarrow 0$. Since the boundary maps in the complexes $K_j^\bullet$ are direct sums of signed inclusions, splicing for each $A_l$ appearing in $K_j^\bullet$ the relevant $\mathcal{I}$-minimal resolution itself results in a $\mathcal{I}$-minimal resolution of $A_j$ (as composition of a right minimal approximation with an inclusion gives a right minimal approximation). If $j\not\in I$, applying $\Hom(\mathbb{B},-)$ to this $\mathcal{I}$-minimal resolution of $A_j$ gives (see \cite[Lemma 2.10]{DFI16})  a minimal projective resolution of $P'_j$. By finite global dimension of $\Lambda'$, this is a finite complex.
    
   \noindent Suppose now that $I$ was not lockable. Then there is a sequence $A_{j_1},\dots, A_{j_m}, A_{j_{m+1}}=A_{j_1}$ with $A_{j_a}$ appearing in $K^\bullet_{j_{a+1}}$ and $j_a\not\in I$ for $1\le a\le m$. Then the $\mathcal{I}$-minimal resolution of $A_{j_{a+1}}$ contains the $\mathcal{I}$-minimal resolution of $A_{j_{a}}$ for $1\le a\le m$ - which means that the $\mathcal{I}$-minimal resolution of $A_{j_1}$ contains itself. Since Lemma \ref{Lem:NoPureDegree1Loop} excludes the loop of being in pure degree 0, this is a contradiction to the finiteness of the resolution of $P'_j$.

\end{proof}

\begin{proposition}
    \label{Prop:NCRgivesLock}
    Let $I$ be a subset of $S$ and $\mathbb{B}=\bigoplus_{i\in I}A_i$. If $\Lambda'=\End_R(\mathbb{B})$ is an NCR, then $I$ is lockable. If $\Lambda'$ is an NCCR, then $I$ is incredulous.
\end{proposition}

\begin{proof}
This is immediate from Proposition \ref{Prop:fingldim implies lockable}, as being an NCR implies that the global dimension of $\Lambda'$ is finite. Hence the set is lockable. The conclusion on being incredulous follows. 
\end{proof}

\begin{proof}[Theorem \ref{Thm:NCRiffLock}]
    The Theorem is a combination of Propositions \ref{Prop:LockGivesNCR} and \ref{Prop:NCRgivesLock}.
\end{proof}

Observe the following corollary of Theorem \ref{Thm:NCRiffLock}.

\begin{corollary}
\label{Cor:SubsetOFIncr}
Let $J$ be an incredulous set. Then if $I\subseteq J$ is lockable, it is incredulous. 
\end{corollary}
\begin{proof}
    Consider the two endomorphism algebras $\Lambda_I=\End_R(\bigoplus_{i\in I}A_i)$ and $\Lambda_J=\End_R(\bigoplus_{j\in J}A_j)$. Then $\Lambda_J=\Lambda_I\oplus \Lambda_L$ for some algebra $\Lambda_L$. Note $\Lambda_J\in \operatorname{CM}R\Leftrightarrow \Ext^i(\Lambda_J,R)=0$ for all $i>0$. But $\Ext^i(\Lambda_J,R)=\Ext^i(\Lambda_I,R)\oplus \Ext^i(\Lambda_L,R)$. 

    Thus $\Ext^i(\Lambda_I,R)=0$ for $i>0$ and so $\Lambda_i\in \operatorname{CM}R$. Since $I$ is lockable, by Theorem \ref{Thm:NCRiffLock}, $\Lambda_I$ is an NCR and so $\Lambda_I\in \operatorname{CM}R$ shows that $\Lambda_I$ is an NCCR. Applying the Theorem again gives that $I$ is an incredulous set.
\end{proof}

Knowing that looking for NCCRs as endomorphism algebras of incomplete sums of conic modules is the same as looking for incredulous sets, we now focus on computing the complexes $K_i^\bullet$ so that we can identify such sets.

\section{Computing complexes of conic modules}\label{sec:Compute}

To compute the complexes $K_i^\bullet$ associated to conic modules $A_i$, we are going to use the geometry of the secondary fan, demonstrating the computational advantage of this point of view.

\subsection{Complexes of conic modules and paths in the secondary fan}

To study when the method of substitution works, we need to understand first and foremost which conic modules appear in a given complex $K_v^\bullet$. In the following, we investigate this question using the combinatorics and geometry of the secondary fan.
Recall that we fixed a cone $\sigma$ with simplicialisation $\Sigma$, $r$ the number of lattice points in $Z_\mathcal{X}$ and $s$ the number of isomorphism classes of conic modules.

\begin{lemma}
    \label{Lem:Number of conic modules}
    With $\sigma,\Sigma, r, s$ as above, \[
    s=r\cdot |\operatorname{Tors}(\operatorname{Cl}(\mathcal{X}_\Sigma))|=|\Theta_\mathcal{X}|.
    \]
\end{lemma}
\begin{proof}
    For each isomorphism class of conic modules, pick a representative $A_{v_i}$, $0\le i\le s-1$ such that $v_i$ lies in the fundamental domain of $M_\R/M$ (which we can usually think of as $(-1,0]^{\dim\sigma}$). 
    Given such an element $v_i$, by definition $-d(v_i)$ lies in the Bondal-Thomsen collection $\Theta_{\mathcal{X}_\Sigma}$. Two vectors $v,w\in M_\R$ give the same isomorphism class of conic module $A_v$ if and only if $-d(v)\sim -d(w)$. Then, however, $-d(v)$ and $-d(w)$ designate the same element in the Bondal-Thomsen collection and so $s$ is simply the size of the Bondal-Thomsen collection. Using Proposition \ref{Prop:BTcollnequiv}, this is equivalent to the subset of divisors in $\operatorname{Cl}(\mathcal{X}_\Sigma)$ whose image in $\operatorname{Cl}(\mathcal{X}_\Sigma)_\Q$ are lattice points of $Z_{\mathcal{X}_\Sigma}$.
    The proof of the lemma thus reduces to showing that each lattice point has exactly $|\operatorname{Tors}(\operatorname{Cl}(\mathcal{X}_\Sigma))|$ corresponding elements in $\operatorname{Cl}(\mathcal{X}_\Sigma)$ mapping to it. 

    \noindent The existence of some divisor mapping to any given lattice point is evident by recalling that the lattice is defined via the image of $\bigoplus_{\rho\in \Sigma(1)}\Z\cdot D_\rho=:\Z^{\Sigma(1)}$. 
    
    \noindent Suppose that two divisors $D_1, D_2\in \operatorname{Cl}(\mathcal{X}_\Sigma)$ map to the same lattice point. We aim to show this happens if and only if $D_1-D_2$ is torsion. Let $D'_i$ be an element of $\Z^{\Sigma(1)}$ mapping to $D_i$. Consider the exact sequence obtained by applying $-\otimes_\Z\Q$ to \eqref{eqn:Coxseqpre}.
    \[
    0\rightarrow M_\Q\rightarrow \Q^{\Sigma(1)}\rightarrow \operatorname{Cl}(\mathcal{X}_\Sigma)_\Q\rightarrow 0.
    \]
    Two divisors $D'_1, D'_2$ mapping to the same lattice point is equivalent to $D'_1-D'_2$ being in the kernel of the map $\Q^{\Sigma(1)}\rightarrow \operatorname{Cl}(\mathcal{X}_\Sigma)_\Q$. By exactness of the sequence, this happens if and only if $D'_1-D'_2=\sum_{\rho\in \Sigma(1)} \langle m_q,u_\rho\rangle D_\rho$ for some $m_q\in M_\Q$. But then there exists some $p\in \Z$ such that $p\cdot m_q\in M$ and $p(D'_1-D'_2)=\sum_{\rho\in \Sigma(1)}\langle p\cdot m_q,u_\rho\rangle D_\rho\in \ker(\Z^{\Sigma(1)}\rightarrow \operatorname{Cl}(\mathcal{X}_\Sigma))$, where the map comes from the original exact sequence \eqref{eqn:Coxseqpre}. Hence, $p(D'_1-D'_2)\sim 0$ which happens precisely if $D_1-D_2$ is torsion.
\end{proof}

For $\rho\in \Sigma(1)$, denote $f_{\sigma}(D_\rho)=:\beta_\rho$. The collection $\{\beta_\rho\mid \rho\in \Sigma(1)\}$ gives a set of ray generators of $\Sigma_{GKZ}$, which may appear with multiplicity (as the fan is a generalised fan). The maximal cones of the fan $\Sigma_{GKZ}$ have dimension $|\sigma(1)|-\dim \sigma$.

Using the same considerations as in the above proof of Lemma \ref{Lem:Number of conic modules}, we also obtain the following result, which will simplify the combinatorics of conic modules later on. 
\begin{lemma}
    \label{Lem:TorsionInvarianceMQ}
    Consider $v, v'\in M_\R$ such that $-d(v)-(-d(v'))$ is a torsion divisor. This happens if and only if $f_{\sigma}(d(v')-d(v))=0$. Then $\Delta_{v'}=\Delta_{v}+m_q$ for some $m_q\in M_\Q$. Furthermore, $m_q\in \frac{1}{|\operatorname{Tors}(\operatorname{Cl}(\mathcal{X}_\Sigma))|}M$.
\end{lemma}
\begin{proof}
    The first part of the lemma was shown above. Given that $f_{\sigma}(d(v')-d(v))=0$, we note that, via the same exact sequence as before (obtained by tensoring \eqref{eqn:Coxseqpre} by $\Q$), we have $d(v')-d(v)=\operatorname{div}(x^{m_q})$ for some $m_q\in M_\Q$. Hence $d(v')_\rho=d(v)_\rho+\langle m_q,u_\rho\rangle$ for all $\rho\in \sigma(1)=\Sigma(1)$. Furthermore, since both $d(v)_\rho, d(v')_\rho$ are integer for all rays $\rho$, we have $\langle m_q,u_\rho\rangle\in \Z$ as well. But then the defining inequalities of $\Delta_v$ are the same as the $\Delta_{v'}$, shifted by the terms $\langle m_q,u_\rho\rangle$. Hence $\Delta_{v'}=\Delta_{v}+m_q$, as required.
    For the last observation, let $k$ be minimal such that $k\cdot m_q\in M$. Then $k$ is minimal such that $k\cdot f_{\sigma}(d(v')-d(v))=0$. Thus, $k$ divides $\vert \operatorname{Tors}(\operatorname{Cl}(\mathcal{X}_\Sigma))|$. 
\end{proof}

\begin{lemma}
    \label{Lem:FacetCondition}
    Let $\sigma$ be a cone and let $v,w\in M_\R$.
    The isomorphism class of the conic module $A_w$ appears in $K_v^\bullet$ if and only if there is some subset $J\subset \sigma(1)$ such that $-d(w)\sim -d(v)-\sum_{\rho\in J}D_\rho$ and there is a non-empty set  of solutions $x\in M_\R$ to the system of (in)equalities.

        \begin{equation}
            \label{eq:FacetCondition}
            \begin{cases} d(w)_\rho-1 <\langle x,u_\rho\rangle <d(w)_\rho \text{ for }\rho \not\in J,\\
             d(w)_\rho-1 =\langle x,u_\rho\rangle\text{ for } \rho \in J.\end{cases}
        \end{equation}

    Furthermore, if $A_w$ appears in $K_v^\bullet$, then it does so in all degrees $l$ such that there is a subset $J$ with $\dim\operatorname{Span}(u_\rho\vert\rho\in J)=l$.
\end{lemma}

\begin{proof}
For any ray $\rho\in \sigma(1)$, denote by $u_\rho$ its primitive generator.
    $A_w$ appears in $K_v^\bullet$ if and only if it is isomorphic to $\opA_{\tau}$ for some open cell $\tau$ of the chamber of constancy $\Delta_v$. Note that $-d(v)_\rho=\lfloor \langle -v,u_\rho\rangle\rfloor=-\lceil\langle v,u_\rho\rangle\rceil$.

    Any open cell of $\Delta_v$ can be described (cf. \cite[Proposition 4.6]{FMS19}) as set $\tau_J$ of elements
    $x\in M_\R$ for which there exists $J\subset\sigma(1)$ and \begin{equation}
    \label{eq:FacetConditionPf}
    \begin{cases}
          d(v)_\rho-1 <\langle x, u_\rho\rangle <d(v)_\rho \text{ for } u_\rho\not\in J, \\
          \langle x,u_\rho\rangle = d(v)_\rho\text{ for } u_\rho\in J.
    \end{cases}
    \end{equation}
    Note that a given subset $J\subset \sigma(1)$ gives an open cell if and only if this set is non-empty. 

    Pick a point $y\in \operatorname{Int}(\sigma^\vee)$ and a sufficiently small, positive $\varepsilon$. Consider the divisor $D=\sum_{\rho\in \sigma(1)} d_{J,\rho} D_\rho$, where $d_{J,\rho}=\begin{cases}
        d(v)_\rho\text{ for }\rho\not\in J,\\
        d(v)_\rho +1\text{ for } \rho\in J.
    \end{cases}$
    
\noindent If there is an $x\in M_\R$ fulfilling the above restrictions, then for $z=x+\varepsilon \cdot y$  we obtain\begin{equation}
    \label{eq:CellEquivShift}
        d_{J,\rho} -1 <\langle z,u_\rho\rangle \le d_{J,\rho}\text{ for }\rho\in \sigma(1).
    \end{equation}
    Given a subset $J\subset \sigma(1)$, it thus defines an open cell of $\Delta_v$ if and only if the inequalities \eqref{eq:CellEquivShift} define a non-empty set. Furthermore, all open cells of $\Delta_v$ arise this way. Observe now that given an element $w'\in M_\R$ that satisfies the inequalities \eqref{eq:CellEquivShift}, the inequalities define its chamber of constancy $\Delta_{w'}$ (indeed, we necessarily have $d_{J,\rho}=\lceil\langle w',u_\rho\rangle \rceil$). Given an element $x\in \tau_J$, we thus note that $z=x+\varepsilon\cdot y$ lies in $\Delta_{w'}$. This small perturbation by $\varepsilon\cdot y$ is precisely the perturbation that assigns to the open conic module $\opA_{x}$ a conic module $A_{z}$, and thus $\opA_{x}=A_z=A_{w'}$ (see Proposition \ref{Prop:OpenconicIsNearby}).

    So far, we have shown that conic modules appearing in $K_v^\bullet$ satisfy the inequalities \eqref{eq:CellEquivShift}. Given an element $w\in M_\R$, $A_w$ appears in $K_v^\bullet$ if and only if $d(w)_\rho=d_{J,\rho} \forall \rho\in \sigma(1)$, for some subset $J\subset \sigma(1)$ with non-empty open cells $\tau_J$. We consider conic modules up to isomorphism, thus the condition $d(w)_\rho=d_{J,\rho}$ is precisely equivalent to $-d(w)\sim-d(v)-\sum_{\rho\in J}D_\rho$.

    Observe that given such a $w\in M_\R$ with $d(w)_\rho=d_{J,\rho} \forall \rho\in \sigma(1)$ for some $J\subset \sigma(1)$, the open cell $\tau_J$ is non-empty if and only if the system of (in)equalities \eqref{eq:FacetConditionPf} holds, which is equivalent to \eqref{eq:FacetCondition}.

    The degree of the appearance of $A_w$ in $K_v^\bullet$ is the codimension of the open cell, which corresponds to the number of linearly independent hyperplanes where the inequalities become strict, i.e. the dimension of $\operatorname{Span}(u_\rho\mid \rho\in J)$. 
\end{proof}

\begin{remark}
    \label{Rem:FacetComputation}
    On the computational side, to find all isomorphism classes $A_w$ possibly appearing in $K_v^\bullet$, one checks all divisors of the form $-D=-d(v)-\sum_{\rho\in J}D_\rho$ to establish when there is a non-empty set of solutions to 
    \[
        \begin{cases} d(v)_\rho-1 <\langle x,u_\rho\rangle <d(v)_\rho \text{ for }\rho \not\in J,\\
        d(v)_\rho =\langle x,u_\rho\rangle \text{ for } \rho \in J.\end{cases}
    \]
    Once the set of divisors is established, we can pick a representative of each conic module and check which one the divisors are linearly equivalent to.
\end{remark}

It is tempting to assume that $-d(w)\sim -d(v)-\sum_{\rho\in J}D_\rho$ is a sufficient condition for $A_w$ to appear in $K_v^\bullet$, as then one could also determine the codimension of the facet.
Unfortunately, this is false; the system of inequalities \eqref{eq:FacetCondition} may not have a solution. Indeed, consider Example \ref{Exa:710FMS} again, as we did after Proposition \ref{Prop:BTisConic}.

\begin{example}
    \label{Exa:710Again}
    As before, we consider the fan $\Sigma_1$ with maximal cones \[
    \sigma_1=\cone((1,0,0),(0,1,0),(0,-1,1)),\quad \sigma_2=\cone((0,1,0),(-1,0,1),(0,-1,1)).
    \]
    The toric algebra $R$ has the form $k[x,y,z,w]/(xz-yw)$ and the three conic modules are\begin{eqnarray*}
        && A_{v_0}=R,\\
        && A_{v_+}=(x,y)R,\\
        && A_{v_-}=(x,w)R.
    \end{eqnarray*}
    These three conic modules are associated to the points $v_0=(0,0,0),$ $v_{+}=(0,-\frac{1}{4},0),$ and $v_-=(-\frac{1}{4},0,0)$ respectively. We recall that $-d(v_+)=-D_4$, where $D_4$ is the toric divisor associated to the ray $(0,-1,1)$. Then $-d(v_+)\sim-d(v_+)-(D_1+D_2)$ as $D_1+D_2=\operatorname{div}(x^{(1,1,1)})\sim 0$. However, the system on inequalities becomes\begin{eqnarray*}
        &&x_1=0,\\
        && x_2=0,\\
        && -1< x_3-x_1<0,\\
        &&0 < x_3-x_2<1.
    \end{eqnarray*}
    This has no solutions. Analogously, all other sets $J\neq \emptyset$ with $-d(v_+)\sim -d(v_+)-\sum_{\rho\in J}D_\rho$ admit no solutions to their respective sets of inequalities - which is why $K_{v_+}^\bullet$ has no $A_{v_+}$ appearing outside of degree 0.
\end{example}

We now wish to find a combinatorial condition for the form the complexes $K_v^\bullet$ can take, one which can be verified by knowing the structure of the secondaty fan and the halfopen zonotope $Z_{\mathcal{X}}$. Given two points $P_1, P_2$ corresponding to $A_w$ and $A_v$, necessary and sufficient conditions can be formulated to obtain the appearance of $A_w$ in $K_v^\bullet$, at least up to torsion\footnote{By which we mean that there exists some $m_q\in M_\Q$ such that $d(w+m_q)-d(w)$ is torsion and $A_{w+m_q}$ appears in $K_v^\bullet$.}. The combinatorial formulations require us to introduce the notion of \newterm{(viable) paths}, but first we fix the following notation.
We often work with the following exact sequence, induced by \eqref{eqn:Coxseqpre}:
\begin{equation}\label{eq:CoxR}
0\rightarrow M_\R\xrightarrow{g} \bigoplus_{\rho\in\sigma(1)}\R\cdot D_\rho\xrightarrow{f_{\sigma}} \R^{|\sigma(1)|-n}\rightarrow 0.
\end{equation}
The map $g$ can be represented by a $|\sigma(1)|\times n$ matrix $A$, with rows being the primitive generating vectors $u_\rho$. Similarly, $f_{\sigma}$ is represented by a $(|\sigma(1)|-n)\times |\sigma(1)|$ matrix $C$, whose columns are the vectors $\beta_\rho$. 

Fix the following pieces of notation.

\begin{notation}
    \label{Not:AJandCJ}
     For a subset $J\subset\sigma(1)$, denote by $V_J$ the space $\bigoplus_{\rho\in J}\R\cdot D_\rho$. Denote by $C_J$ the restriction $C\vert_{V_J}$ of $C$ to $V_J$. Similarly, given a vector $x\in \bigoplus_{\rho\in \sigma(1)}\R\cdot D_\rho$, denote by $x_J$ the projection of $x$ to $V_J$. 
     We also write $A_J$ for the sub-matrix of $A$ with rows given by $u_\rho, \rho\in J$.
\end{notation}

\begin{definition}
    \label{Def:PathInZonotope}
    Let $P_1, P_2$ be two lattice points\footnote{In an abuse of notation, we also consider $P_i$ to be the vector from the origin to $P_i$.} inside the zonotope $Z_\mathcal{X}$. We say there is a path of length $k$ from $P_1$ to $P_2$ if there is a subset $J\subset\sigma(1)$ with $\operatorname{rk}(A_J)=k$ such that $P_1+\sum_{\rho\in J} \beta_\rho=P_2$. The set $\{\beta_\rho\mid \rho\in J\}$ is called the \newterm{path from $P_1$ to $P_2$} and will be denoted by $\beta_J$ and is said to \newterm{begin} at $P_1$ and \newterm{end} at $P_2$. We will write $P_1\rightsquigarrow P_2$, and do so to indicate the existence of a viable path even when not specifying the subset $J$.
\end{definition}

Here, we emphasize that the notion of paths is only to be applied to lattice points inside the zonotope: Not every linear combination of $\beta_\rho$ applied to any given lattice point will stay within the zonotope and we shall not consider such non-existent paths. 

\begin{proposition}
    \label{Prop:CplxViaPaths1}
    Let $\sigma$ be a cone and $v,w\in M_\R$. If the isomorphism class $A_w$ appears in $K_v^\bullet$ then there is a path from $f_\sigma(-d(w))$ to $f_\sigma(-d(v))$.
\end{proposition}

\begin{proof}
    By Lemma \ref{Lem:FacetCondition}, if $A_w$ appears in $K_v^\bullet$, $-d(w)\sim -d(v)-\sum_{\rho\in J}D_\rho$ for some $J\subset \sigma(1)$. Applying the map $f_{\sigma}$, we find \[f_\sigma(-d(v))=f_\sigma(-d(w))+\sum_{\rho\in J} f_{\sigma}(D_\rho)=f_\sigma(-d(w))+\sum_{\rho\in J}\beta_\rho.\]
    Thus, there is a path from $f_\sigma(-d(w))$ to $f_\sigma(-d(v))$, as claimed. 

\end{proof}

In our search for lockable subsets, we thus have identified a necessary condition for a module $A_w$ to appear in a complex $K_w^\bullet$;  there needs to be a path from $f_\sigma(-d(w))$ to $f_\sigma(-d(v))$. This is not a sufficient condition, however, as the next example illustrates.
\begin{example}
    \label{Exa:Hexagon}
    Consider the following Gorenstein cone in $\R^4$:
    \[
    \sigma=\cone((1,0,0,1),(0,1,0,1),(0,1,1,1),(0,0,1,1),(1,0,-1,1),(0,0,0,1))=\cone(P\times\{1\}),
    \]
    where lattice polytope $P$ obtained as convex hull \[
    P=\conv((1,0,0),(0,1,0),(0,1,1),(0,0,1),(1,0,-1),(0,0,0)).
    \]
    We consider some simplicial fan $\Sigma$ obtained by regularly triangulation $P$ without adding any additional vertices. The sequence \eqref{eqn:Coxseqpre} becomes\[
    0\rightarrow M\cong \Z^4\rightarrow \Z^6\rightarrow \Z^2\rightarrow 0,
    \] 
    where the first map is given by $\begin{pmatrix}
         1&0&0&1\\
         0&1&0&1\\
         0&1&1&1\\
         0&0&1&1\\
         1&0&-1&1\\
         0&0&0&1
    \end{pmatrix}.$
    The second map, the cokernel, is given by $\begin{pmatrix}
        1&1&-1&0&-1&0\\
        -1&0&0&1&1&-1
    \end{pmatrix}.$
    This gives the collection \[
    \mathcal{B}=\{\beta_\rho\mid \rho\in \sigma(1)\}=\left\{\begin{pmatrix}1\\ -1\end{pmatrix}, \begin{pmatrix}1\\ 0\end{pmatrix}, \begin{pmatrix}-1\\ 0\end{pmatrix}, \begin{pmatrix}0\\ 1\end{pmatrix}, \begin{pmatrix}-1\\ 1\end{pmatrix}, \begin{pmatrix}0\\ -1\end{pmatrix}\right\}.
    \]
    Figure \ref{fig:Hexagon} shows the secondary fan and the zonotope $Z_{\mathcal{X}_\Sigma}$.

\begin{figure}[ht!]
    \centering
\begin{tikzpicture}[scale=1]

\draw[thick] (-3,0) -- (3,0);
\draw[thick] (0,-3) -- (0,3);

\draw[thick] (-3,3) -- (3,-3);

\fill[gray!40, opacity=0.6]
(-2,1) -- (-1,2) -- (1,1) -- (2,-1) -- (1,-2) -- (-1,-1) -- cycle;

\fill (-1,1) circle (2pt) node[above right] {d4};
\fill (0,1) circle (2pt) node[above right] {d6};
\fill (-1,0) circle (2pt) node[above left] {d3};
\fill (0,0) circle (2pt) node[above right] {d0};
\fill (1,0) circle (2pt) node[above right] {d5};
\fill (0,-1) circle (2pt) node[below left] {d2};
\fill (1,-1) circle (2pt) node[below right] {d1};

\end{tikzpicture}
 \caption{Secondary fan and Zonotope for Example \ref{Exa:Hexagon}}
  \label{fig:Hexagon}
\end{figure}
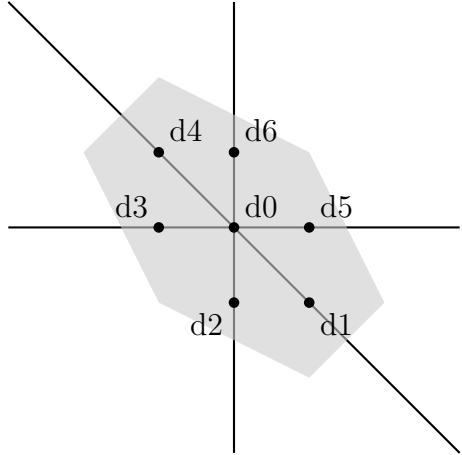
 
    Since the variety has no torsion, each of the points in the zonotope corresponds to a unique conic module and those are the only conic modules for $\sigma$. Denote them by $A_i$, where $A_i$ is associated to the point $d_i$. To compute the complexes associated to the conic modules, we can find for each $d_i$ a vector $v_i\in M_\R$ such that $f_\sigma(-d(v_i))=d_i$, and for each possible subset of rays $J\subset \sigma(1)$ consider the system of (in)equalities \[
    \begin{cases}
        d(v_i)_\rho-1 < \langle x,u_\rho\rangle< d(v_i)_\rho\text{ for }\rho\not\in J,\\
        d(v_i)_\rho=\langle x,u_\rho\rangle\text{ for }\rho\in J.
    \end{cases}
    \] 
    We spare the reader of this tedious process, but point out the chain complex $K_2^\bullet$:
    \[
    K_2^\bullet=A_6\rightarrow A_0\oplus A_4\oplus A_5\rightarrow A_0\oplus A_1\oplus A_3\rightarrow A_2.
    \]

    One can without much effort spot the paths of appropriate lengths that explain the appearance of the corresponding conic modules. However, there is also a path from $d_1$ to $d_2$ of length 3: $d_1+\begin{pmatrix} 1\\-1\end{pmatrix}+\begin{pmatrix} -1\\0\end{pmatrix}+\begin{pmatrix} -1\\1\end{pmatrix}=d_2$. Yet, the conic module $A_1$ does not feature in degree 3 in the complex $K_2^\bullet$, reaffirming that the existence of a path does not guarantee the appearance of the associated conic modules.
\end{example}

A good reason the condition in Proposition \ref{Prop:CplxViaPaths1} is merely necessary and not sufficient is that it combinatorially only expresses half of Lemma \ref{Lem:FacetCondition}; the solubility of the system of (in)equalities \eqref{eq:FacetCondition} is not included. Thus, we aim to find an additional condition to the existence of a path $P_1\rightsquigarrow f_\sigma(-d(v))$ to guarantee the appearance of $A_w$ in $K_v^\bullet$, for some $w$ with $f_\sigma(-d(w))=P_1$. 
To establish such a condition, let us first give an equivalent, yet computationally perhaps more straightforwardly verifiable, formulation to the solubility of \eqref{eq:FacetCondition}, involving the map $f_{\sigma}$.
\begin{notation}
    \label{Not:Matrixsplitting}
    Given a subset $J\subset \sigma(1)$ and a divisor $d=\sum_{\rho\in \sigma(1)} \alpha_\rho D_\rho$, denote by $W_{d,J}$ the open subspace of $\bigoplus_{\rho\in J}\R\cdot D_\rho$ defined by restricting the coefficients to the open intervals $(\alpha_\rho-1,\alpha_\rho)$, i.e. $W_{d,J}=\{\sum_{\rho\in J}\gamma_\rho D_\rho\mid \gamma_\rho\in (\alpha_\rho-1,\alpha_\rho)\}$. If $J=\emptyset$, and so $W_{d,\emptyset}$ would be the empty set $\emptyset$, we instead set $W_{d,\emptyset}=\{0\}$ or simply write $W_{d,\emptyset}=0$.
\end{notation}

\begin{lemma}
    \label{Lem:FacetSufficientMatrix}
    Let $\sigma$ be a cone and let $v,w\in M_\R$. Then the isomorphism class of $A_w$ appears in $K_v^\bullet$ if and only if $-d(w)\sim -d(v)-\sum_{\rho\in J}D_\rho$ for some subset $J\subset \sigma(1)$ and \begin{equation}\label{eq:EquivFacet}
    -C_Jd(v)_{J}\in C_{J^c}W_{d(v),J^c}.
    \end{equation}
    Furthermore, in this case $A_w$ appears in $K_v^\bullet$ in degree $\operatorname{rk}(A_J)$.
\end{lemma}

\begin{proof}
 By Lemma \ref{Lem:FacetCondition}, the isomorphism class of $A_w$ appears in $K_v^\bullet$ if and only if $-d(w)\sim -d(v)-\sum_{\rho\in J}D_\rho$ for some $J\subset \sigma(1)$ and \eqref{Lem:FacetCondition} has a solution $x\in M_\R$. It is therefore sufficient to show that the solubility of \eqref{eq:FacetCondition} is equivalent to \eqref{eq:EquivFacet}.

 Solubility of \eqref{eq:FacetCondition} means that there is an $x\in M_\R$ such that \[Ax=\sum_{\rho\in \sigma(1)}\gamma_\rho D_\rho\text{ with }\begin{cases}
     \gamma_\rho\in (d_\rho-1,d_\rho)\text{ for }\rho\in J^c,\\
     \gamma_\rho=d_\rho\text{ for }\rho \in J.
 \end{cases}\]
 Then $y:=Ax\in f_\sigma(A)=\ker(C)$. Note that $y_{J}=d(v)_J$ and $y_{J^c}\in W_{d(v),J^c}$. As $y\in \ker(C)$, we have\[
 0=Cy=C_J y_J+C_{J^c}y_{J^c}\Leftrightarrow -C_Jy_J=C_{J^c}y_{J^c}.
 \]
 Thus, the system \eqref{eq:FacetCondition} has a solution if and only if $-C_Jd(v)_J=-C_Jy_J\in C_{J^c}W_{d(v),J^c}$, as required.

 To obtain the degree of $A_w$ in $K_v^\bullet$, observe that it coincides with the codimension $\operatorname{codim}(\tau)$, where $\tau$ is the face of the chamber of constancy $\Delta_v$ with $\opA_\tau\cong A_w$. 
 The codimension can be determined by considering the dimension of the space intersecting $\Delta_v$ to fix $\tau$, which is the intersection of the hyperplanes perpendicular to $u_\rho, \rho\in J$, in other words the number of linearly independent equalities in \ref{eq:FacetCondition}. Thus, the degree of $A_w$ in $K_v^\bullet$ is precisely the maximal number of linearly independent $u_\rho, \rho\in J$, i.e. $\operatorname{rk}(A_J)$.
\end{proof}

\begin{remark}
\label{Rem:LinEquivFacet}
Observe here that solubility of both conditions \eqref{eq:FacetCondition} and \eqref{eq:EquivFacet} are invariant under linear equivalence. Two torus-invariant Weil divisors $d_1, d_2$ are linearly equivalent iff there is an $m\in M$ such that $d_1=d_2+\operatorname{div}(x^m)=d_2+\sum_{\rho\in\sigma(1)}\langle m,u_\rho\rangle$. But then a solution $x$ to \eqref{eq:FacetCondition} associated to $d_1$ exists if and only if $x+m$ is a solution to the corresponding set of (in)equalities for $d_2$. Similarly, $f_\sigma(A)=\ker(C)$ by the exactness of the exact sequence \eqref{eq:CoxR} and so $d(v)=\sum_{\rho\in \sigma(1)}\alpha_\rho D_\rho$ satisfies \eqref{eq:EquivFacet} if and only if the linearly equivalent divisor $\sum_{\rho\in \sigma(1)}d_\rho D_\rho$ does. Intuitively this was clear, as we always consider isomorphism classes of conic modules.

\noindent Note further that the same argument works for $-d(v)\sim_\Q -d(v')$. That is, if $-d(v)\sim -d(v')=-d(v)+\sum_{\rho\in \sigma(1)}\langle m_q,u_\rho\rangle$ for some $m_q\in M_\Q$ giving torsion, then the solubility of \eqref{Lem:FacetCondition} for $-d(v)$ is equivalent to solubility for $-d(v')$. Hence $A_{w+m_q}$ appears in $K_{v'}^\bullet$ if and only if $A_w$ appears in $K_v^\bullet$. We will prove this again in a more geometrically intuitive way in Lemma \ref{Lem:TorsionCellsPropagate}. 
\end{remark}

\begin{definition}
    Write $P_2$ as $\sum_{\rho\in \sigma(1)}\alpha_\rho\beta_\rho$ for some $\alpha_\rho\in \Q$. For a given ray $\rho$, consider the interval $I_{D,\rho}=\{\gamma_\rho\beta_\rho\mid \gamma_\rho\in (-\alpha_\rho-1,-\alpha_\rho)\}\subset \R\cdot \beta_\rho$.
    A path $\beta_J$ from $P_1$ to $P_2$ corresponding to a subset $J\subset \sigma(1)$ is said to be \newterm{valid} if \begin{equation}
        \label{eq:ValidPath}
        \sum_{\rho\in J}\alpha_\rho\beta_\rho\in \sum_{\rho\in J^c}I_{D,\rho},
    \end{equation}
    where the right hand side denotes the Minkowski-sum of the intervals $I_{D,\rho}$.
\end{definition}
\begin{remark}
    We consider the Minkowski sum of an empty set of intervals to be 0, and so for $J=\sigma(1)$ the condition becomes $\sum_{\rho\in \sigma(1)}\alpha_\rho\beta_\rho=0$, i.e. $P_2=0$. The proof of Proposition \ref{Prop:CplxViaPaths2} will shed light on the computational side, but the reason to consider the Minkowski sum to be 0 is that we consider the set $W_{d,\emptyset}$ as a space, so it is the zero-space $\{0\}$ and not the set $\emptyset$.
\end{remark}

To simplify the computations of validity for given paths, we note here explicitly that if $\beta_{\rho_1}=\beta_{\rho_2}$, then they are interchangeable without affecting validity of the path. This will often be used implicitly.
\begin{lemma}
    \label{Lem:InterchangBeta}
    Fix a cone $\sigma$ and consider the corresponding collection of vectors $\beta_\rho$, $\rho\in \sigma(1)$. Suppose $\beta_{\rho_1}=\beta_{\rho_2}$ for some $\rho_1\neq \rho_2$. Fix a lattice point $P_2\in Z_{\mathcal{X}}$ and its representation as $\sum_{\rho\in \sigma(1)}\alpha_\rho\beta_\rho$. Let $J$ be a valid path from some lattice point $P_1$ to $P_2$ such that $\rho_1\in J$, $\rho_2\not\in J$. Then $J'=J\setminus\{\rho_1\}\cup \rho_2$ is also a valid path from $P_1$ to $P_2$.
\end{lemma}
\begin{proof}
    Validity of $J$ by definition means\[
    \sum_{\rho\in J\setminus\{\rho_1\}}\alpha_\rho\beta_\rho +\alpha_{\rho_1}\beta_{\rho_1}=\sum_{\rho\in J^c\setminus \{\rho_2\}}\gamma_\rho\beta_\rho +\gamma_{\rho_2}\beta_{\rho_2},
    \]
    where $\gamma_\rho\in (-\alpha_\rho-1,-\alpha_\rho)$ and $\gamma_2\in (-\alpha_{\rho_2}-1,-\alpha_{\rho_2})$, so $\gamma_{\rho_2}=-\alpha_{\rho_2}-\delta_{\rho_2}$ with $\delta_{\rho_2}\in (0,1)$. Rearranging the equation gives\[
    \sum_{\rho\in J\setminus\{\rho_1\}\cup\{\rho_2\}}\alpha_\rho\beta_\rho=\sum_{\rho\in J^c\setminus\{\rho_2\}}\gamma_\rho\beta_\rho -\alpha_1\beta_{\rho_1}-\delta_{\rho_2}\beta_{\rho_2}.
    \]
    Noting $\beta_{\rho_1}=\beta_{\rho_2}$, we obtain\[
    \sum_{\rho\in J'}\alpha_\rho\beta_\rho=\sum_{\rho\in J'^c}\gamma_\rho\beta_\rho,
    \]
    with $\gamma_\rho\in (-\alpha_\rho-1,-\alpha_\rho)$, and so $J'$ is valid.
\end{proof}

\begin{proposition}
    \label{Prop:CplxViaPaths2}
    There is a valid path from a lattice point $P_1$ to $f_\sigma(-d(v))$, if and only if there is a $w\in M_\R$ such that the isomorphism class of $A_w$ appears in $K_v^\bullet$ and $P_1=f_\sigma(-d(w))$. For a given valid path $\beta_J$, this $w\in M_\R$ fulfills $-d(w)\sim -d(v)-\sum_{\rho\in J}D_\rho$ and the degree in which $A_w$ appears in $K_v^\bullet$ is equal to the length of the path $\beta_J$.
\end{proposition}

\begin{proof}
If $A_w$ appears in $K_v^\bullet$, the existence of the path $P_1\rightsquigarrow f_\sigma(-d(v))$ follows by Proposition \ref{Prop:CplxViaPaths1}, and we fix the corresponding subset $J\subset \sigma(1)$. Treat first the case of a proper subset $J\subset\sigma(1)$. Note here that $f_\sigma(-d(v))=\sum_{\rho\in \sigma(1)}-d(v)_\rho D_\rho$.
To check the validity of the path, note that by Lemma \ref{Lem:FacetSufficientMatrix}, $-C_Jd(v)_{J}\in C_{J^c}W_{d(v),J^c}$, and so there exist $\gamma_\rho\in (d(v)_\rho-1,d(v)_\rho)$ for $\rho \in J^c$ such that $-C_Jd(v)_J=\sum_{\rho\in J^c}\gamma_\rho\beta_\rho$.
Recall that $C$ represents the map $f_{\sigma}$ and so we equivalently obtain \[
-\sum_{\rho\in J}d(v)_\rho\beta_\rho=\sum_{\rho\in J^c}\gamma_\rho\beta_\rho\in \left\{\sum_{\rho\in J^c}\delta_\rho\beta_\rho\mid \delta_\rho\in (d(v)_\rho-1,d(v)_\rho)\right\},
\]
which is precisely the condition for the path to be valid.

If $J=\sigma(1)$, we obtain the condition $-Cd(v)=0$, which is equivalent to $-\sum d(v)_\rho\beta_\rho=0$, which is the condition for validity.

Conversely, we note that for any lattice point $P_1$ in the zonotope, there exists $w'\in M_\R$ with $f_\sigma(-d(w'))=P_1$. Then $-d(w')\sim_\Q-d(v)-\sum_{\rho\in J}D_\rho$ and so there is an element $w$ with $d(w)-d(w')\sim_\Q 0$ such that $-d(w)\sim -d(v)-\sum_{\rho\in J}D_\rho$. Validity of the path gives 
\[
-C_Jd(v)_J\in C_{J^c}W_{d(v),J^c},
\]
and so by Lemma \ref{Lem:FacetSufficientMatrix} $A_w$ appears in $K_v^\bullet$, in the predicted degree. 
\end{proof}

Recall that in the proof of Lemma \ref{Lem:Number of conic modules}, we showed that $f_\sigma(-d(w))=f_\sigma(-d(w'))$ if and only if the difference $-d(w)-(-d(w'))$ is torsion. Combining this with Proposition \ref{Prop:CplxViaPaths2}, we obtain the following corollary.

\begin{corollary}
    \label{Cor:NoTorsionSym}
    Suppose $\operatorname{Cl}(\mathcal{X})$ has no torsion. Then the isomorphism class $A_w$ appears in $K_v^\bullet$ if and only if there is a valid path from $f_\sigma(-d(w))$ to $f_\sigma(-d(v))$. The length $l$ of each such valid path corresponds to the degree of one copy of $A_w$ in $K_v^\bullet$.
\end{corollary}

When we reduce to studying the secondary fan, we generally lose information on torsion as we tensor the exact sequence \eqref{eqn:Coxseqpre} by $\R$. However, the appearance of torsion is not as cumbersome as one might expect. Choosing conic modules based on lattice points in $Z_{\mathcal{X}}$ means choosing classes of divisors differing only by torsion, and the structure of the associated complexes respects the additional condition of torsion.

\begin{lemma}
    \label{Lem:TorsionCellsPropagate}
    Let $v_1, v_2\in M_\R$ be such that $-d(v_1)-(-d(v_2))$ is a torsion element $E$, i.e. $-d(v_1)-(-d(v_2))$ is sent to 0 by the map $\R^{\Sigma(1)}\rightarrow \operatorname{Cl}_\R(\mathcal{X}_\Sigma)$ induced by \ref{eqn:Coxseqpre}. Then if there is a codimension $i$ cell $\tau_{1,i}$ of $\Delta_{v_1}$ with $\opA_{\tau_{1,i}}\simeq A_{w_1}$ for some $w_1\in M_\R$, then there exists $w_2\in M_\R$ such that $-d(w_1)-(-d(w_2))\sim E$ and $A_{w_2}\simeq \opA_{\tau_{2,i}}$ for some codimension $i$ cell of $\Delta_{v_2}$.
\end{lemma}
\begin{proof}
    As $-d(v_1)-(-d(v_2))$ is torsion, Lemma \ref{Lem:TorsionInvarianceMQ} (and its proof) implies that $\Delta_{v_2}=\Delta_{v_1}+m_q$ for some $m_q\in M_\Q$, such that $\langle m_q, u_\rho\rangle\in\Z$ for all $\rho\in\sigma(1)=\Sigma(1)$. The open cell $\tau_{i,1}$ corresponds to the intersection of $\Delta_{v_2}$ with a hyperplane defined by making exactly those defining inequalities strict which correspond to the rays $\rho$. Consider the parallel hyperplane, obtained via a shift by $m_q$. It is evident that we obtain an open cell $\tau_{i,2}$ of equal codimension as $\tau_{i,1}$ such that $x\in \tau_{i,1}\Leftrightarrow x+m_q\in \tau_{i,2}$. Thus $\opA_{\tau_{i,2}}\simeq A_{w_1+m_q}$. So we let $w_2=w_1+m_q$ and obtain $-d(w_1)-(-d(w_2))=-d(v_1)-(-d(v_2))=\operatorname{div}(x^{m_q})$.
\end{proof}

\begin{remark}
    Note that if $\sigma$ is simplicial, all $-d(v_i)$ only differ by torsion and so the Lemma \ref{Lem:TorsionCellsPropagate} shows that all complexes $K_i^\bullet$ have the same length -  which is the length of $K_0^\bullet$. 
\end{remark}

The Lemma \ref{Lem:TorsionCellsPropagate} suggests that for arguments of substitution, it is a valid strategy to reduce to the secondary fan as opposed to directly constructing the cell decomposition of $\Delta_v$. Including and excluding conic modules based on which lattice point they correspond to gives all the necessary information for the complexes, and so we may in a certain sense ignore torsion. Making this more precise, we formally introduce notions of lockable and incredulous sets when reducing to the geometry of secondary fans.

\noindent Let $L_\mathcal{X}$ denote the set of lattice points in $Z_\mathcal{X}$. For each $P\in L_\mathcal{X}$, let $A_P=\bigoplus A_v$, where the summation is over all those $v$ such that $f_\sigma(-d(v))=P$. Then for each $P\in L_\mathcal{X}$ we consider the complex $\bigoplus K_v$ where the summation is as above. Using Lemma \ref{Lem:TorsionCellsPropagate}, in each degree we obtain a direct sum of conic modules of the form $A_Q$, $Q\in \mathcal{L}_\mathcal{X}$. Denote by $K_P^\bullet$ the resulting complex for $P\in L_\mathcal{X}$.

\begin{remark}
    If $\operatorname{Cl}(\mathcal{X})$ has no torsion, we can substitute the unique conic module $A_v$ for $A_P$ where $f_\sigma(-d(v))=P$ and complete the complex with the appropriate morphisms, recovering the complexes $K_v^\bullet$.
\end{remark}

Now we can define substitution, lockability and incredulousity as before.

\begin{definition}
    \label{Def:IncPtSet}
    For $P\in L_\mathcal{X}$, given a complex $K^\bullet$ where the module $A_P$ appears, we define \newterm{substitution by $P$}, $S(P,K^\bullet)$ to be the complex obtained by substituting all occurrences of $A_P$ with $K_P^\bullet$. 
    A non-empty set of lattice points $I\subset L_\mathcal{X}$ is \newterm{lockable} if there is a finite sequence of successive substitutions by elements in $I^c$ that, applied to the set $\{K_P^\bullet\vert P\in I\}$, yields a set $K_I^\dagger$ of complexes of modules $A_Q$ with $Q\in I$. Such a set is called incredulous if the complexes $K_P^{\dagger,\bullet}\in K_I^\dagger$ are all of the same length, being $\dim\sigma$.
\end{definition}

Whether a set of lattice points is lockable or incredulous can be determined just using the secondary fan and valid paths therein - any consideration of torsion disappears in this notion. This is of significant computational advantage and we will show that this "torsion-free`` version of lockable and incredulous sets is in fact equivalent in its existence to the previous version. Explicitly, we obtain the following result.

\begin{theorem}
    \label{Thm:Incr definitions agree}
    Let $\sigma$ be a cone with a collection of conic modules $S$. Let $I\subseteq S$ be an incredulous set of conic modules. Then the set of lattice points $I'=\{f_\sigma(-d(v))\mid A_v\in I\}$ is an incredulous set of lattice points. In particular, the following are equivalent.
    \begin{enumerate}
        \item There exists an incredulous set of conic modules.
        \item There exists an incredulous set of lattice points in $Z_\mathcal{X}$.
        \item There exists an NCCR of $\sigma$ of the form $\End(\mathbb{A})$ where $\mathbb{A}$ is a direct sum of conic modules.
    \end{enumerate}
\end{theorem}

\begin{proof}
The first and third item have been shown to be equivalent in Theorem \ref{Thm:NCRiffLock}. Given an incredulous set $I$ of lattice points in $Z_\mathcal{X}$, the collection of conic modules $I'=\{A_v\vert f_\sigma(-d(v))\in I\}$ is clearly itself an incredulous set of conic modules. Indeed, the substitutions performed on the complexes $K^\bullet_P, P\in I$ correspond to substitutions of conic modules not in $I'$ and thus performing these substitutions on the set of complexes $\{K_v^\bullet\mid A_v\in I'\}$ yields a collection of complexes made up of conic modules $A_v\in I'$, which are further of length $\dim \sigma$. 

It remains to show that the existence of an incredulous set $I$ of conic modules gives an incredulous set of lattice points in the zonotope. We claim that the set $I'=\{P\mid P=f_\sigma(-d(v))\text{ for some }A_v\in I\}$ is an incredulous set of lattice points in $Z_\mathcal{X}$. It is clear that lockability of $I'$ is inherited from $I$.

Consider the cone $\sigma_{N'}$ which is the cone $\sigma$ in the saturated sublattice $N'$ of $N$ spanned by $u_\rho, \rho\in\sigma(1)$. In that lattice, the cone $\sigma_{N'}$ admits no torsion but the collection $\{\beta_\rho\mid \rho\in \sigma'(1)\}$ coincides with the collection $\{\beta_\rho\mid \rho\in \sigma(1)\}$. Note the map $M'\rightarrow \bigoplus_{\rho\in \sigma'}\R\cdot D_\rho$ in \eqref{eq:CoxR} has the same image as $M\rightarrow \bigoplus_{\rho\in \sigma} \R\cdot D_\rho$, since the cokernel maps $f_{\sigma}$  and $f_{\sigma'}$ agree. Thus, for any $v\in M_\R$, there is $v'\in M'_\R$ such that $-d(v)=-d(v')$ and in particular they map to the same lattice point. 

\noindent We note that $R'\rightarrow R$ is a finite flat morphism, and so if $L\in \operatorname{CM}R$, $R'\otimes_\R L\in \operatorname{CM}R'$. Further, observe that $R'\otimes_R\Hom(A_v,A_w)=\Hom(A_{v'},A_{w'})$. To see this, use Corollary \ref{Cor:HomBT} to directly compute both sides via the polytopes $Q_{d(v)-d(w)}, Q_{d(v')-d(w')}$ and their respective intersections with $M, M'$. Tensoring $\Hom(A_v,A_w)$ with $R'$ precisely yields the expression for $\Hom(A_{v'},A_{w'})$. Consider now the set of conic modules for $\sigma'$ of the form $A_{v'}$ with $f_\sigma(-d(v'))\in I'$ and form the direct sum $\mathbb{B}$. By the above, $\Lambda'=\End_{R'}(\mathbb{B})\in \operatorname{CM}R'$. The set $I'$ is lockable and so we obtain $\gldim\Lambda'<\infty$, and thus $\Lambda'\in \operatorname{CM} R'$ is an NCCR of $R'$. That means the set of conic modules in the direct sum $\mathbb{B}$ is incredulous by Theorem \ref{Thm:Incr definitions agree}. But each lattice point in $Z_{\mathcal{X},N'}$ has a unique isomorphism class of conic $R'$-modules $A_{v'}$ associated to it, and so the set $I'$ is an incredulous set of lattice points for $R'$. Being an incredulous set of lattice points is a combinatorial condition, so as the collections of $\{\beta_\rho\}$ agree, $I'$ is also incredulous when considered as set of lattice points for $R$, concluding the proof of the Theorem.
\end{proof}

A nice consequence of this theorem is the following.
\begin{corollary}
    \label{Cor:IfNCCRIncludeTorsion}
    Let $\sigma\subset N_\R$ be a cone admitting and NCCR $\End(\mathbb{A})$ for an incomplete sum of conic modules $\bigoplus_{v\in V}A_v$ for some set $V\subset \R^n$. For each $v\in V$, denote by $T_v$ the set of isomorphism classes of conic modules $\{A_w\mid d(w)-d(v)\sim_{\Q}0\}$, i.e. those conic modules whose divisors $-d(w)$ differ from $-d(v)$ by torsion. Then $\End(\mathbb{B})$ is an NCCR of $R_\sigma=k[\sigma^\vee\cap M]$, where \[\mathbb{B}=\bigoplus_{w\in \bigcup_{v\in V}T_v}A_w.\]
\end{corollary}

\begin{corollary}
\label{Cor:2ndFanDetermines}
Suppose $\sigma, \sigma'\subset N_\R$ are two cones such that the collections $\{\beta_\rho\mid \rho\in \sigma(1)\}$ and $\{\beta_\rho\mid \rho\in \sigma'(1)\}$ agree. Then there is an NCCR  $\End_{R_\sigma}(\mathbb{A})$ of $R_\sigma=k[\sigma^\vee\cap M]$ via an incomplete sum of conic modules $\mathbb{A}$ iff there is an NCCR $\End_{R_{\sigma'}}(\mathbb{B}')$ via an incomplete sum of conic modules of $R_{\sigma'}=k[(\sigma')^\vee\cap M]$.
\end{corollary}

Let us now reexamine our motivating example, Example \ref{Exa:710FMS}, using the Proposition \ref{Prop:CplxViaPaths2}.

\begin{example}
Recall the cone \[\sigma=\cone((1,0,0),(0,1,0),(-1,0,1),(0,-1,1)).\] We established that the secondary fan has the 4 generators $\beta_1=\beta_3=+1,$ $\beta_2=\beta_4=-1$. Use now Proposition \ref{Prop:CplxViaPaths2} to obtain the complexes $K_0, K_{+}, K_{-}$ associated to the conic modules $A_0, A_{v_+}=A_1$ and $A_{v_-}=A_2$ respectively. Note that $A_{1}$ corresponds to the lattice point $+1$ and $A_{2}$ to $-1$, while $A_0$ corresponds to $0$ inside the zonotope $(-2,2)$.
Since there is no torsion, each lattice point has a unique conic module associated to it, so via Corollary \ref{Cor:NoTorsionSym}, constructing the complexes simplifies to finding valid paths.

We start by considering the complex $K_0^\bullet$. Firstly, we write $0=\sum_{\rho\in \sigma(1)}0\cdot \beta_\rho$ and so the validity condition for a path $\beta_J$ becomes $0\in \sum_{\rho\in J^c}(-1,0)\cdot\beta_\rho$, where $(-1,0)\cdot\beta_\rho$ is a shorthand to denote the set $\{\gamma\cdot \beta_\rho\mid \gamma\in (-1,0)\}$. The only valid paths from $-1$ to $0$ are $\beta_1$ or $\beta_3$, as for either of those we indeed have $0\in (-1,0)+(0,1)+(-1,0)=(-2,1)$. The only other way to obtain a path from $-1$ to $0$ is $\beta_1+\beta_3+\beta_2$ (or $\beta_4$ as last summand). But $0\not\in (-1,0)$ and so such a path is not valid. Hence $A_{2}$ appears exactly twice in degree 1 (as there is two valid paths). Similarly, $A_{1}$ appears twice in degree 1. A path from $0$ to $0$ is either given as $-1+1$, which is valid as $0\in (-1,0)+(0,1)=(-1,1)$, via the maximal length path $\beta_1+\beta_2+\beta_3+\beta_4$, which is valid as $0\in W_{0,\emptyset}=\{0\}$, or via the empty path, which is always valid. 
Thus, $A_0$ appears $2\times 2=4$ times in degree 2 and once in degree $\dim(\operatorname{Span}\{u_\rho\mid \rho\in \sigma\})=3$. This yields the complex \[
K_0^\bullet= A_0\rightarrow A_0^{\oplus 4}\rightarrow A_{2}^{\oplus 2}\oplus A_{1}^{\oplus2}\rightarrow A_0.
\]

Now consider the complex $K_{1}^\bullet$. The lattice point that is the end of the paths we record is $+1$, which we choose to write as $1\cdot\beta_1+\sum_{i=2}^40\cdot\beta_i$. From $-1$, there can only be one path to $+1$: $\beta_1+\beta_3$. This path is valid as $1\in (0,1)+(0,1)=(0,2)$. From $0$, the possible paths to $+1$ are $\beta_1$ and $\beta_3$ of length 1 and $\beta_1+\beta_3+\beta_2$ or $\beta_1+\beta_3+\beta_4$ of length 3. Whilst the first two are valid as $1\in (-1,0)+(0,1)+(0,1)=(-1,2)$ and $0\in(-2,-1)+(0,1)+(0,1)=(-2,1)$, the latter two are not valid since $1\not\in (0,1)$. Finally, consider the possible paths from $+1$ to itself. The empty path is valid, and the other options are $\beta_1+\beta_2$, $\beta_1+\beta_4$, $\beta_3+\beta_2$, $\beta_3+\beta_4$ and $\beta_1+\beta_2+\beta_3+\beta_4$. The length two paths are not valid as $1\not\in (-1,0)+(0,1)=(-1,1),$ $0\not\in (-2,-1)+(0,1)=(-2,0)$ and the length 4 path is not valid as $1\neq 0$. Hence, the full complex is\[
K_{1}^\bullet=A_{2}\rightarrow A_0^{\oplus 2}\rightarrow A_{1}.
\]
By symmetry, \[
K_{2}^\bullet=A_{1}\rightarrow A_0^{\oplus 2}\rightarrow A_2.
\]
This recovers the results of Faber-Muller-Smith, detailed in Example \ref{Exa:710FMS}.
\end{example}

The secondary fan and zonotope in the above example are particularly simple: the fan is one-dimensional and the zonotope is the interval $(-2,2)$, so only contains three lattice points.

Reducing to the combinatorics of the secondary fan is of a computational advantage to explicitly computing all chambers of constancy and visualising them to find their adjacencies. Let us now consider a few more examples to illustrate the methodology.

\begin{example}
    \label{Exa:NonAlmlSimInc}
    Consider the four-dimensional Gorenstein cone \[\sigma=\cone((2,1,1,1),(0,-1,-1,1),(2,1,-1,1),(0,-1,1,1),(1,1,-1,1),(1,-1,1,1)).\]
    The map $f_{\sigma}$ can be represented via the matrix $\begin{pmatrix}
         1& 1& 0& 0 &-1 & -1\\
         0 &0 & 1& 1& -1& -1
    \end{pmatrix}.$
    We highlight the presence of torsion in this example: considering an appropriate simplicial subdivision $\Sigma$ of $\sigma$, we find that $\operatorname{Cl}(X_\Sigma)\cong\Z^2\times (\Z/2\Z)^2$. The Lemma \ref{Lem:TorsionCellsPropagate} and Theorem \ref{Thm:Incr definitions agree} allow us to simplify the computations as if no torsion was present, and reduce to the studying the secondary fan.
    
    The generators (each with multiplicity 2) of the secondary fan are thus $\beta_1=\beta_2=\begin{pmatrix}1\\ 0\end{pmatrix}$, $\beta_3=\beta_4=\begin{pmatrix}0\\ 1\end{pmatrix}$ and $\beta_5=\beta_6=\begin{pmatrix}    -1\\-1\end{pmatrix}$.
    The Table \ref{Tab:PathsInRevEngineer} lists all possible paths by type (i.e. by the sum $\sum_{\rho\in J}\beta_\rho$), including the length and the number of such paths. In each row, we include an example of how to obtain said path type.

\begin{table}[ht!]
\begin{center}
\begin{tabular}{ |c|c|c|c||c|c|c|c| } 
\hline
   \text{Type} & \text{Example} & \text{Length}& \text{Number} &\text{Type} & \text{Example}& \text{Length}& \text{Number}\\
\hline
$\begin{pmatrix}0\\0\end{pmatrix}$ & $\emptyset$ & 0 & 1 &$\begin{pmatrix}-1\\0\end{pmatrix}$ & $\beta_3+\beta_5$ & 2 & 4\\
\hline
   & $\beta_1+\beta_3+\beta_5$ & 3 & 8&$\begin{pmatrix}2\\1\end{pmatrix}$ & $\beta_1+\beta_2+\beta_3$ & 3& 2\\
   \hline
   & $\beta_1+\dots+\beta_6$ &  4& 1&$\begin{pmatrix}1\\-1\end{pmatrix}$ & $\beta_1+\beta_2+\beta_5$ & 3 & 2\\
\hline
$\begin{pmatrix} 1\\0\end{pmatrix}$ & $\beta_1$ & 1 & 2 &$\begin{pmatrix}1\\2\end{pmatrix}$ & $\beta_1+\beta_3+\beta_4$ & 3 &2\\
\hline
& $\beta_1+\beta_2+\beta_3+\beta_5$ & 4 & 4 &
$\begin{pmatrix} -1\\1\end{pmatrix}$ & $\beta_3+\beta_4+\beta_5$ & 3 & 2\\
\hline
$\begin{pmatrix}0\\1\end{pmatrix}$ & $\beta_3$ & 1 & 2 &$\begin{pmatrix}-1\\-2\end{pmatrix}$ & $\beta_1+\beta_5+\beta_6$ &  3& 2\\
\hline
& $\beta_1+\beta_3+\beta_4+\beta_5$ &4 & 4 &$\begin{pmatrix}-2\\-1\end{pmatrix}$ & $\beta_1+\beta_5+\beta_6$ & 3 & 2\\
\hline
$\begin{pmatrix}-1\\-1\end{pmatrix}$ & $\beta_5$ & 1 &2 & $\begin{pmatrix}-2\\0\end{pmatrix}$ & $\beta_3+\beta_4+\beta_5+\beta_6$ &  3& 1\\
\hline
& $\beta_1+\beta_3+\beta_5+\beta_6$ & 4& 4 &$\begin{pmatrix}0\\-2\end{pmatrix}$ & $\beta_1+\beta_2+\beta_5+\beta_6$ & 3& 1\\
\hline
$\begin{pmatrix}2\\0\end{pmatrix}$ & $\beta_1+\beta_2$ & 2 &1 &$\begin{pmatrix}2\\2\end{pmatrix}$ & $\beta_1+\beta_2+\beta_3+\beta_4$ & 3& 1\\
\hline
$\begin{pmatrix}0\\2\end{pmatrix}$ & $\beta_3+\beta_4$ & 2 &1& & & &\\
\hline
$\begin{pmatrix}-2\\-2\end{pmatrix}$ & $\beta_5+\beta_6$ & 2 &1& & & &\\
\hline
$\begin{pmatrix}1\\1\end{pmatrix}$ & $\beta_1+\beta_3$ & 2 &4& & & &\\
\hline
 & $\beta_1+\beta_2+\beta_3+\beta_4+\beta_5$ & 4 &2& & & &\\
\hline
$\begin{pmatrix}0\\-1\end{pmatrix}$ & $\beta_1+\beta_5$ & 2 &4& & & & \\
\hline
& $\beta_1+\beta_2+\beta_3+\beta_5+\beta_6$ & 4 & 2& & & & \\
\hline
\end{tabular}
\end{center}
\caption{Possible Paths}
\label{Tab:PathsInRevEngineer}
\end{table}

\begin{figure}
\begin{tikzpicture}[scale=1.2]

\draw[->] (-2.5,0) -- (2.5,0) ;
\draw[->] (0,-2.5) -- (0,2.5) ;

\fill[gray!40, opacity=0.6]
(-2,-2) -- (0,-2) -- (2,0) -- (2,2) -- (0,2) -- (-2,0) -- cycle;

\fill (0,0) circle (1.5pt) node[below right] {$P_0$};
\fill (1,0) circle (1.5pt) node[below] {$P_1$};
\fill (1,1) circle (1.5pt) node[above right] {$P_2$};
\fill (0,1) circle (1.5pt) node[left] {$P_3$};
\fill (-1,0) circle (1.5pt) node[below] {$P_4$};
\fill (-1,-1) circle (1.5pt) node[left] {$P_5$};
\fill (0,-1) circle (1.5pt) node[right] {$P_6$};

\end{tikzpicture}
\label{fig:2Dexample}
\caption{The zonotope and interior lattice points}
\end{figure}
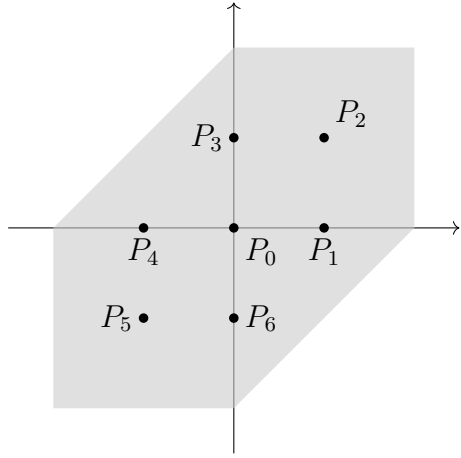

\begin{align*}
    A_0\rightarrow A_0^{\oplus 8}\rightarrow A_1^{\oplus 4}\oplus A_3^{\oplus 4}\oplus A_5^{\oplus 4}\rightarrow A_2^{\oplus 2}\oplus A_4^{\oplus 2}\oplus A_6^{\oplus 2}\rightarrow A_0,&\\
    A_4\rightarrow A_0^{\oplus 2}\rightarrow A_1,&\\
    A_5\rightarrow A_0^{\oplus 4}\rightarrow A_1^{\oplus 2}\oplus A_3^{\oplus 2}\rightarrow A_2,&\\
    A_3^{\oplus 8}\rightarrow A_6\rightarrow A_0^{\oplus 2}\rightarrow A_3,&\\
    A_1\rightarrow A_0^{\oplus 4}\rightarrow A_3^{\oplus 2}\oplus A_5^{\oplus 2}\rightarrow A_4,&\\
    A_2\rightarrow A_6^{\oplus 4}\rightarrow A_5,&\\
    A_3\rightarrow A_0^{\oplus 4}\rightarrow A_1^{\oplus 2}\rightarrow A_6.&\;
\end{align*}

Note now that the set $\{0,1,3,5\}$ is incredulous. The set of complexes after substitution is\begin{align*}
    A_0\rightarrow A_0^{\oplus 8}\oplus A_1^{\oplus 2}\oplus A_3^{\oplus 2}\oplus A_5^{\oplus 2}\rightarrow A_0^{\oplus 24}\oplus A_1^{\oplus 4}\oplus A_3^{\oplus 4}\oplus A_5^{\oplus 4}\rightarrow A_1^{\oplus 8}\oplus A_3^{\oplus 8}\oplus A_5^{\oplus 4}\rightarrow A_0,&\\
    A_1\rightarrow A_0^{\oplus 4}\rightarrow A_3^{\oplus 2}\oplus A_5^{\oplus 2}\rightarrow A_0^{\oplus 2}\rightarrow A_1,&\\
    A_3\rightarrow A_0^{\oplus 4}\oplus A_3^{\oplus 8}\rightarrow A_1^{\oplus 2}\rightarrow A_0^{\oplus 2}\rightarrow A_3,&\\
    A_5\rightarrow A_3^{\oplus 4}\oplus A_0^{\oplus 4}\rightarrow A_0^{\oplus 16}\oplus A_1^{\oplus 2}\oplus A_3^{\oplus 2}\rightarrow A_1^{\oplus 8}\rightarrow A_5.&\;
\end{align*}

Hence, the ring $R_\sigma=k[\sigma^\vee\cap M]$ has an NCCR by $\End_R(\mathbb{B})$, where $\mathbb{B}$ is the direct sum of all conic modules $A_v$ (one direct summand for each isomorphism class) with $f_\sigma(-d(v))\in \{P_0, P_1, P_3, P_5\}$. 
\end{example}

\begin{remark}
In \cite{HM22}, the authors provide NCCRs for edge rings of complete multipartite graphs, reducing to two explicit cases not otherwise covered by previous literature: The edge rings of $K_{2,2,2}$ and $K_{1,1,1,1}=K_4$, respectively. When considering the toric cone associated to the edge ring of $K_4$, one in fact recovers the situation of Example \ref{Exa:NonAlmlSimInc}. However, the other Gorenstein cases of edge rings are computationally much more challenging to construct using conic modules ($K_{2,2,2}$ comes with 167 lattice points inside a 6 dimensional polytope).
\end{remark}

\begin{example}
    Let us once more examine Example \ref{Exa:Hexagon}. Here, the set of generators for the secondary fan is \[
    \mathcal{B}=\{\beta_\rho\mid \rho\in \sigma(1)\}=\left\{\begin{pmatrix}1\\ -1\end{pmatrix}, \begin{pmatrix}1\\ 0\end{pmatrix}, \begin{pmatrix}-1\\ 0\end{pmatrix}, \begin{pmatrix}0\\ 1\end{pmatrix}, \begin{pmatrix}-1\\ 1\end{pmatrix}, \begin{pmatrix}0\\ -1\end{pmatrix}\right\}.
    \]
    When we last considered this example, we mentioned that the path from $d_1$ to $d_2$ given by $\begin{pmatrix}1\\ -1\end{pmatrix}+\begin{pmatrix} -1\\0\end{pmatrix}+\begin{pmatrix}-1\\1\end{pmatrix}$ did not give the module $A_1$ in degree 3 in $K_2^\bullet$. Indeed, we can now see that this path is not valid: represent $d_2$ via $\sum_{\rho\in \sigma(1)}\alpha_\rho\beta_\rho$ with $\alpha_\rho=1$ exactly if $\beta_\rho=\begin{pmatrix} 0\\ -1\end{pmatrix}$ and 0 otherwise. Then validity of the path amounts to the statement\[
    \begin{pmatrix}-1\\0\end{pmatrix}\in (-1,0)\cdot \begin{pmatrix}0\\1\end{pmatrix}+(-1,0)\cdot \begin{pmatrix}1\\0\end{pmatrix}+(-1,0)\cdot \begin{pmatrix}0\\-1\end{pmatrix}.
    \]
    Since $-1\not\in (-1,0)$, this cannot hold and as such the path cannot be valid.
    Let us now examine the complete set of complexes of conic modules for this example.\begin{eqnarray*}
        && K_0^\bullet= A_0^{\oplus4}\rightarrow A_0^{\oplus2}\rightarrow A_0^{\oplus3}\oplus A_1\oplus \dots\oplus A_6\rightarrow A_1\oplus\dots\oplus A_6\rightarrow A_0,\\
        && K_1^\bullet=A_4\rightarrow A_1^{\oplus2}\oplus A_3\oplus A_6\rightarrow A_0\oplus A_2\oplus A_5\rightarrow A_1,\\
        && K^\bullet_2=A_6\rightarrow A_0\oplus A_4\oplus A_5\rightarrow A_0\oplus A_1\oplus A_3\rightarrow A_2,\\
        && K_3^\bullet= A_5\rightarrow A_0\oplus A_1\oplus A_2\oplus A_6\rightarrow A_0\oplus A_4\rightarrow A_4,\\
        && K_4^\bullet= A_1\rightarrow A_2\oplus A_4^{\oplus2}\oplus A_5\rightarrow A_0\oplus A_3\oplus A_6\rightarrow A_4,\\
        && K_5^\bullet=A_3\rightarrow A_0\oplus A_2\oplus A_4\oplus A_6\rightarrow A_0\oplus A_1\rightarrow A_5,\\
        && K_6^\bullet =A_2\rightarrow A_0\oplus A_1\oplus A_3\rightarrow A_4\oplus A_5\oplus A_6.\;
    \end{eqnarray*}

    It is straightforward to see that there is no incredulous set. Indeed, the conic modules come in pairs $A_1\leftrightarrow A_4$, $A_2\leftrightarrow A_6$ and $A_3\leftrightarrow A_5$ and one cannot substitute by both members of the same set (as substitution by both results in a loop - the first one to be substituted reappears). Also, if only one of either pair is substituted, the length of the complex corresponding to the other module in the pair becomes longer than $\dim \sigma=4$. Hence, one can neither substitute 1 nor 2 of either pair. As such, we cannot substitute any of the 6 modules $A_1,\dots, A_6$. Finally, note that $A_0$ appears in degrees $>0$ in $K_0^\bullet$ so substitution of $A_0$ is not useful either, and no incredulous set exists.
\end{example}

\noindent By way of counterexample, we obtain a first partial answer to Question \ref{Qn:IncompleteConMod}. \begin{theorem}
    \label{Thm:Not always exist incomplete NCCR}
    Not every Gorenstein cone admits an incomplete sum of conic modules $\mathbb{A}$ such that $\End\mathbb{A}$ is an NCCR of $R=k[\sigma^\vee\cap M]$.
\end{theorem}

\section{Gorenstein cones}\label{sec:Gor}
For the remainder of the paper, we shall focus on the case of Gorenstein cones, at the heart of Conjecture \ref{Conj:affineGor}.

The combinatorics in the setting of Gorenstein cones gives rise to some symmetry that is going to be helpful. Consider the Gorenstein element $\mathfrak{m}_\sigma$ of the cone $\sigma$. Then \[
0\sim \operatorname{div}(x^{\mathfrak{m}_\sigma})=\sum_{\rho\in \sigma(1)}\langle \mathfrak{m}_\sigma,u_\rho\rangle D_\rho=\sum_{\rho\in\sigma(1)}D_\rho,
\]
where the last equality follows by the definition of the Gorenstein element. This shows that for all subsets $J\subset \sigma(1)$, we have\begin{equation}
    \label{eq:SymmetryGorJDIV}
    \sum_{\rho\in J}D_\rho\sim -\sum_{\rho\in J^c}D_\rho,
\end{equation}
where $J^c$ is the complement of $J$ in $\sigma(1)$. Similarly, on the level of the $\beta_\rho$, we obtain\begin{equation}
\label{eq:GorSymBeta}
    \sum_{\rho\in J}\beta_\rho=-\sum_{\rho\in J^c}\beta_\rho.
\end{equation}

\begin{lemma}
    \label{Lem:ZonoIsSymmetric}
    Consider a Gorenstein cone $\sigma$ with an associated simplicial fan $\Sigma$ (obtained by regularly subdividing the cone without adding additional rays), giving a smooth toric DM stack $\mathcal{X}$. Then the zonotope $Z_{\mathcal{X}}$ is symmetric, i.e. $p\in Z_\mathcal{X}\Leftrightarrow -p\in Z_\mathcal{X}$. Thus, for any $v\in M_\R$, there is a $w\in M_\R$ such that $f_\sigma(-d(w))=-f_\sigma(-d(v))$.
\end{lemma}

\begin{proof}
Note that by \eqref{eq:GorSymBeta}, $\sum_{\rho\in\sigma(1)}\beta_\rho=0$.
The zonotope is given by $f_{\sigma}((-1,0]^{|\Sigma(1)|})$, and thus $Z_{\mathcal{X}_\Sigma}=\{\sum_{\rho\in \sigma(1)}\alpha_\rho \beta_\rho\mid \alpha\in (-1,0]\}$. Write $p$ in this form as $\sum_{\rho\in \sigma(1)}\alpha_\rho\beta_\rho$. We distinguish two cases. If none of the $\alpha_\rho$ are zero, then \[
-p=0+\sum_{\rho\in \sigma(1)}-\alpha_\rho\beta_\rho=-\sum_{\rho\in \sigma(1)}\beta_\rho+\sum_{\rho\in \sigma(1)}-\alpha_\rho\beta_\rho=\sum_{\rho\in\sigma(1)}(-1-\alpha_\rho)\beta_\rho.
\]
Since $\alpha_\rho\in (-1,0)$, $-1-\alpha_\rho\in (-1,0)$ for all $\rho$, and so $-p\in Z_{\mathcal{X}_\Sigma}$. Now suppose some $\alpha_\rho$ are 0. Then\[
p=\sum_{\rho\in\sigma(1)}\alpha_\rho\beta_\rho -r\cdot 0=\sum_{\rho\in\sigma(1)}\alpha_\rho\beta_\rho -r\cdot \sum_{\rho\in\sigma(1)}\beta_\rho=\sum_{\rho\in\sigma(1)}(\alpha_\rho -r)\beta_\rho.
\]
This holds for any $r\in \R$, so choose $0<r\ll 1$. Then $\alpha_\rho\in(-1,0)$ for all $\rho\in \sigma(1)$ and we reduce to the first case, obtaining $-p\in Z_{\mathcal{X}_\Sigma}$.
\end{proof}

\noindent From now on, we leave the simplicial fan $\Sigma$ refining a given cone $\sigma$ implicit and denote by $Z_\sigma$ the zonotope. The symmetry of the zonotope reflects in the complexes $K_v^\bullet$. 
Let us via the next lemma, that also functions as a definition, fix a piece of notation regarding conic modules.
\begin{lemma}
    \label{Lem:NegativeModuleExists}
    Let $\sigma$ be a Gorenstein cone and $v\in M_\R$. There exists a $v'\in M_\R$ such that $-d(v')\sim d(v)$. We define\footnote{up to isomorphism} $B_v$ to be the conic module $A_{v'}$. Similarly, denote the complex $K^\bullet_{v'}$ by $L_v^\bullet$.
\end{lemma}

\begin{proof}
    Note $f_\sigma(d(v))=-f_\sigma(-d(v))$. Since the point $f_\sigma(-d(v))$ is in the zonotope $Z_{\sigma}$, by Lemma \ref{Lem:ZonoIsSymmetric}, so is $f_\sigma(d(v))$. Hence, there is $w\in M_\R$ such that $f_\sigma(-d(w))=f_\sigma(d(v))$ as there is a bijection between points in the zonotope and elements of $f_{\sigma}(\Theta_{\mathcal{X}})$ (see Lemma \ref{Lem:Number of conic modules}).

    But then $d(v)-(-d(w))\in \ker f_{\sigma}=\operatorname{Im}(M_\R\rightarrow \bigoplus_{\rho\in \sigma(1)}\R\cdot D_\rho)$. Thus, $d(v)$ and $-d(w)$ differ only by torsion over $\Z$ and there is $m_q\in M_\Q$ and $v'=w+m_q$ such that $d(v)\sim -d(v')$ and $f_\sigma(-d(v'))=-f_\sigma(-d(v))$.
\end{proof}

\noindent The symmetry of lattice points carries over to a symmetry in the associated complexes, simplifying computations.

\begin{proposition}
    \label{Prop:GorSymDoubleswitch}
    Let $v,w\in M_\R$ such that $A_w$ appears in $K_v^\bullet$. Then $B_v$ appears in $L_{w}^\bullet$ in the same degree.
\end{proposition}
\begin{proof}
    By Lemma \ref{Lem:FacetCondition}, $A_w$ appearing in $K_v^\bullet$ is equivalent to $-d(w)\sim -d(v)-\sum_{\rho\in J}D_\rho$ for some subset $J\subset \sigma(1)$ and existence of a solution to \[
            \begin{cases} d(v)_\rho-1 <\langle x,u_\rho\rangle <d(v)_\rho \text{ for }\rho \not\in J,\\
             d(v)_\rho =\langle x,u_\rho\rangle\text{ for } \rho \in J.\end{cases}
        \]
    This is equivalent to $d(v)\sim d(w)-\sum_{\rho\in J}D_\rho$ and the existence of a solution to\[
    \begin{cases}
        & -d(v)_\rho <\langle x,u_\rho\rangle < -d(v)_\rho+1\text{ for }\rho\not\in J,\\
        &-d(v)_\rho =\langle x,u_\rho\rangle \text{ for }\rho\in J.
    \end{cases}
    \]
    Note now that (up to linear equivalence) $d(w)_\rho=\begin{cases}
        & d(v)_\rho\text{ for }\rho\not\in J,\\
        & d(v)_\rho+1 \text{ for }\rho\in J.
    \end{cases}$
    
    \noindent Therefore, we have a solution $x$ to the system of (in)equalities \[
    \begin{cases}
        & -d(w)_\rho <\langle x,u_\rho\rangle <-d(w)_\rho+1\text{ for }\rho\not\in J,\\
        & 1- d(w)_\rho =\langle x,u_\rho\rangle\text{ for }\rho\in J.
    \end{cases}
    \]

    Consider now $y=x-\mathfrak{m}_\sigma$ for such a solution $x$ to obtain a solution to the system of inequalities\[
    \begin{cases}
        & -d(w)_\rho-1<\langle y,u_\rho\rangle <-d(w)_\rho\text{ for }\rho\not\in J,\\
        & -d(w)_\rho=\langle y,u_\rho\rangle\text{ for }\rho\in J.
    \end{cases}
    \]
    By definition, $B_v$ is the conic module associated to $v'$ such that $-d(v')\sim d(v)$ and similarly $L_w^\bullet$ is the complex $K_{w'}$ where $d(w)\sim -d(w')$. Using Lemma \ref{Lem:FacetCondition}, we obtain that $B_v$ appears in $L_w^\bullet$, as required.
\end{proof}

One conic module that naturally takes on a special role in light of the above result is the conic module associated to $0\in M_\R$, as $A_0=B_0$.
\begin{corollary}
    \label{Cor:A0PseudoSym}
    \begin{enumerate}
        \item If $A_0\in K_v^l$, then $B_v\in K_0^l$.
        \item If $A_v\in K_0^l$, then $A_0\in L_v^l$.
    \end{enumerate}
\end{corollary}

It thus makes sense to start our study of conic modules for Gorenstein modules by establishing which conic modules appear in $K_0^\bullet$ and thus which complexes $K_v^\bullet$ the conic modules $A_0$ appears in.

\begin{lemma}
    \label{Lem:Gor0isself}
    $A_0$ appears in degree $\dim \sigma$ in $K_0^\bullet$. 
\end{lemma}
\begin{proof}
    We know that $0\sim \sum_{\rho\in \sigma(1)}D_\rho$ and that $\sum_{\rho\in \sigma(1)}\beta_\rho=0$, and so there is a path of length $|\sigma(1)|$ from 0 to 0. This path is valid as $0=\sum_{\rho\in \sigma(1)}0\cdot \beta_\rho$ and so the validity condition becomes the tautological statement $0=0$. The degree is the rank of the collection of primitive generators of $\rho\in \sigma(1)$ and so $A_0$ appears in degree $\dim\sigma$.
\end{proof}

\begin{remark}
The Lemma \ref{Lem:Gor0isself} together with Lemma \ref{Lem:TorsionCellsPropagate} and Theorem \ref{Thm:NCRiffLock}  reprove that Gorenstein simplicial cones admit an NCCR, a special case of \cite[Proposition 7.5]{FMS19}. Indeed, any $-d(v)\in \Theta_{\mathcal{X}_\Sigma}$ differs from $0$ by torsion, and so Lemma \ref{Lem:TorsionCellsPropagate} implies that all complexes $K_v^\bullet$ have the same length. Since this length is $\dim \sigma$, a complete sum of conic modules $\mathbb{A}$ gives an NCCR.
\end{remark}

\subsection{Almost simplicial Gorenstein cones}

A special case that we treat in more detail is the case of \newterm{almost simplicial cones}.
\begin{definition}
    A (full-dimensional) cone $\sigma\subset N_\R$ is \newterm{almost simplicial} if $|\sigma(1)|=\dim \sigma+1$.
\end{definition}

In the case of almost simplicial cones, the resulting secondary fan and zonotope are one-dimensional, which makes computations much easier. 
The remainder of this paper is dedicated to investigating which toric algebras of almost simplicial Gorenstein cones have NCCRs obtained as endomorphism algebras of incomplete sums of conic modules. We begin by setting up some notation.

\begin{notation}
    \label{Not:AlmSim}
    Let $\sigma$ be an almost simplicial cone. Since the secondary fan we consider is one-dimensional for this case, all lattice points can be represented by integers. Note that each $\beta_\rho$ is either negative, 0 or positive. Given a subset $J\subseteq\sigma(1)$, write $J_-$ for the set $\{\rho\in J^c\vert \beta_\rho<0\}$ and $J_+$ for the set $\{\rho\in J^c\vert \beta_\rho>0\}$. Write $S_-=\{\rho\in \sigma(1)\mid \beta_\rho<0\}$ and $S_+=\{\rho\in \sigma(1)\mid \beta_\rho>0\}$. Where clear, we will abuse notation and write $S_+, S_-$ for the index set as well as the set of $\beta_\rho$ themselves. Similarly, for a lockable set $I$ we will, where clear, also refer to $I$ as the set of conic modules $A_i, i\in I$.
\end{notation}

Using this notation, we can formulate the following criterion of validity for paths.
\begin{lemma}
    \label{Lem:AlmSimValidity}
    Let $\sigma$ be an almost simplicial Gorenstein cone and $\beta_J:l-k\rightsquigarrow l$ be a path with $\sum_{\rho\in J}\beta_\rho=k\in \Z$. Then $\beta_J$ is valid if and only if \begin{equation}
    \label{eq:AlmSimValidity}
    \sum_{\rho\in J_-}\beta_\rho<l-k<\sum_{\rho\in J_+}\beta_\rho.
    \end{equation}
\end{lemma}
\begin{proof}
Assume first that $k\neq 0$. Note that $\sum_{\rho\in J}\beta_\rho=k$ and so $\sum_{\rho\in J^c}\beta_\rho=-k$. Write now $l=\sum_{\rho\in J}\frac{l}{k}\beta_\rho$. Then the path $\beta_J$ is valid if and only if
    \[
    l\in \sum_{\rho\in J^c}(-1,0)\beta_\rho.
    \]
    This is equivalent to \begin{eqnarray*}
        && -\sum_{\rho\in J_+}\beta_\rho<l<-\sum_{\rho\in J_-}\beta_\rho\\
        &&\Leftrightarrow \sum_{\rho\in J_-}\beta_\rho<l+\sum_{\rho\in J^c}\beta_\rho<\sum_{\rho\in J_+}\beta_\rho\\
        && \Leftrightarrow \sum_{\rho\in J_-}\beta_\rho<l-k<\sum_{\rho\in J_+}\beta_\rho,
    \end{eqnarray*}
    as required.

If $k=0$, we reduce to the case $l\ge 0$ as the case $l<0$ is analogous. There is now some $\beta_\rho>0$ in $J$, as otherwise $J$ consists only of $\beta_\rho=0$. In that case, however, the inequality \eqref{eq:AlmSimValidity} reduces to $l\in Z_\sigma$ so always holds. Similarly, we can choose to write $l= l/\beta_{\rho_+}\cdot \beta_{\rho_+}$ for some $\beta_{\rho_+}>0$ and the validity condition becomes $0\in \sum_{\rho\in S_+\setminus\{\rho_+\}}(-1,0)\beta_\rho+(-l-1,-l)+\sum_{\rho\in S_-}(-1,0)\beta_\rho$, which always holds.

Write $l=\sum_{\rho\in J\cap S_+}l/p \cdot \beta_\rho$ where $p=\sum_{\rho\in J\cap S_+}\beta_\rho$. Validity is equivalent to \[
l\in \sum_{\rho\in J^c}(-1,0)\beta_\rho,
\]
which is equivalent to \[
-\sum_{\rho\in J_+}\beta_\rho <l<-\sum_{\rho\in J_-}\beta_\rho.
\]
As noted before, this is equivalent to \eqref{eq:AlmSimValidity}.
\end{proof}

The following Theorem, which is the main result in this section of the paper, classifies which almost simplicial Gorenstein cones admit NCCRs via endomorphism algebras of incomplete sums of conic modules.

\begin{theorem}
    \label{Thm:AlmSimplConicNCCR}
    Let $\sigma$ be an almost simplicial  Gorenstein cone. Then there exists an incomplete sum $\mathbb{B}$ of conic modules such that $\End_{R}(\mathbb{B})$  is an NCCR of $R=k[\sigma^\vee\cap M]$ if and only if any of the following holds:
    \begin{itemize}
        \item The collection of $\beta_\rho$ associated to $\sigma$ is, up to flipping all signs, $\{2,1,-1,-1,-1\}$.
        \item The collection of $\beta_\rho$ associated to $\sigma$ is $\{1,1,1,-1,-1,-1\}$.
        \item $\sigma$ is a 3-dimensional Gorenstein cone lattice equivalent to $\sigma=\cone(P\times\{1\})$, where $P$ is a trapezoid. 
    \end{itemize}
\end{theorem}

Before we dive into the details, let us here outline the ideas and structure behind the proof. Firstly, note that we completely shift focus to the geometry and combinatorics of the secondary fans. In other words, we use Theorem \ref{Thm:Incr definitions agree} to reduce to deciding the existence of incredulous sets of lattice points. 

Given an almost simplicial Gorenstein cone $\sigma$,  we establish some basic facts about the collection of $\beta_\rho$ associated to it: for instance, $|S_\pm|\ge 2$ and none of the $\beta_\rho$ vanish. We show that the length of a path corresponds to the size of the corresponding set of $J\subsetneq \sigma(1)$, allowing us to construct the complexes associated to the conic modules simply by classifying the valid paths from the collection of $\beta_\rho$ alone. We then provide a collection of lattice points that need to be contained in any potential incredulous set $I$. To prove Theorem \ref{Thm:AlmSimplConicNCCR} we use this set of points that "cannot be substituted'' to reduce the statement to a handful of explicit cases, which we verify in Propositions \ref{Prop:SpecialCasesWithNCCR} and \ref{Prop:SpecialCasesWithoutNCCR}.

\begin{lemma}
    \label{Lem:AlmGorAtleast2eachset}
    Let $\sigma$ be an almost simplicial Gorenstein cone. Then $|S_-|,|S_+|\neq 0,1$.
\end{lemma}
\begin{proof}
    Note $\sum_{\rho\in \sigma(1)}\beta_\rho=0$ for Gorenstein cones, and so $|S_+|=0$ implies $\beta_\rho=0$ for all $\rho$. But then there is no linear relation between the primitive generators $u_\rho, \rho\in \sigma(1)$. This is not possible as $|\sigma(1)|=n+1>\dim \sigma$. If $|S_+|=1$, then without loss of generality $\beta_{\rho_1}>0$ for $\rho_1\in \sigma(1)$ and $\beta_\rho\le 0$ for all other $\rho\in \sigma(1)\setminus\{\rho_1\}$. Thus we obtain a linear relationship $\beta_1 u_1=\sum_{\rho\in\sigma(1)\setminus\{\rho_1\}}-\beta_\rho u_\rho$. This, however, implies that $u_{\rho_1}\in \cone(u_\rho\vert\rho\in \sigma(1)\setminus \{\rho_1\})$. But then $\rho$ is not an extremal ray of $\sigma$ and the cone is in fact simplicial. Therefore, $|S_+|\ge 2$ and analogously $|S_-|\ge 2$.
\end{proof}

\begin{lemma}
    \label{Lem:AlmGor0isUnsub}
    Let $\sigma$ be an almost simplicial Gorenstein cone of dimension $n$ and $I$ a lockable set. Then $0\in I$.
\end{lemma}
\begin{proof}
    If $0\not\in I$, neither are any $l$ such that there is a valid path $0\rightsquigarrow l$ as otherwise substituting $A_0$ simply reintroduces another copy of $A_0$ in a higher degree (as $A_0$ appears in degree $n$ in $K_0^\bullet$ by Lemma \ref{Lem:Gor0isself}). A path $\beta_J:0\rightsquigarrow \beta_J$ is valid (confounding notations by denoting the endpoint of the path with $\beta_J$) if and only if $J_+\cap S_+$ and $J_-\cap S_-$ are both non-empty (using Lemma \ref{Lem:AlmSimValidity}). So no such $\beta_J$ is in the set $I$. But substituting $A_{\beta_J}$ will introduce an $A_0$ into any complex, which cannot be substituted away. Hence, no complex where $A_{\beta_J}$ appears can be valid either. Continuing this line of argument, if there is a concatentation of valid paths $0\rightsquigarrow M_1\rightsquigarrow\dots\rightsquigarrow M_k\rightsquigarrow l$, then $l\not \in I$.
    It is thus sufficient to prove that there is a concatenation of valid paths from $0$ to any $l\in I$.
    
    Without loss of generality let $l>0$. Note that there exists $\beta_i\in S_+$ and $\beta_j\in S_-$ such that $\gcd{(\beta_i,\beta_j)}=1$. 
    So there exist $r,s\in \Z$ such that $r|\beta_i|+s|\beta_j|=1$.
    Not both $r,s$ can be $>0$ so either $r\ge0, s\le0$ or $r\le0, s\ge0$. Then $1=r\beta_i+(-s)\beta_j$ with $-s\ge 0$ and similarly $l=rl\beta_i+(-s)l\beta_j$. Since $l>0$, we cannot have $r\le 0$ and $s\ge 0$, and thus we know that $r\ge0$ and $s\le 0$.
    
    Note that for all $M\in Z_\sigma$, either the path $\beta_i$ or the path $\beta_j$ is valid. Indeed, $\beta_i$ is a valid path from $M$ to $M+\beta_i$ if and only if \[
    \sum_{\beta_\rho\in S_-}\beta_\rho<M<\sum_{\beta_\rho\in S_+}\beta_\rho - \beta_i.
    \]
    Thus, if $\beta_i$ is not a valid path, $\sum_{\beta_\rho\in S_+}\beta_\rho>M\ge\sum_{\beta_\rho\in S_+}\beta_\rho - \beta_i>\sum_{\beta_\rho\in S_-}\beta_\rho$ where the last inequality follows as $\sum_{\beta_\rho\in S_-}\beta_\rho+\sum_{\beta_\rho\in S_+}\beta_\rho=0$ and $|S_+|\ge 2$ and so $\beta_i<|\sum_{\beta_\rho\in S_-}\beta_\rho|$. This chain of inequalities however implies $\beta_j$ is valid, and so we have that either $\beta_i$ or $\beta_j$ has to be a valid path from $M$.

    Begin thus at 0 and move along a valid choice of path, either $\beta_i$ or $\beta_j$. Repeat this process, always landing at a point of the form $k_1\beta_i+k_2\beta_j$ and at each step increasing the sum $k_1+k_2$ by one. If $k_1=r l$, but $k_2<(-s)l$, note that subsequently the path $\beta_j$ is valid, as \[\sum_{\rho\in S_+}\beta_\rho>rl\beta_i+k_2\beta_j>rl\beta_i+(-s)l\beta_j=l>0>\sum_{\rho\in S_-}\beta_\rho-\beta_j.\]
    Hence, we can keep picking the valid path $\beta_j$ until we reach the point $rl\beta_i+(-s)l\beta_j=l$. Similarly if we reach $k_2=(-s)l$ before $k_1$ becomes $rl$, the path $\beta_i$ remains valid. Thus, we have found a concatentation of valid paths from 0 to $l$.

    For $l<0$, we note that $l'=l+\sum_{\beta_\rho\in S_+}\beta_\rho>0$ and so there is a concatenation of valid paths from 0 to $l'$. Using the path $\sum_{\rho\in S_-}\beta_\rho$, we have \[
    0<l'<\sum_{\rho\in S_+}\beta_\rho,
    \]
    and so there is a valid path from $l'$ to $l$, i.e. there exists a concatenation of valid paths starting at 0 and ending at $l$, as required. 

    \noindent In conclusion, all $l\in Z_\sigma$ admit concatenations of valid paths from $0$ to $l$, and so $l\not\in I$ for $l\in Z_\sigma$, which means $I=\emptyset$, a contradiction. Hence, $0\in I$.
\end{proof}

\begin{proposition}
    \label{Prop:AlmGorEverythingissubbed}
    Let $\sigma$ be an almost simplicial Gorenstein cone. Given a lockable set $I$, for every $l\in Z_\sigma$ either $l\in I$ or there is an $i\in I$ such that the sequence of substitutions $S(j,-), j\in I^c$ performed on $K_i^\bullet$ to obtain $K_i^{\dagger,\bullet}$ contains the substitution $S(l,-)$. 
\end{proposition}
\begin{proof}
    By the proof of Lemma \ref{Lem:AlmGor0isUnsub}, there exists a concatenation of valid paths $0\rightsquigarrow M_1\rightsquigarrow M_2\rightsquigarrow \dots \rightsquigarrow M_k \rightsquigarrow-l$ for any $l\in Z_\sigma$. But then by Proposition \ref{Prop:GorSymDoubleswitch}, there is a concatenation of valid paths $l\rightsquigarrow -M_k\rightsquigarrow\dots\rightsquigarrow-M_1\rightsquigarrow 0$. Either no $-M_i$ is in $I$ or there is a maximal $i_m$ such that $-M_{i_m}\in I$. In the former case, to obtain $K_0^{\dagger,\bullet}$, all $-M_i$ need to be substituted and then so does $l$. In the latter case, the same holds true for all $-M_i$ with $i>i_m$ as well as $l$.
\end{proof}

\begin{proposition}
    \label{Prop:AlmGor0isbad}
    Let $\sigma$ be an almost simplicial Gorenstein cone. Suppose there is a $\rho\in \sigma(1)$ with $\beta_\rho=0$. Then there is no incredulous set $I$.
\end{proposition}

\begin{proof}
    Suppose $\rho_1$ has the property that $\beta_{\rho_1}=0$. Then for any $l\in Z_\sigma$, we have $\beta_{\rho_1}:l\rightsquigarrow l$ is valid, as $l\in Z_\sigma\Leftrightarrow \sum_{\rho\in S_-}\beta_\rho <l<\sum_{\rho\in S_+}\beta_\rho$. So none of the complexes can be substituted without immediately reappearing. But Proposition \ref{Prop:AlmGorEverythingissubbed} implies that any $l\in Z_\sigma$ for which $A_l\in K_l^\bullet$ needs to be in $I$ itself and thus $I$ is the full set $\{1,\dots,n+1\}$. But \cite[Proposition 7.9]{FMS19} shows that no almost simplicial cones $\sigma$ exist for which this set is incredulous.
\end{proof} 

 From now on, we focus on those almost simplicial Gorenstein cones for which incredulous sets can, a priori, exist, i.e. where no $\beta_\rho$ is zero. Thus, we fix some notation.
\begin{notation}
We may order the $\beta_\rho, \rho\in \sigma(1)$, and do so. Without loss of generality let \[\beta_1\ge \beta_2\ge\dots\ge\beta_k >0>\beta_{k+1}\ge \beta_{k+2}\geq\dots\geq \beta_{n+1}.\]
Thus, $S_+=\{\beta_1,\dots,\beta_k\}$ and $S_-=\{\beta_{k+1},\dots,\beta_{n+1}\}$ (abusing notation, we refer to $S_+$ and $S_-$ as the index sets $\{1,\dots,k\}$ and $\{k+1,\dots,n+1\}$ respectively).
\end{notation}

\begin{lemma}
    \label{Lem:AlmGorPathLength}
    Let $\sigma$ be an almost simplicial Gorenstein cone such that $\beta_\rho\neq 0$ for all $\rho\in \sigma(1)$. Then $\beta_J$ is a path of length $|J|$.
\end{lemma}
\begin{proof}
    If there was a $J\subsetneq \sigma(1)$ with $\operatorname{rk}(A_J)\neq |J|$, then the primitive generators $u_\rho, \rho\in J$ are linearly dependent. But then the set of linear relations between the primitive generators of $\sigma(1)$ are spanned by this linear relation between the generators of $J$, and so for $\rho\not\in J$, $\beta_\rho=0$.
\end{proof}

\begin{lemma}
\label{Lem:AlmGorFullLength}
Let $\sigma$ be an almost simplicial Gorenstein cone of dimension $n$. For $l\in Z_\sigma= (\sum_{\beta_\rho\in S_-}\beta_\rho,\sum_{\beta_\rho\in S_+}\beta_\rho)$, the complex $K_l^\bullet$ associated to the conic modules $A_l$ has length $n$ if and only if $l\in (-\beta_1,-\beta_{n+1})$.
\end{lemma}
\begin{proof}
    By Lemma \ref{Lem:AlmGorPathLength}, the paths of length $n$ come from subsets of $\sigma(1)$ of precisely size $n$ or $n+1$. The unique subset of size $n+1$ contains all $\beta_\rho$ and is only valid if it starts at 0. Otherwise, each path of length $n$ is of the form $\sum_{\rho\neq \rho_1}\beta_\rho$ for some $\rho_1$, i.e. coming from the set $J=\sigma(1)\setminus\{\rho_1\}$. This path is valid if and only if it starts at $l$ such that \[
    \sum_{\rho\in J_-}\beta_\rho<l<\sum_{\rho\in J_+}\beta_\rho.
    \]
    If $\beta_{\rho_1}>0$, this is equivalent to $0<l<\beta_{\rho_1}$ and if $\beta_{\rho_1}<0$ we instead obtain $\beta_{\rho_1}<l<0$. The path leads to $l-\beta_{\rho_1}$ and so the indices $p$ of the complexes which have a degree $n$ component  via the path $J$ fulfill $\begin{cases}
        0<p<-\beta_{\rho_1} \text{ if }\beta_{\rho_1}<0;\\
       -\beta_{\rho_1}<p<0 \text{ if }\beta_{\rho_1}>0.
    \end{cases}$
    
    Consequently, the set of complexes $K_l^\bullet$ of length $n$ is the set such that \[l\in \left(\min(-\beta_\rho\vert \beta_\rho\in S_+),\max(-\beta_\rho\vert\beta_\rho\in S_-)\right)=(-\beta_1,-\beta_{n+1}).\]
\end{proof}

\begin{proposition}
    \label{Prop:AlmGorUnsubinterval}
    Let $\sigma$ be an almost simplicial Gorenstein cone and $I$ an incredulous set. Then for $l\in (\beta_{n+1}+\dots+\beta_{k+2}+1,\beta_{k-1}+\dots+\beta_{1}-1)$, $l\in I$.
\end{proposition}
\begin{proof}
   Lemma \ref{Lem:AlmGor0isUnsub} shows 0 is "unsubbable'', i.e. $0$ is necessarily in $I$. Let $l>0$ lie in the above interval. Then there is a valid path $l-\sum_{\rho\in S_+}\beta_\rho\rightsquigarrow l$ as $\sum_{\rho\in S_-}\beta_\rho<l-\sum_{\rho\in S_+}\beta_\rho<0$. Similarly, there is a valid path $l\rightsquigarrow l+\sum_{\rho\in S_-}\beta_\rho+\beta_k$ as $l+\sum_{\rho\in S_-}\beta_\rho+\beta_k<\beta_1+\dots+\beta_{k-1}$ (as $|\sum_{\rho\in S_-}\beta_\rho|>|\beta_k|$, given that $|S_+|\ge 2$). This gives a concatenation of valid paths $l-\sum_{\rho\in S_+}\beta_\rho\rightsquigarrow l\rightsquigarrow l+\sum_{\rho\in S_-}\beta_\rho+\beta_k$, the two paths having length $|S_+|$ and $|S_-|+1$ respectively. Substitution by $l$ would therefore lead to an occurence of $l-\sum_{\rho\in S_+}\beta_\rho$ in degree $|S_-|+|S_+|+1-1=n+1>n$ in $K_{l+\sum_{\rho\in S_-}\beta_\rho+\beta_k}^{\bullet}$ after substitution -  and thus no incredulous set can exist as the complex is longer than $n$ and thus the associated conic module can neither be in $I$ nor can it be substituted into any other complex.
\end{proof}

We now have all the tools necessary to prove Theorem \ref{Thm:AlmSimplConicNCCR}. The proof will show that, in the most general cases, NCCRs constructed via conic modules do not exist -  but to do so, a few assumptions will need to be made. We thus first treat the cases not covered by these assumptions via Propositions \ref{Prop:3DAlmostSimp}, \ref{Prop:SpecialCasesWithNCCR} and \ref{Prop:SpecialCasesWithoutNCCR}.

\begin{proposition}
    \label{Prop:3DAlmostSimp}
    Let $\sigma$ be an almost simplicial Gorenstein cone of the form $\cone(P\times\{1\})\subset \R^3$ where $P$ is a quadrilateral. Then there is an incredulous set of conic modules for $\sigma$, and hence an NCCR of $R=k[\sigma^\vee\cap M]$ constructed via conic modules, if and only if $P$ is a trapezoid.
\end{proposition}

\begin{proof}
Consider the collection of $\beta_\rho, \rho\in \sigma(1)$. Note first and foremost that by Theorem \ref{Thm:Incr definitions agree}, the existence of an NCCR is equivalent to the existence of an incredulous set of lattice points in $Z_\mathcal{X}=(\beta_3+\beta_4,\beta_1+\beta_2)$, which are precisely the integers inside the zonotope. By Lemma \ref{Lem:AlmGorAtleast2eachset} and Proposition \ref{Prop:AlmGor0isbad}, we may write $\beta_1\ge \beta_2>0>\beta_3\ge \beta_4$. Without loss of generality, we may assume $|\beta_2|\ge |\beta_3|$. There are 16 possible paths, and we list them in Table \ref{Tab:3Dcase}, using Lemma \ref{Lem:AlmSimValidity} to exhibit when the path is valid. We use $\beta_1+\beta_2=-(\beta_3+\beta_4)$ repeatedly to simplify expressions (e.g. $\beta_1+\beta_2+\beta_3=-\beta_4$).

\begin{figure}[ht!]
   \begin{tabular}{|r|r|r|r|}
     \hline
          Path & Valid Start & Valid End & Length of path \\ \hline
          $\emptyset$ & $\beta_3+\beta_4<P<\beta_1+\beta_2$ & $\beta_3+\beta_4<Q<\beta_1+\beta_2$ & 0 \\ \hline
          $\beta_1$ & $\beta_3+\beta_4<P<\beta_2$ & $-\beta_2<Q<\beta_1+\beta_2$& 1 \\ \hline
          $\beta_2$ & $\beta_3+\beta_4<P<\beta_1$& $-\beta_1<Q<\beta_1+\beta_2$& 1\\ \hline
          $\beta_3$ & $\beta_4<P<\beta_1+\beta_2$& $\beta_3+\beta_4<Q<-\beta_4$ & 1 \\ \hline
          $\beta_4$ & $\beta_3<P<\beta_1+\beta_2$ &  $\beta_3+\beta_4<Q<-\beta_3$& 1\\ \hline
          $\beta_1+\beta_2$ & $\beta_3+\beta_4<P<0$ & $0<Q<\beta_1+\beta_2$ & 2\\ \hline
          $\beta_1+\beta_3$ &  $\beta_4<P<\beta_2$&  $-\beta_2<Q<-\beta_4$& 2\\ \hline
          $\beta_1+\beta_4$ & $\beta_3<P<\beta_2$&  $-\beta_2<Q<-\beta_3$& 2\\ \hline
          $\beta_2+\beta_3$ & $\beta_4<P<\beta_1$ & $-\beta_1<Q<-\beta_4$ & 2\\ \hline
          $\beta_2+\beta_4$ & $\beta_3<P<\beta_1$& $-\beta_1<Q<\beta_3$& 2\\ \hline
          $\beta_3+\beta_4$ & $0<P<\beta_1+\beta_2$&$\beta_3+\beta_4<Q<0$ & 2\\ \hline
          $\beta_1+\beta_2+\beta_3$ & $\beta_4<P<0$ & $0<Q<-\beta_4$ & 3\\ \hline
          $\beta_1+\beta_2+\beta_4$ & $\beta_3<P<0$ & $0<Q<-\beta_3$& 3\\ \hline
          $\beta_1+\beta_3+\beta_4$ & $0<P<\beta_2$& $-\beta_2<Q<0$ & 3 \\ \hline
          $\beta_2+\beta_3+\beta_4$ & $0<P<\beta_1$& $-\beta_1<Q<0$& 3\\ \hline 
          $\beta_1+\dots+\beta_4$ &0 & 0& 3\\ \hline
   \end{tabular}
   \caption{Valid paths for 3 dimensional almost simplicial Gorenstein cones.}
   \label{Tab:3Dcase}
\end{figure}

Note $|\beta_2|=|\beta_3|\Leftrightarrow |\beta_1|=|\beta_4|$. Consider first the situation where this is not the case, i.e. $|\beta_2|>|\beta_3|$ and consequently $|\beta_1|<|\beta_4|$. Suppose we have an incredulous set of lattice points $I$. By Lemma \ref{Lem:AlmGor0isUnsub}, $0\in I$ and by Proposition \ref{Prop:AlmGorUnsubinterval} $-\beta_1\in I$ as $0>-\beta_1\ge \beta_4+1$.
The complex associated to $A_{-\beta_1}$ is \[
A_{\beta_2}\rightarrow A_{\beta_2+\beta_4}\oplus A_{\beta_2+\beta_3}\rightarrow A_{-\beta_1}. \]
Note that the complex associated to $A_0$ contains, in degree 2, the modules $A_{\beta_2+\beta_3}, A_{\beta_2+\beta_4}$. 
But $\beta_2+\beta_3,\beta_2+\beta_4>-\beta_1$ (as $\beta_1+\beta_2=-(\beta_3+\beta_4)>-\beta_j$, $j=3,4$) and thus the conic modules $\beta_2+\beta_3, \beta_2+\beta_4$  are in $I$, as their complexes are of length 3 (since a valid path to them of length 3 exists) and so we cannot substitute them into $K_0^\bullet$ without creating a complex of length superior to $\dim \sigma=3$. For the set $I$ to be incredulous, we thus need to substitute $A_{\beta_2}$ as otherwise $K_{-\beta_1}^{\dagger,\bullet}=K_{-\beta_1}^\bullet$, which is not of sufficient length. Hence, we cannot have $\beta_2\in I$. The complex $K_{\beta_2}^\bullet$ has length 3 as $\beta_2<-\beta_4$. Thus, any $l$ such that $A_{\beta_2}$ appears in degree $\ge 2$ can also not be in $I$, as otherwise the substitution $S(K_l^\bullet,\beta_2)$ yields a complex longer than $\dim \sigma$. There is a valid path of length 2 from $\beta_2$ to $2\beta_2+\beta_4\in Z_\mathcal{X}$ and so by this argumentation, $2\beta_2+\beta_4\not \in I$. If $2\beta_2+\beta_4=0$, this is a contradiction as $0\in I$, and so we may assume $2\beta_2+\beta_4\neq 0$.  Note $\beta_4<2\beta_2+\beta_4<\beta_1$ as $\beta_1-\beta_4>\beta_2+\beta_2$ since $|\beta_4|>|\beta_1|\ge|\beta_2|$. Thus $\beta_4<2\beta_2+\beta_4<\beta_1$. If $\beta_4<2\beta_2+\beta_4<0$, then the path (of length 3) $\beta_1+\beta_3+\beta_4=-\beta_2$ to $\beta_2+\beta_4$ is valid, and $\beta_2+\beta_4\in I$, a contradiction. If $0<2\beta_2+\beta_4<\beta_1$, then $2\beta_2+\beta_4\le \beta_1-1$ and so $2\beta_2+\beta_4\in I$ by Proposition \ref{Prop:AlmGorUnsubinterval}, a contradiction. Hence, if $|\beta_2|\neq |\beta_3|$, there cannot be an incredulous set and thus no NCCR via conic modules. 

We are now left to consider the situation $\beta_1=-\beta_4, \beta_2=-\beta_3$. 
If $\beta_1\neq\beta_2$, we claim the set $I=(\beta_4,\beta_1)\cap \Z$ is incredulous. All the complexes associated to $l\in I$ have length 3 (see Table \ref{Tab:3Dcase}). For any $l\ge \beta_1$ (and by symmetry we may reduce to this case), $A_l$ appears in degree 1 in $K_{l-\beta_1}^\bullet, K_{l-\beta_2}^\bullet$ and in degree 2 in $K_{l-(\beta_1+\beta_2)}^\bullet$ and nowhere in degree 3. The complex associated to $l$ is
\[A_{l-(\beta_1+\beta_2)}\rightarrow A_{l-\beta_1}\oplus A_{l-\beta_2}\rightarrow A_l.\] 
But since $\beta_2\le \beta_1\le l<\beta_1+\beta_2$, we have $0\le l-\beta_i<\beta_j$ for $\{i,j\}=\{1,2\}$. Thus, $l-\beta_1,l-\beta_2\in I$. Similarly, $\beta_2\neq \beta_1$ means that $l-(\beta_1+\beta_2)\ge -\beta_2>\beta_4$ and so $l-(\beta_1+\beta_2)\in I$. Substituting the length 2 complex $K_l^\bullet$ thus does not change the length of the complexes where $A_l$ appears, as these are already of length 3. Hence, $I$ is indeed incredulous. 

If $\beta_1=\beta_2$, we no longer have $l-(\beta_1+\beta_2)\in I$. So consider the set $I'=(\beta_4,\beta_1]\cap \Z=(-\beta_1,\beta_1]\cap \Z$. For $l>\beta_1$, $l-(\beta_1+\beta_2)>-\beta_1=\beta_4$ and the previous argument applies. For $l=\beta_4$, $l-(\beta_3+\beta_4)=\beta_1\in I'$ and so the same argument applies as well. Hence, $I'$ is incredulous provided substitution increases the length of the complex $K_{\beta_1}^\bullet$ to 3. But $A_{-\beta_1}$ appears in degree 2 in $K_{\beta_1}^\bullet$ and has length 2, so $K_{\beta_1}^{\dagger,\bullet}$ is indeed of length 3 and the set $I'$ is incredulous.

Finally, it suffices to show that $\beta_1=-\beta_4, \beta_2=-\beta_3$ arises exactly when the quadrilateral $P$ is a trapezoid. Note that $\beta_\rho$ give the linear relation between the four vertices of $P$, and so we have $\beta_1v_1+\beta_2v_2=\beta_2v_3+\beta_1v_4$. Thus, the line segments $v_1v_2$ and $v_3v_4$ meet such that the intersection point divides both lines in the same proportion. The quadrilateral is convex and so, as no $\beta_i$ here is zero, $v_1v_2$ and $v_3v_4$ are the diagonals, i.e. $v_1$ and $v_2$ are not adjacent and neither are $v_3v_4$. Denote the intersection point of the diagonals by $x$ and consider the triangles $xv_1v_4$ and $xv_2v_3$. The triangles are similar to each other since $x$ divides either diagonal in the same ratio and the internal angle at $x$ is of the same magnitude. Thus, the sides $v_1v_4$ and $v_2v_3$ are in fact parallel since the angles $\angle xv_4v_1$ and $\angle xv_3v_2$ agree. The quadrilateral thus has two parallel sides, i.e. it is a trapezoid.
\end{proof}

\begin{remark}
    Example \ref{Exa:710FMS} concerns a cone over a square, thus the Proposition \ref{Prop:3DAlmostSimp} implies the existence of an NCCR - predicting the incredulous set $(-1,1]\cap \Z=\{0,1\}$.
\end{remark}

\begin{proposition}
    \label{Prop:SpecialCasesWithNCCR}
    Let $\sigma$ be an almost simplicial Gorenstein cone such that either of the following is true:
    \begin{enumerate}
        \item The collection of $\beta_\rho$ associated to $\sigma$ is, up to flipping all signs, $\{2,1,-1,-1,-1\}$.
        \item The collection of $\beta_\rho$ associated to $\sigma$ is $\{1,1,1,-1,-1,-1\}$.
    \end{enumerate}
    Then there exists an incredulous set of conic modules, and thus there is an incomplete direct sum $\mathbb{B}$ of conic modules such that $\End_{R_\sigma}(\mathbb{B})$ is an NCCR of $R_\sigma$.
\end{proposition}
\begin{proof}
We will leave out the computations of valid paths and simply list the complexes before and after substitution.

\underline{$(1)$:} The five complexes corresponding to the lattice points $-2,-1,0,1,2$ are:\begin{align*}
 A_1\longrightarrow A_0^{\oplus 3}\longrightarrow A_{-1}^{\oplus 3}\longrightarrow&A_{-2},\\
A_{-1}\longrightarrow A_2\oplus A_0^{\oplus 3}\longrightarrow A_1^{\oplus 3}\oplus A_{-1}^{\oplus 3}\longrightarrow A_{-2}\oplus A_0^{\oplus 3}\longrightarrow&A_{-1},\\
A_0\longrightarrow A_0^{\oplus 3}\oplus A_1^{\oplus 3}\longrightarrow A_2^{\oplus 3}\oplus A_0^{\oplus 3}\oplus A_{-1}^{\oplus 3}\longrightarrow A_{-2}\oplus A_{-1}\oplus A_1^{\oplus 3}\longrightarrow&A_0,\\
 A_{-1}^{\oplus 3}\longrightarrow A_1^{\oplus 3}\oplus A_0^{\oplus 3}\oplus A_{-2}\longrightarrow A_{-1}\oplus A_0\oplus A_2^{\oplus 3}\longrightarrow&A_1,\\
 A_{-1}\longrightarrow A_0\oplus A_1\longrightarrow&A_2.
\end{align*} 
The set $\{-1,0,1\}$ is lockable, after substituting giving the complexes:
\begin{align*}
   A_{-1}\oplus A_1\longrightarrow A_0^{\oplus 4}\oplus A_1^{\oplus 2}\longrightarrow A_{-1}^{\oplus 3}\oplus A_0^{\oplus 3}\oplus A_1^{\oplus 3}\longrightarrow A_0^{\oplus 3}\oplus A_{-1}^{\oplus 3}\longrightarrow& A_{-1},\\
   A_0\longrightarrow A_{-1}^{\oplus 3}\oplus A_0^{\oplus 3}\oplus A_1^{\oplus 4}\longrightarrow A_{-1}^{\oplus 3}\oplus A_0^{\oplus 9}\oplus A_1^{\oplus 3}\longrightarrow A_{-1}^{\oplus 4}\oplus A_1^{\oplus 3} \longrightarrow&A_0,\\
  A_1\longrightarrow A_{-1}^{\oplus 3}\oplus A_0^{\oplus 3}\longrightarrow A_{-1}^{\oplus 6}\oplus A_0^{\oplus 3}\oplus A_1^{\oplus 3}\longrightarrow A_{-1}\oplus A_0^{\oplus 4}\oplus A_1^{\oplus 4\longrightarrow}&A_1.
\end{align*}
Thus, the set $\{-1,0,1\}$ is in fact incredulous, hence giving an NCCR for the toric algebra by Theorem \ref{Thm:Incr definitions agree}.

\underline{$(2)$:}
Here, the five complexes corresponding to the lattice points inside the zonotope $(-3,3)$ are:
\begin{align*}
    A_1\longrightarrow A_0^{\oplus 3}\longrightarrow A_{-1}^{\oplus 3}\longrightarrow & A_{-2},\\
    A_1\longrightarrow A_0^{\oplus 8}\oplus A_2\longrightarrow A_{-1}^{\oplus 9}\oplus A_1^{\oplus 3}\longrightarrow A_{-2}^{\oplus 3}\oplus A_0^{\oplus 3}\longrightarrow & A_{-1},\\
    A_0\longrightarrow A_0^{\oplus 9}\longrightarrow A_{-1}^{\oplus 9}\oplus A_1^{\oplus 9}\longrightarrow A_{-2}^{\oplus 3}\oplus A_0^{\oplus 9}\oplus A_2^{\oplus 3}\longrightarrow A_{-1}^{\oplus 3}\oplus A_1^{\oplus 3}\longrightarrow & A_0,\\
    A_-1^{\oplus 3}\longrightarrow A_{-2}\oplus A_0^{\oplus 9}\longrightarrow A_{-1}^{\oplus 3}\oplus A_1^{\oplus 9}\longrightarrow A_0^{\oplus 3}\oplus A_2^{\oplus 3}\longrightarrow &A_1,\\
    A_{-1}\longrightarrow A_0^{\oplus 3}\longrightarrow A_1^{\oplus 3}\longrightarrow & A_2.
\end{align*}
The set $\{-1,0,1\}$ is incredulous, after substitution giving the complexes:
\begin{align*}
    A_{-1}\longrightarrow A_0^{\oplus 3}\oplus A_1^{\oplus 3}\longrightarrow A_{-1}^{\oplus 9}\oplus A_0^{\oplus 9}\oplus A_1^{\oplus 3}\longrightarrow A_{-1}^{\oplus 9}\oplus A_0^{\oplus }9\oplus A_1^{\oplus 3}\longrightarrow A_{-1}^{\oplus 9}\oplus A_0^{\oplus 3}\longrightarrow & A_{-1},\\
    A_0\longrightarrow A_{-1}^{\oplus 3}\oplus A_0^{\oplus 9}\oplus A_1^{\oplus 3}\longrightarrow A_{-1}^{\oplus 9}\oplus A_0^{\oplus 18}\oplus A_1^{\oplus 9}\longrightarrow A_{-1}^{\oplus 9}\oplus A_0^{\oplus 9}\oplus A_1^{\oplus 9}\longrightarrow A_{-1}^{\oplus 3}\oplus A_1^{\oplus 3}\longrightarrow & A_0,\\
    A_{1}\longrightarrow A_0^{\oplus 3}\oplus A_{-1}^{\oplus 3}\longrightarrow A_{1}^{\oplus 9}\oplus A_0^{\oplus 9}\oplus A_{-1}^{\oplus 3}\longrightarrow A_{1}^{\oplus 9}\oplus A_0^{\oplus }9\oplus A_{-1}^{\oplus 3}\longrightarrow A_{1}^{\oplus 9}\oplus A_0^{\oplus 3}\longrightarrow & A_{1}.
\end{align*}
Hence, we obtain the required NCCR via Theorem \ref{Thm:Incr definitions agree}.
\end{proof}

\begin{proposition}
    \label{Prop:SpecialCasesWithoutNCCR}
    Let $\sigma$ be an almost simplicial Gorenstein cone such that either of the following is true:
    \begin{enumerate}
        \item The collection of $\beta_\rho$ associated to $\sigma$ is, up to flipping all signs, $\{2,1,1,-2,-2\}$.
        \item The collection of $\beta_\rho$ associated to $\sigma$ is, up to flipping all signs, $\{2,2,2,-3,-3\}$.
        \item The collection of $\beta_\rho$ associated to $\sigma$ is, up to flipping all signs, $\{2,2,-1,-1,-1,-1\}$.
    \end{enumerate}
    Then there does not exist an incredulous set of conic modules, and thus there is no incomplete direct sum $\mathbb{B}$ of conic modules such that $\End_{R_\sigma}(\mathbb{B})$ is an NCCR of $R_\sigma$.
\end{proposition}
\begin{proof}
    \underline{(1):} The zonotope is $(-4,4)$, with the seven associated complexes being:
    \begin{align*}
        A_1\longrightarrow A_{-1}^{\oplus 2}\longrightarrow & A_{-3},\\
        A_1^{\oplus 2}\longrightarrow A_{-1}^{\oplus 4}\oplus A_2\longrightarrow A_{-3}^{\oplus 2}\oplus A_0^{\oplus 2}\longrightarrow & A_{-2},\\
        A_1 \longrightarrow A_{-1}^{\oplus 2}\oplus A_1\oplus A_2^{\oplus 2}\longrightarrow A_{-3}\oplus A_{-1}^{\oplus 2}\oplus A_0^{\oplus 4}\oplus A_3\longrightarrow A_{-3}\oplus A_{-2}^{\oplus 2}\oplus A_1^{\oplus 2}\longrightarrow & A_{-1},\\
        A_0\longrightarrow A_{-1}^{\oplus 4}\oplus A_0^{\oplus 2}\longrightarrow A_{-3}^{\oplus 2}\oplus A_{-2}\oplus A_0^{\oplus 2}\oplus A_1^{\oplus 4}\longrightarrow A_{-2}\oplus A_{-1}^{\oplus 2}\oplus A_{2}^{\oplus 2}\longrightarrow & A_0,\\
        A_{-1}^{\oplus 2}\longrightarrow A_{-3}\oplus A_0^{\oplus 4}\oplus A_1^{\oplus 2}\longrightarrow A_{-2}^{\oplus 2}\oplus A_{-1}\oplus A_1^{\oplus 2}\oplus A_2^{\oplus 4}\longrightarrow A_{-1}\oplus A_0^{\oplus 2}\oplus A_3^{\oplus 2}\longrightarrow & A_1,\\
        A_{-2}\longrightarrow A_{-1}^{\oplus 2}\oplus A_0\longrightarrow A_0\oplus A_1^{\oplus 2}\longrightarrow & A_2,\\
        A_{-1}\longrightarrow A_0^{\oplus 2}\oplus A_1\longrightarrow A_1\oplus A_2^{\oplus 2}\longrightarrow &  A_3.
    \end{align*}

    Observe that no subset is incredulous, hence no NCCR is obtained as endomorphism algebra of an incomplete direct sum of conic modules. 

    \underline{(2):} In this case, the zonotope is $(-6,6)$, with the 11 complexes associated to the lattice points being:

    \begin{align*}
        A_{1}\longrightarrow A_{-2}^{\oplus 2}\longrightarrow & A_{-5},\\
        A_2\longrightarrow A_{-1}^{\oplus 2}\longrightarrow & A_{-4},\\
        A_1^{\oplus 3}\longrightarrow A_{-2}^{\oplus 6}\oplus A_3\longrightarrow A_0^{\oplus 2}\oplus A_{-5}^{\oplus 3}\longrightarrow & A_{-3},\\
        A_2^{\oplus 3}\longrightarrow A_{-1}^{\oplus 6}\oplus A_4\longrightarrow A_1^{\oplus 2}\oplus A_{-4}^{\oplus 3}\longrightarrow & A_{-2},\\
        A_1^{\oplus 3}\longrightarrow A_{-2}^{\oplus 6}\oplus A_3^{\oplus 3}\longrightarrow A_{-5}^{\oplus 3}\oplus A_0^{\oplus 6}\oplus A_5\longrightarrow A_2^{\oplus 2}\oplus A_{-3}^{\oplus 3}\longrightarrow & A_{-1},\\
        A_0\longrightarrow A_{-1}^{\oplus 6}\longrightarrow A_{-4}^{\oplus 3}\oplus A_1^{\oplus 6}\longrightarrow A_3^{\oplus 2}\oplus A_{-2}^{\oplus 3}\longrightarrow & A_0,\\
        A_{-2}^{\oplus 2}\longrightarrow A_{-5}\oplus A_0^{\oplus 6}\longrightarrow A_{-3}^{\oplus 3}\oplus A_2^{\oplus 6}\longrightarrow A_4^{\oplus 2}\oplus A_{-1}^{\oplus3}\longrightarrow & A_1,\\
        A_{-1}^{\oplus 2}\longrightarrow A_{-4}\oplus A_1^{\oplus 6}\longrightarrow A_{-2}^{\oplus 3}\oplus A_3^{\oplus 6}\longrightarrow A_5^{\oplus 2}\oplus A_0^{\oplus 3}\longrightarrow & A_2,\\
        A_{-3}\longrightarrow A_{-1}^{\oplus 3}\longrightarrow A_1^{\oplus 3}\longrightarrow & A_3,\\
        A_{-2}\longrightarrow A_0^{\oplus 3}\longrightarrow A_2^{\oplus 3}\longrightarrow & A_4,\\
        A_{-1}\longrightarrow A_1^{\oplus 3}\longrightarrow A_3^{\oplus 3}\longrightarrow & A_5.
    \end{align*}

    No incredulous subset exists and hence no NCCR can be obtained via conic modules.
    
    \underline{(3):} The zonotope here is $(-4,4)$ and the associated complexes are:
    \begin{align*}
        A_1\longrightarrow A_0^{\oplus 4}\longrightarrow A_{-1}^{\oplus 6}\longrightarrow A_{-2}^{\oplus 4}\longrightarrow &A_{-3},\\
        A_2\longrightarrow A_1^{\oplus 4}\longrightarrow A_0^{\oplus 6}\longrightarrow A_{-1}^{\oplus 4}\longrightarrow& A_{-2},\\
        A_1^{\oplus 2}\longrightarrow A_0^{\oplus 8}\oplus A_3\longrightarrow A_{-1}^{\oplus 12}\oplus A_2^{\oplus 4}\longrightarrow A_{-2}^{\oplus 8}\oplus A_1^{\oplus 6}\longrightarrow A_{-3}^{\oplus 2}\oplus A_0^{\oplus 4}\longrightarrow & A_{-1},\\
        A_0\longrightarrow A_1^{\oplus 8}\longrightarrow A_0^{\oplus 12}\oplus A_3^{\oplus 4}\longrightarrow A_{-1}^{\oplus 8}\oplus A_2^{\oplus 6} \longrightarrow A_{-2}^{\oplus 2}\oplus A_1^{\oplus 4}\longrightarrow & A_0,\\
        A_{-1}^{\oplus 6}\longrightarrow A_{-2}^{\oplus 4}\oplus A_1^{\oplus 12}\longrightarrow A_{-3}\oplus A_0^{\oplus 8}\oplus A_3^{\oplus 6}\longrightarrow A_{-1}^{\oplus 2}\oplus A_2^{\oplus 4}\longrightarrow & A_1,\\
        A_{-1}^{\oplus 4}\longrightarrow A_{-2} \oplus A_1^{\oplus 8}\longrightarrow A_0^{\oplus 2}\oplus A_3^{\oplus 4}\longrightarrow & A_2,\\
        A_{-1}\longrightarrow A_1^{\oplus 2}\longrightarrow&  A_3.
    \end{align*}

    Again, no incredulous subset exists, hence we have no NCCR via conic modules.
\end{proof}

Now that the exceptional cases are covered, we will be able to make the assumptions necessary to proceed with the proof for the general case of almost simplicial Gorenstein toric algebras in Theorem \ref{Thm:AlmSimplConicNCCR}.

\begin{proof}[Proof of Theorem \ref{Thm:AlmSimplConicNCCR}]
Proposition \ref{Prop:SpecialCasesWithNCCR} shows that the listed cones admit NCCRs constructed via incomplete sums of conic modules. We shall now show that if $\sigma$ fits in neither of those cases, no NCCR exists.

We may without loss of generality assume that $|\beta_k|\le |\beta_{k+1}|$. Note that $-\beta_{n+1}\le \beta_1+\dots+\beta_{k-1}$. Otherwise, we would have
    \[-\beta_{n+1}>\beta_1+\dots+\beta_{k-1}
    \Leftrightarrow -\beta_{n+1}+\beta_k>\sum_{\beta_\rho\in S_+}\beta_\rho. \]
    But $-\beta_{n+1}+\beta_k\le -\beta_{n+1}-\beta_{k-1}\le -\sum_{\beta_\rho\in S_-}\beta_\rho=\sum_{\rho\in S_+}\beta_\rho$, a contradiction.

Two cases arise. If $-\beta_{n+1}=\beta_1+\dots+\beta_{k-1}$ then we have equalities in the non-strict inequalities above, which happens if and only if $\beta_k=-\beta_{k+1}$ and $|S_-|=2$. Note that $\beta_k=-\beta_{k+1}$ in particular implies $|\beta_k|\ge|\beta_{k+1}|$. The same argument as above then leaves us with two possibilities: either $-\beta_1>\beta_{n+1}+\dots+\beta_{k+2}+1$, which is a situation that will run analogous to the arguments in the remainder of the proof, or $|S_+|=|S_-|=2$ and $\beta_1=-\beta_4,$ $\beta_2=-\beta_3$. This however is the case of a cone over a trapezoid, and Proposition \ref{Prop:3DAlmostSimp} provides the existence of an incredulous set and thus an NCCR.

For the remainder of the proof, we have $-\beta_{n+1}\le \beta_1+\dots+\beta_{k-1}-1=:l$. Then $l\in I$ and $K_l^\bullet$ is a complex of length $\le n-1$ by Lemma \ref{Lem:AlmGorFullLength} and Proposition \ref{Prop:AlmGorUnsubinterval}. Valid paths ending at $l$ necessarily start at $p\in [\beta_1+\dots+\beta_{k-1}-1-\sum_{\rho\in S_+}\beta_\rho,\sum_{\rho\in S_+}\beta_\rho)$, noting that the left side of the interval is simply $-\beta_k-1$. 

By Proposition \ref{Prop:AlmGorUnsubinterval}, $-\beta_k-1\in I$ so long as $-\beta_k-1\ge \beta_{k+2}+\dots+\beta_{n+1}+1$. We will first deal with the case where this inequality holds, leaving the other cases for later.  

Then the only conic modules $A_p$ appearing in $K_l^\bullet$ that potentially have $p\not\in I$ have $p\in [\beta_1+\dots+\beta_{k-1},\sum_{\beta_\rho\in S_+}\beta_\rho)$. For a path $\beta_J:k\rightsquigarrow l$ to be valid is equivalent to $k<\sum_{\rho\in J_+}\beta_\rho$, i.e. $J_+=S_+$ which means $J\subseteq S_-$. But since $|\beta_k|\le|\beta_{k+1}|$ and $p<\sum_{\rho\in S_+}\beta_\rho$, the only possibility for such a valid path to exist is when $|\beta_k|=|\beta_{k+1}|$ so $p=\beta_1+\dots+\beta_k-1$ and the path is $\beta_J=\beta_{k+1}$. This path has length 1. We can now study the complex $K_p^\bullet$ itself. Any valid path $\beta_J:q\rightsquigarrow p$ has $q+\beta_J=\beta_1+\dots+\beta_k-1$ and $\sum_{\rho\in J_-}\beta_\rho<q<\sum_{\rho\in J_+}\beta_\rho$.

Thus\[
\sum_{\rho\in J_-}\beta_\rho +\beta_J<\beta_1+\dots+\beta_k-1<\sum_{\rho\in J_+}\beta_\rho+\beta_J.
\]
Write $\beta_J=\beta_{J\cap S_-}+\beta_{J\cap S_+}$, and so we obtain\[
\sum_{\rho\in J_-}\beta_\rho+\beta_{J\cap S_-}+\beta_{J\cap S_+}<\beta_1+\dots+\beta_{k}-1<\beta_{J\cap S_+}+\sum_{\rho\in J_+}\beta_\rho+\beta_{J\cap S_-}.
\]
 But $\beta_{J_+}+\beta_{J\cap S_+}=\beta_1+\dots+\beta_k$ and therefore $\beta_{J\cap S_-}=0$, implying $J\subseteq S_+$. Thus, $q\in[-1,p]$, with a valid path starting at $q>l$ necessarily being a path $p\rightsquigarrow p$ not using any elements of $S_-$. This is not possible outside of degree 0, as the only such path is the empty path. The only conic modules appearing in $K_p^\bullet$ are thus of the form $q\in [-1,l]$. 
 
 As $[0,l]\cap \Z\subseteq I$, the only case where the complex potentially $K_p^\bullet$ contains modules not in $I$ is thus when $-1\not\le \beta_{k+2}+\dots+\beta_{n+1}+1$. This can only happen when $|S_-|=2$ and $\beta_{n+1}=\beta_n=-1$, i.e. $n=3$ and the full collection of $\beta_\rho$ is $\beta_1=\beta_2=-\beta_3=-\beta_4=1$, a situation covered by Proposition \ref{Prop:3DAlmostSimp}.

 But if all conic modules in $K_p^\bullet$ are of the form $A_q, q\in I$, then as the length of $K_p^\bullet$ is $\le n-1$ and $A_p$ appears exactly in degree 1 in $K_l^\bullet$, we have that $K_l^{\dagger,\bullet}$ has length $\le n-1$ and can thus not be of length $n=\dim\sigma$. Therefore, the set $I$ is not incredulous, i.e. no NCCR via direct sums of conic modules exists.

It now remains to consider the cases excluded earlier, i.e. $-\beta_k-1<\beta_{k+2}+\dots+\beta_{n+1}+1$. If $|S_-|\ge 3$, note that $-\beta_k-1\ge\beta_{k+1}-1\ge \beta_{k+2}-1\ge \beta_{k+2}+\beta_{n+1}$, with equality only if $\beta_{n+1}=\dots=\beta_{k+1}=-\beta_k=-1$. Further, this is $\ge \beta_{k+2}+\dots+\beta_{n+1}+1$ if $|S_-|\ge 4$. Thus, the remaining cases are $|S_-|=3$ with $k=n-2$ and $-\beta_{n-2}=\beta_{n-1}=\beta_{n}=\beta_{n+1}=-1$ or $|S_-|=2$. In the first case, since $\beta_1+\dots+\beta_{n+1}=0$ and $|S_+|\ge 2$ we obtain either $n=4$ and $\beta_1=2$ or $n=5$ and $\beta_i=1=-\beta_j$ for $i=1,2,3$ and $j=4,5,6$, both cases part of Proposition \ref{Prop:SpecialCasesWithNCCR}.

If $|S_-|=2$, then we have $-\beta_k-1\le \beta_{n+1}$. If
$|\beta_k|=|\beta_{k+1}|$ this would give $\beta_{n+1}=\beta_{k+1}$ or $\beta_{n+1}=\beta_{k+1}-1$. The former gives, as $\beta_1\ge \beta_k$, $n=3$ and $\beta_1=-\beta_4$, a case covered by Proposition \ref{Prop:3DAlmostSimp}.

For the latter, we have $\beta_k=-\beta_{k+1}, \beta_{n+1}=\beta_{k+1}-1$ which has possibilities $n=3$ and $\beta_1=-\beta_4$ or $n=4$ and $\beta_1=\beta_2=\beta_3=1,$ $\beta_4=-1$ and $\beta_5=-2$, covered by Propositions Proposition \ref{Prop:3DAlmostSimp} and \ref{Prop:SpecialCasesWithNCCR}.

Finally, if $|\beta_k|<|\beta_{k+1}|$, then $-\beta_k-1\ge \beta_{k+1}\ge \beta_{n+1}$. For $-\beta_k-1<\beta_{n+1}+1$, the equality needs to hold, i.e. $\beta_{k+1}=\beta_{n+1}=-\beta_k-1$. Since $\beta_1+\dots+\beta_{n+1}=0$, this means $\beta_1+\dots+\beta_{k-1}=\beta_k+2$. If $n=3$, i.e. $|S_+|=2$, we get $\beta_1=\beta_2+2,$ $\beta_3=\beta_4=-\beta_2-1$. This case is covered by Proposition \ref{Prop:3DAlmostSimp}. 

If $|S_+|=3$, then $2\beta_3\le \beta_1+\beta_2=\beta_3+2$ which gives
$\beta_3=1=\beta_2$ and $\beta_1=2$ with $\beta_4=\beta_5=-2$ or $\beta_1=\beta_2=\beta_3=2$ and $\beta_4=\beta_5=-3$, both cases covered by Proposition \ref{Prop:SpecialCasesWithoutNCCR}.
Finally, for $|S_+|\ge 4$, we have $\beta_k+2=\beta_1+\beta_2+\beta_3\ge 3\beta_k$ and so $\beta_1=\dots=\beta_4=1$ and $\beta_5=\beta_6=-2$, which is covered by Proposition \ref{Prop:SpecialCasesWithoutNCCR}. 

With this, we have shown that an almost simplicial Gorenstein cone admits an incredulous set, and thus an NCCR as endomorphism algebra of an incomplete sum of conic modules, exactly in the cases stipulated by Theorem \ref{Thm:AlmSimplConicNCCR}.
\end{proof}

\bibliography{Bib}

\end{document}